\documentclass[]{amsart} 
\usepackage{amsfonts}
\usepackage{bbding}
 \usepackage[dvips]{graphicx}
\usepackage{amsfonts,amssymb,amsmath,amsthm,amscd}
\usepackage{mathrsfs}
\usepackage{xcolor}
\usepackage{empheq}
\usepackage{adforn}
\usepackage{cancel}
\usepackage{xr}
\usepackage{mdframed}
\usepackage{tikz}
\usepackage{amsfonts}
\usepackage{mathdots}
\usepackage{enumerate}
\usepackage{lmodern}
\usepackage{mdframed}
\usepackage{stmaryrd}
\usepackage{blindtext}
\usepackage[all, knot]{xy}
\usepackage{leftidx}
\usepackage{accents}
\usepackage{graphicx}

\usepackage{subfig}

\definecolor{slightblue}{rgb}{.8, .8, 1}
\definecolor{tif}{RGB}{10, 186, 181}
\definecolor{hair}{RGB}{100,225,190}
\definecolor{ruby}{RGB}{220,50,120}
\definecolor{grass}{RGB}{150,220,110}

\usepackage[colorlinks=true, pdfstartview=FitP, linkcolor=tif!85!black, citecolor=ruby, urlcolor=blue!80!black]{hyperref}

\usepackage{tgpagella}
\usepackage[T1]{fontenc}

\makeatletter
\newtheorem*{rep@theorem}{\rep@title}
\newcommand{\newreptheorem}[2]{%
\newenvironment{rep#1}[1]{%
 \def\rep@title{#2 \ref{##1}}%
 \begin{rep@theorem}}%
 {\end{rep@theorem}}}
\makeatother

\makeatletter
\newtheorem*{rep@cor}{\rep@title}
\newcommand{\newrepcor}[2]{%
\newenvironment{rep#1}[1]{%
 \def\rep@title{#2 \ref{##1}}%
 \begin{rep@cor}}%
 {\end{rep@cor}}}
\makeatother

\makeatletter
\newtheorem*{rep@prop}{\rep@title}
\newcommand{\newrepprop}[2]{%
\newenvironment{rep#1}[1]{%
 \def\rep@title{#2 \ref{##1}}%
 \begin{rep@prop}}%
 {\end{rep@prop}}}
\makeatother

\newtheorem{corollary}{Corollary}[section]
\newtheorem{corx}{Corollary}

\newtheorem{theorem}[corollary]{Theorem}
\newtheorem{thmx}[corx]{Theorem}

\newtheorem{proposition}[corollary]{Proposition}
\newtheorem{propx}[corx]{Proposition}

\newrepcor{cor}{Corollary}
\newreptheorem{theorem}{Theorem}
\newrepprop{prop}{Proposition}

\newtheorem*{theorem*}{Theorem}
\newtheorem{lemma}[corollary]{Lemma}

\newtheorem{manualtheoreminner}{Theorem}
\newenvironment{manualtheorem}[1]{%
  \renewcommand\themanualtheoreminner{#1}%
  \manualtheoreminner
}{\endmanualtheoreminner}

\newtheorem{manualpropinner}{Proposition}
\newenvironment{manualprop}[1]{%
  \renewcommand\themanualpropinner{#1}%
  \manualpropinner
}{\endmanualpropinner}

\theoremstyle{definition} 
\newtheorem{definition}[corollary]{Definition}

\newtheorem{example}[corollary]{Example}

\theoremstyle{remark} \newtheorem{remark}[corollary]{Remark} \numberwithin{equation}{section}
\newtheorem*{remark*}{Remark}
\numberwithin{figure}{section}

\newcommand{\R}{\mathbb{R}}
\newcommand{\Rpm}{\mathbb{R}\cup\{\pm\infty\}}
\newcommand{\Rp}{\mathbb{R}\cup\{+\infty\}}
\newcommand{\epi}[1]{\mathrm{epi}({#1})}
\newcommand{\gra}[1]{\mathrm{graph}({#1})}
\newcommand{\sepi}[1]{\mathrm{epi}^\circ({#1})}
\newcommand{\dom}[1]{\mathsf{dom}({#1})}
\newcommand{\env}[1]{\overline{#1}}
\newcommand{\interior}{\mathsf{int}\,}
\newcommand{\ri}{\mathsf{ri}\,}
\newcommand{\chull}[1]{\mathsf{Conv}({#1})}
\renewcommand{\phi}{\varphi}
\newcommand{\eps}{\varepsilon}
\newcommand{\C}{\mathsf{C}}
\newcommand{\CC}[1]{\C^0(\overline{#1})\cap\C^\infty({#1})}
\newcommand{\Leb}{\mathcal{L}}
\newcommand{\MA}[1]{\mathcal{M}{#1}}
\newcommand{\dist}{\mathsf{dist}}
\newcommand{\Ha}{\mathcal{H}}

\renewcommand{\S}{\mathsf{S}}

\newcommand{\T}{\mathsf{T}}
\newcommand{\dif}{\mathsf{d}}
\newcommand{\pair}[2]{\langle{#1}, {#2}\rangle}

\newcommand{\bii}[2]{({#1},{#2})}
\newcommand{\transp}[1]{\leftidx{^\mathsf{t}}{#1}}

\newcommand{\vol}{\mathit{vol}}
\newcommand{\id}{\mathrm{id}}

\newcommand{\pa}{\partial}

\newcommand{\dx}{{\dif x}}

\newcommand{\SL}{\mathrm{SL}}

\newcommand{\eg}{\textit{e.g. }}
\newcommand{\cf}{\textit{c.f. }}
\newcommand{\ie}{\textit{i.e. }}

\DeclareMathOperator{\SO}{SO}

\usepackage{enumitem}
\usepackage{enumerate}

\newcommand{\D}{\mathsf{D}}

\newcommand{\V}{\mathbb{V}}

\newcommand{\LC}{\mathsf{LC}}
\newcommand{\GLC}{\widetilde{\mathsf{LC}}}
\newcommand{\A}{\mathbb{A}}

\newlist{steps}{enumerate}{1}
\setlist[steps, 1]{itemsep=8pt,leftmargin=0cm,itemindent=.5cm,labelwidth=\itemindent,labelsep=0cm,align=left,label = \textbf{\emph{Step \arabic*}:\,}}


\begin{document}

\title[Regular domains and surfaces of constant Gaussian curvature]{Regular domains and surfaces of constant Gaussian curvature in three-dimensional affine space} 
\author{Xin Nie}
\address{Xin Nie: School of Mathematics, KIAS, Hoegi-ro 86, 02455 Seoul, South Korea} \email{nie.hsin@gmail.com}
\author[Andrea Seppi]{Andrea Seppi}
\address{Andrea Seppi: CNRS and Universit\'e Grenoble Alpes, 100 Rue des Math\'ematiques, 38610 Gi\`eres, France.} \email{andrea.seppi@univ-grenoble-alpes.fr}
\maketitle

\begin{abstract}
Generalizing the notion of domains of dependence in the Minkowski space, we define and study regular domains in the affine space with respect to a proper convex cone. In dimension three, we show that every proper regular domain is uniquely foliated by a particular kind of surfaces with constant affine Gaussian curvature. The result is based on the analysis of a  Monge-Amp\`ere equation with extended-real-valued lower semicontinuous boundary condition.
\end{abstract}

\section{Introduction}
The results of this paper place in the context of \emph{Affine Differential Geometry} \cite{MR1311248, MR3382197}, and are more precisely concerned with surfaces of \emph{constant affine Gaussian curvature}, which can be considered at the same time as a generalization of \emph{affine spheres} and of surfaces of \emph{constant Gaussian curvature}, and whose study has been started in \cite{lisimoncrelle,MR1772198,MR2793474}. 

 A convex domain is said to be \emph{proper} if it does not contain any entire straight line.
By solving the Dirichlet problem of Monge-Amp\`ere equation
\begin{equation}\label{eqn_intro1}
\begin{cases}
\det\D^2w=(-w)^{-n-2}\ \text{ in $\Omega$}, \\
w|_{\pa\Omega}=0,
\end{cases}
\end{equation}
on any bounded convex domain $\Omega\subset\mathbb{R}^{n}$, Cheng and Yau \cite{chengyau1} showed that in every proper convex cone $C\subset\mathbb{R}^{n+1}$ there exists a unique complete hyperbolic affine sphere $\Sigma_C$ asymptotic to the boundary $\pa C$ with affine shape operator the identity. See also \cite{MR2743442}.

On the other hand, certain convex domains in the Minkowski space $\mathbb{R}^{n,1}$ known as \emph{domains of dependence} or \emph{regular domains} \cite{MR2170277,MR2110829,bbz} are crucial in the study of globally hyperbolic flat  spacetimes \cite{mess}.
Such a domain is by definition the intersection of the futures of null hyperplanes and is determined by a lower semicontinuous function $\phi:\pa\mathbb{D}\rightarrow\mathbb{R}\cup\{+\infty\}$, where $\mathbb{D}\subset\mathbb{R}^n$ is the unit ball. 
Bonsante, Smillie and the second author showed in \cite{bs, bon_smillie_seppi} that every $3$-dimensional proper regular domain $D\subset\mathbb{R}^{2,1}$ contains a unique complete surface with constant Gaussian curvature $1$ which generates $D$. Analytically, this amounts to the unique existence of a lower semicontinuous convex function $u:\overline{\mathbb{D}}\rightarrow\Rp$ satisfying
\begin{equation}\label{eqn_intro2}
\begin{cases}
\det\D^2u=(1-|z|)^{-2}\,\mbox{ in } U:=\interior\dom{u},\\ u|_{\pa\mathbb{D}}=\phi,\quad |\nabla u(x)|\rightarrow+\infty \text{ as $x\in U$ tends to $\pa U$}.
\end{cases}
\end{equation}
Here ``$\interior$'' stands for the interior and ``$\mathsf{dom}$'' for the subset in the domain of an extended-real-valued function where the values are real. It is also showed in \cite{bon_smillie_seppi} that $\dom{u}$ is exactly the convex hull of $\dom\phi$ in $\R^2$.

The hyperboloid in the future light cone $C_0\subset\mathbb{R}^{2,1}$ provides the simplest example of both results above: It corresponds to the function $u_0(z)=-(1-|z|^2)^\frac{1}{2}$ on $\mathbb{D}\subset\mathbb{R}^2$, which satisfies $\det\D^2u_0=(-u_0)^{-4}=(1-|z|)^{-2}$ with $u_0|_{\pa\mathbb{D}}=0$, hence solves both (\ref{eqn_intro1}) and  (\ref{eqn_intro2}). Geometrically, the two results can be viewed as providing different ways of deforming $C_0$, with a canonical convex surface retained in the deformed convex domain.

\subsection*{Statement of main results} 
The goal of this paper is to unify the two results via surfaces with Constant Affine Gaussian (or Gauss-Kronecker) Curvature (CAGC) $k>0$. The underlying analytic problem, which generalizes both (\ref{eqn_intro1}) and (\ref{eqn_intro2}), is
\begin{equation}\label{eqn_intro3}
\begin{cases}
\det \D^2u=c_k(-w_\Omega)^{-n-2}\mbox{ in } U:=\interior\dom{u}\subset\Omega,\\
u|_{\pa\Omega}=\phi,\quad |\nabla u(x)|\rightarrow+\infty \text{  as $x\in U$ tends to $\pa U$},
\end{cases}
\end{equation}
where $c_k>0$ is a constant determined by $k$ and $n$, and $w_\Omega$ is the solution to (\ref{eqn_intro1}).
The first equation in (\ref{eqn_intro3}) has been called a \emph{two-step Monge-Amp\`ere equation} \cite{lisimoncrelle}, since the Monge-Amp\`ere equation \eqref{eqn_intro1} is involved in $w_\Omega$. 
When $n=2$, we show:
\begin{thmx}\label{thm_intro1}
Let $\Omega\subset\R^2$ be a bounded convex domain satisfying the exterior circle condition and $\phi:\pa\Omega\rightarrow \mathbb{R}\cup\{+\infty\}$ be a lower semicontinuous function such that $\dom{\phi}$ has at least three points. Then there exists a unique lower semicontinuous convex function $u:\overline{\Omega}\to\Rp$  which is smooth in the interior of $\dom{u}$ and satisfies (\ref{eqn_intro3}). Moreover, $\dom{u}$ coincides with the convex hull of $\dom{\phi}$ in $\R^2$.
\end{thmx}
Here, the \emph{exterior circle condition} means for every $x_0\in\pa\Omega$ there is a round disk $B\subset\mathbb{R}^2$ containing $\Omega$ such that $x_0\in\pa B$. Under this condition, the right-hand side of the first equation in (\ref{eqn_intro3}) goes to $+\infty$ fast enough near $\pa\Omega$, which  in turn ensures the gradient blowup property of $u$, namely the last condition in (\ref{eqn_intro3}).

CAGC hypersurfaces and the underlying Monge-Amp\`ere problem (\ref{eqn_intro3}) were first studied by Li, Simon and Chen in \cite{lisimoncrelle}, where they proved unique solvability in any dimension when $\pa\Omega$ and $\phi$ are both \emph{smooth}. Thus, one of the main novelties of this paper is the consideration of boundary value $\phi$ with much weaker regularity assumption and possibly with infinite values, although in this situation we have to restrict to dimension $n=2$ for \emph{regularity} of the solutions (\cf \cite{bonsante-fillastre} and Remark \ref{remark_smooth} below). 


We shall give a more precise geometric description of the CAGC surface resulting from the function $u$ given by Theorem \ref{thm_intro1}. It is known that a non-degenerate hypersurface $\Sigma\subset\mathbb{R}^{n+1}$ has CAGC if and only if its affine normal mapping $N:\Sigma\rightarrow\mathbb{R}^{n+1}$ has image in an affine sphere. Given $k>0$ and a proper convex cone $C\subset\mathbb{R}^{n+1}$, we call $\Sigma$ an \emph{affine $(C,k)$-hypersurface} if it is locally strongly convex, has CAGC $k$ and $N(\Sigma)$ lies in a scaling of the Cheng-Yau affine sphere $\Sigma_C\subset C$ mentioned earlier. We deduce from Theorem \ref{thm_intro1}:
\begin{thmx}\label{thm_intromain}
Let $C\subset\mathbb{R}^3$ be a proper convex cone such that the projectivized dual cone $\mathbb{P}(C^*)\subset \mathbb{RP}^{*2}$ satisfies the exterior circle condition. Let $D\subset\mathbb{R}^3$ be a proper $C$-regular domain. Then for every $k>0$ there exists a unique complete affine $(C,k)$-surface $\Sigma_k\subset D$ which generates $D$. Moreover, $\Sigma_k$ is asymptotic to the boundary of $D$.
\end{thmx}
Here, a \emph{$C$-regular domain} is defined in the same way as regular domains in $\R^{n,1}$ mentioned earlier, except that the role of the future light cone $C_0\subset\R^{n,1}$ is replaced by $C$ (see Section \ref{subsec_cregular} for details). The exterior circle condition on $\mathbb{P}(C^*)$ is equivalent to the \emph{interior} circle condition on $\mathbb{P}(C)$, but $\mathbb{P}(C^*)$ plays a more important role in our analysis because it is essentially the convex domain $\Omega$ in \eqref{eqn_intro3}, while the surface $\Sigma_k$ claimed in Theorem \ref{thm_intromain} is given by the graph of Legendre transform of the function $u$ from Theorem \ref{thm_intro1}. In this regard, the gradient blowup property of $u$ corresponds to the completeness of $\Sigma_k$, whereas the last statement of Theorem \ref{thm_intromain} will be proved using the last statement of Theorem \ref{thm_intro1}.


Theorem \ref{thm_intromain} and its proof imply a  classification of complete affine $(C,k)$-surfaces:
\begin{corx}\label{coro_intro}
Given a constant $k>0$ and a proper convex cone $C\subset\R^3$ such that $\mathbb{P}(C^*)$ satisfies the exterior circle condition, there are natural one-to-one correspondences among the following three types of objects:
\begin{enumerate}[label=(\alph*)]
	\item\label{item_corointro1} proper $C$-regular domains $D\subset\R^3$,
	\item\label{item_corointro2} complete affine $(C,k)$-surfaces $\Sigma\subset\R^3$, and
	\item\label{item_corointro3} lower semicontinuous functions $\phi:\pa\mathbb{P}(C^*)\rightarrow\R\cup\{+\infty\}$ such that $\dom{\phi}$ has at least three points,
\end{enumerate}
where the correspondence \ref{item_corointro1}$\leftrightarrow$\ref{item_corointro2} is given by Theorem \ref{thm_intromain}. Moreover, given $\Sigma$ from \ref{item_corointro2}, the image of the projectivized affine conormal mapping $\mathbb{P}\circ N^*:\Sigma\rightarrow\mathbb{P}(C^*)$ is the interior of the convex hull of $\dom{\phi}$ for the corresponding $\phi$ from \ref{item_corointro3}.
\end{corx}
Here, the affine conormal mapping $N^*:\Sigma\to\R^{*3}$ has image in a scaling of the affine sphere $\Sigma_{C^*}\subset C^*$ dual to $\Sigma_C$, while the projectivization $\mathbb{P}:\R^{*3}\setminus\{0\}\to\mathbb{RP}^{*2}$ gives a bijection from $\Sigma_{C^*}$ to $\mathbb{P}(C^*)$.

Furthermore, we obtain foliations of $C$-regular domains by CAGC surfaces:
\begin{thmx}\label{thm_introfoliation}
The family of surfaces
	$(\Sigma_k)_{k>0}$ from Theorem \ref{thm_intromain} is a foliation of $D$. Moreover, the function $K:D\rightarrow\R$ defined by $K|_{\Sigma_k}=\log k$ is convex.
\end{thmx}
This foliation has been studied in the Minkowski setting in \cite{bbz,bs,bon_smillie_seppi}, although the convexity of the ``time function'' $K$ seems to be new except when $D$ is the cone $C$ itself. When $D=C$, since every $\Sigma_k$ is a scaling of the Cheng-Yau affine sphere $\Sigma_C$,
the function $K$ is relatively easy to understand and is actually a solution to the following Monge-Amp\`ere equation on $C$:
\begin{equation}\label{eqn_intro4}
\begin{cases}
\det\D^2K=a\,e^{bK}\\
K|_{\pa C}=+\infty
\end{cases}
\end{equation}
($a,b>0$ are constants, see \cite[Appendix A]{sasaki}); whereas Cheng and Yau \cite{chengyau2} proved that \eqref{eqn_intro4} has a unique solution not only for proper convex cones $C\subset\R^3$ but for any proper convex domain in $\R^d$ ($d\geq 2$), and the Hessian of the solution gives a complete Riemannian metric on the domain. On a $C$-regular domain $D\subset\R^3$ which is not a translation of $C$, this solution does not coincide with the function $K$ from Theorem \ref{thm_introfoliation} in general, and it is worth further investigation whether $K$ is smooth and gives a complete Riemannian metric.

\subsection*{About the exterior circle condition} 
Let us give more discussions about the exterior circle condition in all the above results, which is responsible for the gradient blowup condition in \eqref{eqn_intro3} and the completeness of the surface $\Sigma_k$ from Theorem \ref{thm_intromain} as mentioned. In the last part of the paper, we will produce a class of examples where this condition is not satisfied and Theorems \ref{thm_intro1} and \ref{thm_intromain} do not hold:
\begin{propx}\label{prop_introimproper}
	Let $\Omega\subset \R^2$ be a bounded convex domain, $\Delta\subset\Omega$ be an open triangle with vertices on $\pa\Omega$ and $\phi$ be the function on $\pa\Omega$ vanishing at the vertices of $\Delta$ with $\phi=+\infty$ everywhere else.
	\begin{enumerate}
		\item\label{item_introimproper1} If $\Omega$ satisfies the exterior circle condition at every vertex of $\Delta$ (see Figure \ref{figure_improper} \subref{figure_improper1}), then there exists a unique $u$ satisfying \eqref{eqn_intro3} as in Theorem \ref{thm_intro1}.
		\item\label{item_introimproper2} If $\pa\Omega$ contains an open line segment meeting $\pa\Delta$ exactly at a vertex (see Figure \ref{figure_improper} \subref{figure_improper2}), then there does not exist $u$ satisfying \eqref{eqn_intro3}. 
		\end{enumerate}
\end{propx}
\vspace{-20pt}
	\begin{figure}[h]
		\centering
		\subfloat[\label{figure_improper1}]{\includegraphics[width=.34\linewidth]{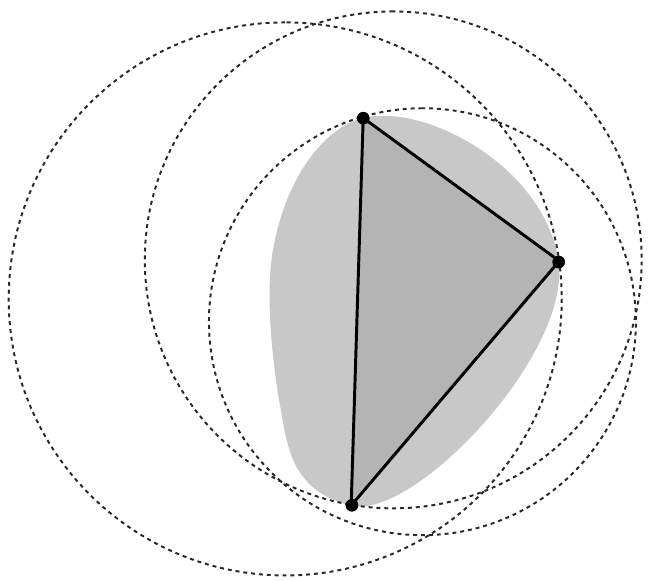}}
		\qquad
		\subfloat[\label{figure_improper2}]{\includegraphics[width=.24\linewidth]{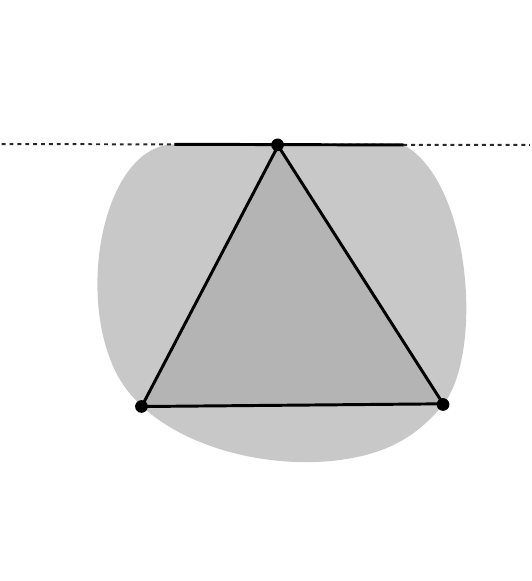}}
	\caption{$\Omega$ and $\Delta$ under the assumptions of Parts (\ref{item_introimproper1}) and (\ref{item_introimproper2}) of Proposition \ref{prop_introimproper}, respectively.}
		\label{figure_improper}
	\end{figure}

The reason for Part \eqref{item_introimproper2} is that any $u$ satisfying \eqref{eqn_intro3} can be shown to solve the first equation in \eqref{eqn_intro3} on the triangle $\Delta$ with vanishing boundary value on $\pa\Delta$, and then one can show that the gradient of $u$ does not blowup at the vertex of $\Delta$ on the line segment because the right-hand side of the equation does not go to $+\infty$ fast enough near the segment.

Geometrically, given a proper convex cone $C$, any circumscribed triangular cone $T$ as shown in Figure \ref{figure_tc} \subref{figure_tc3} is a $C$-regular domain, and we deduce from Proposition \ref{prop_introimproper} that if the radial projections of $C$ and $T$ on $\mathbb{RP}^2$ look like in  Figure \ref{figure_tc} \subref{figure_tc1}, which is dual to Figure \ref{figure_improper} \subref{figure_improper1}, then $T$ is foliated by affine $(C,k)$-surfaces as in Theorem \ref{thm_introfoliation}; whereas in the case of Figure \ref{figure_tc} \subref{figure_tc2}, dual to Figure \ref{figure_improper} \subref{figure_improper2}, $T$ is not generated by any complete affine $(C,k)$-surface. See Corollary \ref{coro_trianglecone} for details.
\begin{figure}[h]
	\centering
	\subfloat[\label{figure_tc1}]{\includegraphics[width=.25\linewidth]{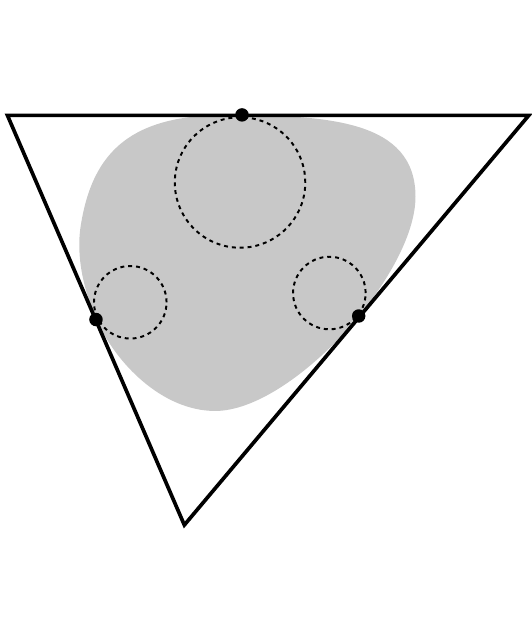}}
		\qquad
	\subfloat[\label{figure_tc2}]{\includegraphics[width=.25\linewidth]{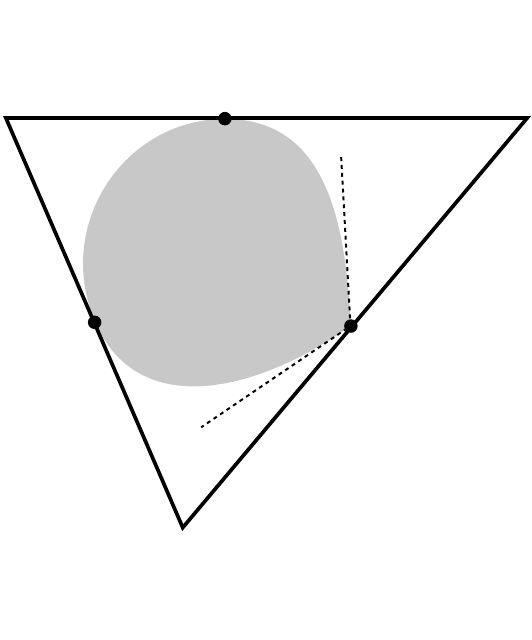}}
		\qquad
		\subfloat[\label{figure_tc3}]{\includegraphics[width=.26\linewidth]{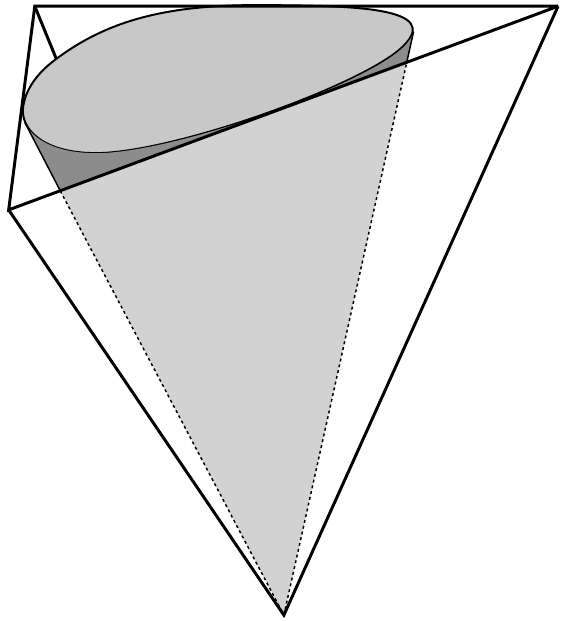}}
	\caption{The dual pictures of Figure \ref{figure_improper1} and \ref{figure_improper2}, and a convex cone with a circumscribed triangular cone.}
    \label{figure_tc}
\end{figure}

Finally, we point out that if the domain $D$ in Theorem \ref{thm_intromain} is preserved by an affine action of the fundamental group $\pi_1(S)$ of a closed topological surface $S$ and the linear part of the action is a \emph{Hitchin representation} $\pi_1(S)\to\SL(3,\R)$ preserving the cone $C$, then the unique existence of a CAGC 1 surface preserved by the action is already proved by Labourie \cite[Section 8]{labourie} from a different point of view. This in turn implies the unique solvability of \eqref{eqn_intro3} for the corresponding $\Omega$ and $\phi$.
In this case, $\Omega$ is a \emph{convex divisible set} \cite{benoist, benoist_survey} and does not satisfy the exterior circle condition. Nevertheless, $\pa\Omega$ and $\phi$ still have certain regularity properties, at least $\phi$ is $\R$-valued and continuous (see also \cite{benoist, guichard} for regularity results on $\pa\Omega$). Finding a simple necessary condition for unique solvability of \eqref{eqn_intro3} covering this case is a problem to be further investigated. 

Despite not being covered by our main results, the case of Hitchin representations has particular geometric significance  and is therefore one of the motivations behind our work. This is because a $C$-regular domain naturally arises in this case as \emph{domain of discontinuity}. In fact, in the Minkowski setting, Mess \cite{mess} showed  that any isometric action of $\pi_1(S)$ on $\R^{2,1}$ with linear part a \emph{Fuchsian representation} $\pi_1(S)\to\SO(2,1)$ is properly discontinuous on some regular domain $D$ and the quotient $D/\pi_1(S)\cong S\times\R$ is a prototype of \emph{maximal globally hyperbolic flat spacetimes}. We anticipate a similar result for affine actions with Hitchin linear parts. However, we will not pursue surface group actions further in this paper.
\vspace{-3pt}
\subsection*{Organization and methods of the paper} 
After reviewing backgrounds on affine differential geometry in Section \ref{sec_affinedg}, we introduce in Section \ref{sec_cregular} the main objects of this paper:  We first define $C$-regular domains and \emph{$C$-convex hypersurfaces}, a natural class of convex hypersurface generalizing the so-called \emph{future-convex} spacelike hypersurfaces in the Minkowski space, then we discuss hypersurfaces with CAGC and define affine $(C,k)$-hypersurfaces. 

Section \ref{sec_convex_prel} introduces some tools from convex analysis, such as convex envelopes, subgradient and Legendre transformation. Then in Section \ref{sec_legendre} we relate $C$-regular domains and {$C$-convex hypersurfaces} to the setting of convex analysis, by establishing some fundamental correspondences that enable to translate our geometric problems into an analytic framework.

Moving forward to the analytic setup, Section \ref{subsec_preliminaries} briefly reviews some ingredients in the theory of Monge-Amp\`ere equations, then the last two sections provide the proofs of our main results. Section \ref{sec_cagc} first provides the concrete formulation of the problem as in \eqref{eqn_intro3}, and then obtains some fundamental local estimates, close to the boundary, on the Cheng-Yau solutions of \eqref{eqn_intro1}.

We finally solve the Monge-Amp\`ere problem \eqref{eqn_intro3} in dimension $2$ in Section \ref{sec_analysis}. 
While the existence and uniqueness of the Monge-Amp\`ere equation with Dirichlet boundary condition follows from more or less standard arguments, proving the gradient blowup involved more subtle estimates, which rely on the estimates on the Cheng-Yau solutions  given in Section \ref{sec_cagc}. Applying techniques of a similar type, we also produce the counterexamples of Proposition \ref{prop_introimproper}.

\subsection*{Acknowledgments} 
We are grateful to Francesco Bonsante for helpful discussions and to Connor Mooney for helps with the proofs in Section \ref{subsec_infiniteness}. The work leading to this paper was done during the first author's visit to University of Pavia and the second author's visit to KIAS. We would like to thank the respective institutes for their warm hospitality. The second author is member of the Italian national research group GNSAGA.

\setcounter{tocdepth}{1}
\tableofcontents

\section{Preliminaries on affine differential geometry}\label{sec_affinedg}
The section gives a concise review of background materials on affine differential geometry used later on.
\subsection{Intrinsic data of affine immersions relative to a transversal vector field}\label{subsec_intrinsic}
Fix $d=n+1\geq 3$. For conceptual clearness, we use different notations for $\mathbb{R}^d$ when different structures are under consideration: 
\begin{itemize}
	\item Let $\V$ denote the $d$-dimensional real vector space endowed with a volume form. We consider the volume form as the determinant $\det:\bigwedge^d\V\rightarrow\mathbb{R}$. 
	\item Let $\V^*$ denote the vector space dual to $\V$, consisting of linear forms on $\V$.  
	\item Let $\A$ denote the $d$-dimensional affine space modeled on $\V$, endowed with the volume form given by that on $\V$. In other words, $\A$ is obtained from $\V$ by ``forgetting'' the origin.
\end{itemize}
In this section, by a \emph{hypersurface} in $\A$, we always mean a smooth embedded one, which has dimension $n$. Given a hypersurfaces $\Sigma\subset\A$,  Affine Differential Geometry studies properties of $\Sigma$ invariant under volume-preserving affine transformations of $\A$. To achieve this, one considers the intrinsic invariants of $\Sigma$ defined relative to a transversal vector field $N:\Sigma\rightarrow\V$. These invariants include a volume form $\nu$, a torsion free affine connection $\nabla$, a symmetric $2$-tensor $h$, a $(1,1)$-tensor $S$ and a $1$-form $\tau$ on $\Sigma$, determined as follows:
\begin{itemize}
	\item
	$\nu$ is the \emph{induced volume form}, defined by
	$$
	\nu(X_1,\cdots,X_n)=\det(X_1,\cdots, X_n,N),
	$$
	Here and below, $X$, $Y$ and $X_1,\cdots, X_n$ are any tangent vector fields on $\Sigma$. 
	\item
	$\nabla$ and $h$ are the \emph{induced affine connection} and \emph{affine metric}, respectively, determined by
	$$
	D_X Y=\nabla_X Y+h(X,Y)N.
	$$
	\item
	$S$ and $\tau$ are the \emph{shape operator} and \emph{transversal connection form}, determined by
	$$
	D_XN=S(X)+\tau(X)N.
	$$
\end{itemize}
\begin{remark}
	Following the original convention of Blaschke, in the literature a minus sign is often added to the definition of $S$, which is inconvenient for our purpose.
\end{remark}

Given an open set $U\subset\R^n$, an immersion $f: U\rightarrow\A$ and a map $N:U\to\V$ transversal to $f$, one can define the above invariants on $U$ by pullback. The \emph{fundamental theorem} of affine differential geometry, similar to the one in classical surface theory, states that these invariants satisfy certain structural equations, and conversely when $U$ is simply connected, any prescribed invariants satisfying the equations are realized by some $f$ and $N$, unique up to volume-preserving affine transformations. Therefore, we call the quintuple of invariants $(\nu, \nabla, h, S, \tau)$ the \emph{intrinsic data} of the pair $(\Sigma,N)$ or $(f,N)$.

We give below some basic facts about these invariants to keep in mind. Here, a smooth function is said to be \emph{locally strongly convex} if its hessian is positive definite everywhere, while a hypersurface is said to be locally strongly convex if it can be presented locally as graphs of such functions.
\begin{itemize}
	\item[-] 
	The rank and signature of the affine metric $h$ only depends on $\Sigma$, not on $N$, and $\Sigma$ is said to be \emph{non-degenerate} if $h$ is. Thus, $h$ is a genuine pseudo-Riemannian metric in this case.
	\item[-] 
	A locally convex hypersurface is non-degenerate if and only if it is locally strongly convex. Moreover, $h$ is a Riemannian metric if and only if $\Sigma$ is locally strongly convex with $N$ pointing towards the convex side.
	\item[-] 
	We have $\dif\tau=0$ if and only if $h(X,S(Y))=h(S(X),Y)$. As a consequence, if $\Sigma$ is non-degenerate and $\dif\tau=0$ then $S$ has $n$ real eigenvalues everywhere.
\end{itemize}


A transversal vector field $N$ is said to be \emph{equi-affine}  if the transversal connection form $\tau$ vanishes. 
We thus omit the $\tau$ component when talking about intrinsic data in this case. From the definitions of $\tau$ and $S$, we deduce the following basic fact:
\begin{lemma}\label{lemma_equiaffine}
	Given $p\in\Sigma$, we have $\tau=0$ at $p$ if and only if the image of the tangent map $\dif N_p:\T_p\Sigma\to\V$ is contained in the tangent space $\T_p\Sigma\subset\T_p\A\cong\V$. In this case, the image of $\dif N_p$ equals $\T_p\Sigma$ if and only if the shape operator $S$ is non-degenerate at $p$.
\end{lemma}

\subsection{Affine normals and affine spheres}\label{subsec_affinenormal}
If $\Sigma$ is non-degenerate, there exists an equi-affine vector field $N$ such that the induced volume form $\nu$ coincides with the volume form $\dif\vol_h$ of the affine metric $h$ (note that $\dif\vol_h$ is defined using the orientation induced by $\nu$). The vector field $N$ is unique up to sign and is called an \emph{affine normal field} of $\Sigma$, or an \emph{affine normal mapping} when $N$ is viewed as a map from $\Sigma$ to $\V$. As before, we extend the definition to define affine normal mapping $N:U\to\V$ of an immersion $f:U\to\A$.

Given an equi-affine transversal vector field $N$, if $\nu$ is a constant multiple of $\dif\vol_h$, one can scale $N$ to get an affine normal field:
\begin{lemma}\label{lemma_scale}
	Let $\Sigma\subset\A$ be a hypersurface with equi-affine transversal vector field $N:\Sigma\to\V$ and intrinsic data $(\nu,\nabla,h,S)$. Suppose $\dif\vol_h=a\nu$ for a constant $a\neq 0$. Then $\pm|a|^{\frac{2}{n+2}}N$ are the affine normal fields of $\Sigma$ (recall that $\dim(\A)=n+1$).
\end{lemma}
\begin{proof}
This follows from the definition and the fact that if we scale $N$ by a constant $\lambda\neq 0$, then $\nu$ and $h$ get scaled by $\lambda$ and $\frac{1}{\lambda}$, respectively.
\end{proof}

While in general there is no privileged choice between the two affine normal fields opposite to each other, we do have one when $\Sigma$ is locally strongly convex: the one pointing towards the convex side of $\Sigma$. In this case, we call the shape operator $S$ of $\Sigma$ with respect to this choice of $N$ the \emph{affine shape operator} and call the determinant $\det(S):\Sigma\to\R$ the \emph{affine Gauss-Kronecker curvature}, or simply \emph{affine Gaussian curvature}.
\begin{definition}[\textbf{Proper and hyperbolic/elliptic affine spheres}]\label{define_affinesphere}
	A non-degenerate hypersurface $\Sigma\subset\A$ is called a \emph{proper affine sphere} if its shape operator with respect to an affine normal field is $S=\lambda\,\id$ for a constant $\lambda\neq0$ (where $\id$ denotes the identity endomorphism of $\T\Sigma$). This condition is equivalent to the existence of a point $o\in\A$, called the \emph{center} of $\Sigma$, such that $N(p)=\lambda\overrightarrow{op}$ (for all $p\in\Sigma$) is an affine normal field. Furthermore, $\Sigma$ is called a \emph{hyperbolic} (resp. \emph{elliptic}) affine sphere if $\Sigma$ is locally convex and the condition is satisfied with $\lambda>0$ (resp. $\lambda<0$) for the affine normal field pointing towards the convex side of $\Sigma$. 
\end{definition}
\begin{remark}
	The terminology ``proper'' is used here as opposed to the case $S=0$, in which $\Sigma$ is known as an \emph{improper} affine sphere.  
\end{remark}
Thus, the center of a hyperbolic (resp. elliptic) affine sphere $\Sigma$ lies on the concave (resp. convex) side of $\Sigma$. In the sequel, when talking about proper affine spheres in the vector space $\V$, we always mean those centered at the origin $0\in\V$.

\subsection{Intrinsic data of equi-affine vector field as centro-affine immersion}\label{subsec_centroaffine}
A hypersurface $\Sigma\subset\V$ is said to be \emph{centro-affine} if the position vector $\overrightarrow{0p}$ of every point $p\in\Sigma$ is transversal to $\Sigma$. Thus, proper affine spheres in $\V$ are centro-affine.


Given a hypersurface $\Sigma\subset\A$ with equi-affine transversal vector field $N:\Sigma\to\V$, if the shape operator of $(\Sigma, N)$ is non-degenerate, then  Lemma \ref{lemma_equiaffine} implies that $N$ is a centro-affine immersion of $N$ into $\V$. Its intrinsic data with respect to the position vector field are given by the intrinsic data of $(\Sigma, N)$ as follows:
\begin{lemma}\label{lemma_data}
	Let $\Sigma\subset\A$ be a hypersurface, $
	N:\Sigma\rightarrow\V$ be an equi-affine transversal vector field and $(\nu,\nabla,h,S)$ be the intrinsic data of $(\Sigma, N)$. Suppose $\det(S)\neq0$ on $\Sigma$ and consider $N$ as a centro-affine hypersurface immersion. Then the intrinsic data of $N$ with respect to its position vector field (given by $N$ itself) is 
	$$
	(\det(S)\nu,\,S^{-1}\nabla S,\, h(\cdot,S(\cdot)),\,\id).
	$$
\end{lemma}
Here, the affine connection $\nabla'=S^{-1}\nabla S$ is the gauge transform of $\nabla$ by $S^{-1}$, defined by $\nabla'_XY:=S^{-1}\nabla_X(S(Y))$ for any tangent vector fields $X$ and $Y$ on $\Sigma$.

To prove Lemma \ref{lemma_data}, we use the following framework to compute intrinsic data in coordinates, which is also needed in Section \ref{subsec_charamonge} below. Let $U\subset\R^n$ be an open set, $f:U\rightarrow\A$ be an immersion and $N:U\to\V$ be a transversal vector field to $f$. Then the induced volume form of $(f,N)$ is
\begin{equation}\label{eqn_inducedvolumeform}
\nu=\det(\pa_1f,\cdots, \pa_n f, N)\,\dx^1\wedge\cdots\wedge \dx^n,
\end{equation}
where $x=(x^1,\cdots, x^n)$ is the coordinate on $U$ and $\pa_if(x),N(x)\in\V\cong\R^{n+1}$ are written in coordinates as column vectors. The other intrinsic invariants are determined by the following equality of $(n+1)\times(n+1)$ matrices of $1$-forms:
\begin{equation}\label{eqn_matrix}
\dif(\pa_1f,\cdots,\pa_nf, N)=(\pa_1f,\cdots,\pa_nf, N)
\begin{pmatrix}
A&S\,\dx\\
\transp\dx\,h&\tau
\end{pmatrix}.
\end{equation}
Here, $\dif x$ represents the column vector of $1$-forms $\transp(\dx^1,\cdots \dx^n)$. The shape operator $S=(S^i_j)$ and affine metric $h=(h_{ij})$ are written as $n\times n$ matrices of functions on $U$, and $A$ is the matrix of $1$-forms such that the induced affine connection $\nabla$ is expressed under the frame $(\pa_1,\cdots,\pa_n)$ as $\nabla=\dif+A$. Note that the gauge transform of $\nabla$ by $S^{-1}$ is $S^{-1}\nabla S=\dif+S^{-1}AS+S^{-1}\dif S$.

\begin{proof}[Proof of Lemma \ref{lemma_data}]
	Let $U\subset\R^n$ be an open set. We can replace $\Sigma$ and $N$ in the statement of the lemma by an immersion $f:U\to\A$ and a map $N:U\to\V$ transversal to $f$, respectively.
	
	The intrinsic data $(\nu,\nabla,h,S)$ of $(f,N)$ are determined by the equalities (\ref{eqn_inducedvolumeform}) and (\ref{eqn_matrix}) with $\tau=0$. The last column of (\ref{eqn_matrix}) gives $\dif N=(\pa_1f,\cdots,\pa_nf)S\,\dx$, hence
	\begin{equation}\label{eqn_proofdata}
	(\pa_1N,\cdots,\pa_nN,N)=(\pa_1f,\cdots,\pa_nf, N)
	\begin{pmatrix}
	S&\\
	&1
	\end{pmatrix}.
	\end{equation}
	Therefore, the induced volume form of $N$ is
	$$
	\nu'=\det(\pa_1N,\cdots,\pa_nN,N)\,\dx^1\wedge\cdots\wedge\dx^n=\det(S)\nu,
	$$
	as required, while the induced affine connection $\nabla'=\dif+A'$ and affine metric $h'$ are characterized by
	$$
	\dif(\pa_1N,\cdots,\pa_nN, N)=(\pa_1N,\cdots,\pa_nN, N)
	\begin{pmatrix}
	A'&\dx\\[3pt]
	\transp\dx\,h'&0
	\end{pmatrix}.
	$$
	By (\ref{eqn_proofdata}) and (\ref{eqn_matrix}), the left-hand side is
	\begin{align*}
	\dif(\pa_1N,\cdots,\pa_nN, N)&=\dif(\pa_1f,\cdots,\pa_nf,N)
	\begin{pmatrix}
	S&\\
	&1
	\end{pmatrix}
	+(\pa_1f,\cdots,\pa_nf,N)
	\begin{pmatrix}
	\dif S&\\
	&0
	\end{pmatrix}\\
	&=(\pa_1f,\cdots,\pa_nf,N)\left[
	\begin{pmatrix}
	A&S\,\dx\\
	\transp{\dx}\,h&0
	\end{pmatrix}
	\begin{pmatrix}
	S&\\
	&1
	\end{pmatrix}
	+
	\begin{pmatrix}
	\dif S&\\
	&0
	\end{pmatrix}
	\right]\\
	&=(\pa_1N,\cdots,\pa_nN, N)
	\begin{pmatrix}
	S^{-1}&\\
	&1
	\end{pmatrix}
	\begin{pmatrix}
	AS+\dif S&S\,\dx\\
	\transp{\dx}\,h S&0
	\end{pmatrix}.
	\end{align*}
	Comparing with the right-hand side, we get the required expressions
	$$
	A'=S^{-1}AS+S^{-1}\dif S,\quad h'=hS.
	$$
\end{proof}

\subsection{Complete hyperbolic affine spheres}\label{subsec_completehyperbolic}
The following theorem classifies hyperbolic affine spheres that are \emph{complete} in the sense that the Riemannian metric on it induced by an ambient Euclidean metric is complete
(see also Remark \ref{remark_completeness} below): 
\begin{theorem}[Cheng-Yau \cite{chengyau1}]\label{thm_chengyau1}
	For any proper convex cone $C\subset\V$, there exits a unique complete hyperbolic affine sphere $\Sigma_C\subset C$ asymptotic to $\pa C$ with affine shape operator the identity.
\end{theorem}
Here, $\Sigma_C$ being \emph{asymptotic} to $\pa C$ means the distance from $x\in\Sigma$ to $\pa C$ with respect to an ambient Euclidean metric tends to $0$ as $x$ goes to infinity in $\Sigma$.

The theorem essentially gives all complete hyperbolic affine spheres because for any such affine sphere $\Sigma\subset\V$, since $\Sigma$ is centro-affine, the projectivization map $\mathbb{P}:\V\setminus\{0\}\to\mathbb{RP}^n$ is a diffeomorphism from $\Sigma$ to a convex domain in $\mathbb{RP}^n$ and if we let $C$ be the component of the pre-image $\mathbb{P}^{-1}(\mathbb{P}(\Sigma))$ containing $\Sigma$, then $C$ is a proper convex cone and $\Sigma$ is a scaling of the affine sphere $\Sigma_C$ from the theorem. For later use, we determine the precise relation between the scale factor and the affine shape operator of $\Sigma$ as follows:
\begin{proposition}\label{prop_scaling}
	Given $\lambda>0$ and a proper convex cone $C\subset\V$, the scaling $\lambda\Sigma_C$ of the affine sphere $\Sigma_C$ from Theorem \ref{thm_chengyau1} is the unique complete hyperbolic affine sphere asymptotic to $\pa C$ with affine shape operator $\lambda^{-\frac{2(n+1)}{n+2}}\id$ (where $\dim(\V)=n+1$).
\end{proposition}
\begin{proof}
Let $h$ and $\nu=\dif\vol_h$ be the affine metric and induced volume form of $\Sigma_C$, which are defined with respect to the affine normal field of $\Sigma_C$, namely the position vector field. Letting $f:\Sigma_C\to \lambda\Sigma_C$ denote the scaling map, one can check from the definitions that the induced volume form $\tilde{\nu}$ and affine metric $\tilde{h}$ of $\lambda\Sigma_C$ with respect to its own position vector field are related to the push-forwards of $\nu$ and $h$ by
$$
\tilde{\nu}=\lambda^{n+1}f_*\nu,\quad \tilde{h}=f_*h.
$$
So we have $\dif\vol_{\tilde{h}}=f_*\dif\vol_h=f_*\nu=\lambda^{-n-1}\tilde{\nu}$, and Lemma \ref{lemma_scale} implies that the affine normals of $\lambda\Sigma_C$ are $\lambda^{-\frac{2(n+1)}{n+2}}$ times its position vectors. The required expression of affine shape operator follows.
\end{proof}

The following fundamental examples are among the rare cases where $\Sigma_C$ admits an explicit expression. See \eg \cite{MR2743442} for details.
\begin{example}[\textbf{Hyperboloid}]\label{example_hyperboloid}
The \emph{Minkowski space} $\R^{n,1}$ is $\V$ endowed with a bilinear form of signature $(n,1)$ whose underlying volume form is the prescribed one. By convention, we pick a component $C_0$ of the quadratic cone
$\{v\in\V\mid \bii{v}{v}<0\}$ and call $C_0$ the \emph{future light cone} in $\R^{n,1}$.
	The affine sphere $\Sigma_{C_0}$ claimed by Theorem \ref{thm_chengyau1} turns out to be the component of the two-sheeted hyperboloid $\{\bii{v}{v}=-1\}$ in $C_0$, namely
	$\Sigma_{C_0}=\mathbb{H}:=\{v\in C_0\mid \bii{v}{v}=-1\}$.
\end{example}
	
\begin{example}[\textbf{\c{T}i\c{t}eica affine sphere}]\label{example_titeica}
	Let $(v_0,\cdots, v_n)$ be a unimodular basis of $\V$.  For the simplicial cone $C_1:=\{t_0v_0+\cdots+t_nv_n\mid t_i>0\}$, there is a constant $\Lambda_n >0$ only depending on $n$ such that the affine sphere claimed by Theorem \ref{thm_chengyau1} is
	$$
	\Sigma_{C_1}=\{t_0v_0+\cdots+t_nv_n\mid t_1,\cdots, t_n>0,\, t_0\cdots t_n=\Lambda_n \}.
	$$
\end{example}

\subsection{Affine conormals and dual affine sphere}\label{subsec_conormal}
Given a hypersurface $\Sigma\subset\A$ with affine normal field $N:\Sigma\to\V$, the \emph{affine conormal} of $\Sigma$ at $p\in\Sigma$ dual to $N$ is the linear form $N^*(p)\in \V^*$ defined by
$$
\pair{N^*(p)}{N(p)}=1,\quad \pair{N^*(p)}{\T_p\Sigma}=0,
$$
where ``$\pair{\cdot}{\cdot}$'' is the pairing between $\V$ and $\V^*$ and the last equality means $\pair{N^*(p)}{v}=0$ for every $v\in\T_p\Sigma\subset\T_p\V\cong\V$.
We call $N^*:\Sigma\to\V^*$ the affine conormal mapping dual to $N$.

If $\Sigma$ has non-degenerate shape operator with respect to $N$, so that $N$ is an immersion of $\Sigma$ into $\V$ as a centro-affine hypersurface (see Lemma \ref{lemma_equiaffine} and Section \ref{subsec_centroaffine}), then $N^*:\Sigma\to\V^*$ is a centro-affine immersion as well and its image $N^*(\Sigma)$ is the centro-affine hypersurface $M^*$ \emph{dual} to $M=N(\Sigma)$, defined by 
$$
M^*=\{\alpha\in\V^*\mid \text{ there is $v\in M$ such that $\pair{\alpha}{v}=1$, $\pair{\alpha}{\T_vM}=0$}\}.
$$

When $M$ is a proper affine sphere with affine shape operator $\lambda\,\id$, it can be shown that the dual centro-affine hypersurface $M^*$ defined above is a proper affine sphere in $\V^*$ with affine shape operator $\lambda^{-1}\,\id$, so $M^*$ is called the \emph{dual affine sphere} of $M$. It can also be shown that for the hyperbolic affine sphere $\Sigma_C$ from Theorem \ref{thm_chengyau1}, the dual affine sphere is exactly $\Sigma_{C^*}$, where $C^*\subset\V^*$ is the dual cone of $C$, consisting of linear forms on $\V$ which take positive values on $\overline{C}\setminus\{0\}$. It follows that the dual affine sphere of $\lambda\Sigma_C$ is $\lambda^{-1}\Sigma_{C^*}$.

\section{$C$-regular domains and hypersurfaces with CAGC}\label{sec_cregular}
In this section, we define the main objects of study of this paper: $C$-regular domains, $C$-convex hypersurfaces and affine $(C,k)$-hypersurfaces. When $C$ is the future light cone $C_0$ in the Minkowski space $\R^{n,1}$, the first two objects are known in the literature as regular domains and future-convex spacelike hypersurfaces, respectively, while we show in Section \ref{subsec_minkowski} that affine $(C_0,k)$-hypersurfaces are exactly $C_0$-convex hypersurface with constant Gaussian curvature in the classical sense.

\subsection{$C$-regular domains and $C$-convex hypersurfaces}\label{subsec_cregular}
As in Section \ref{sec_affinedg}, we fix $n\geq 2$, let $\V$ denote an $(n+1)$-dimensional vector space equipped with a volume form and $\A$ denote the affine space modeled on $\V$.
By a \emph{convex domain}, we mean a convex open set, while a \emph{convex cone} in $\V$ is a convex domain invariant under positive scalings. A convex cone/domain is \emph{proper} if it is nonempty and does not contain any entire straight line.  We further introduce the following definitions and notations:
\begin{itemize}
	\item  Let $\Ha_\A$ denote the space of all \emph{open half-spaces} of $\A$, \ie open subsets whose boundaries are affine hyperplanes.
	\item $\Ha_\V$ denote the space of all open half-spaces of $\V$ with boundaries passing through the origin $0\in\V$. Thus, there is a natural projection $\Ha_\A\to\Ha_\V$ such that the pre-image of $H\in\Ha_\V$ consists of the translations of $H$.
	\item A \emph{supporting half-space} of a convex domain $D$ at a boundary point $p\in\pa D$ is an open half-space $H$ such that $D\subset H$ and $p\in \pa H$.
	\item  Given a convex cone $C\subset\V$, we let $\Ha_\V(C)\subset\Ha_\V$ denote the space of all supporting half-spaces of $C$ and let $\Ha_\V^0(C)\subset\Ha_\V(C)$ denote the set of supporting half-spaces at boundary points other than the origin. Also put $\Ha_\V^1(C):=\Ha_\V(C)\setminus\Ha_\V^0(C)$. See Figure \ref{figure_planes}.
	\begin{figure}[h]
		\includegraphics[width=1.6in]{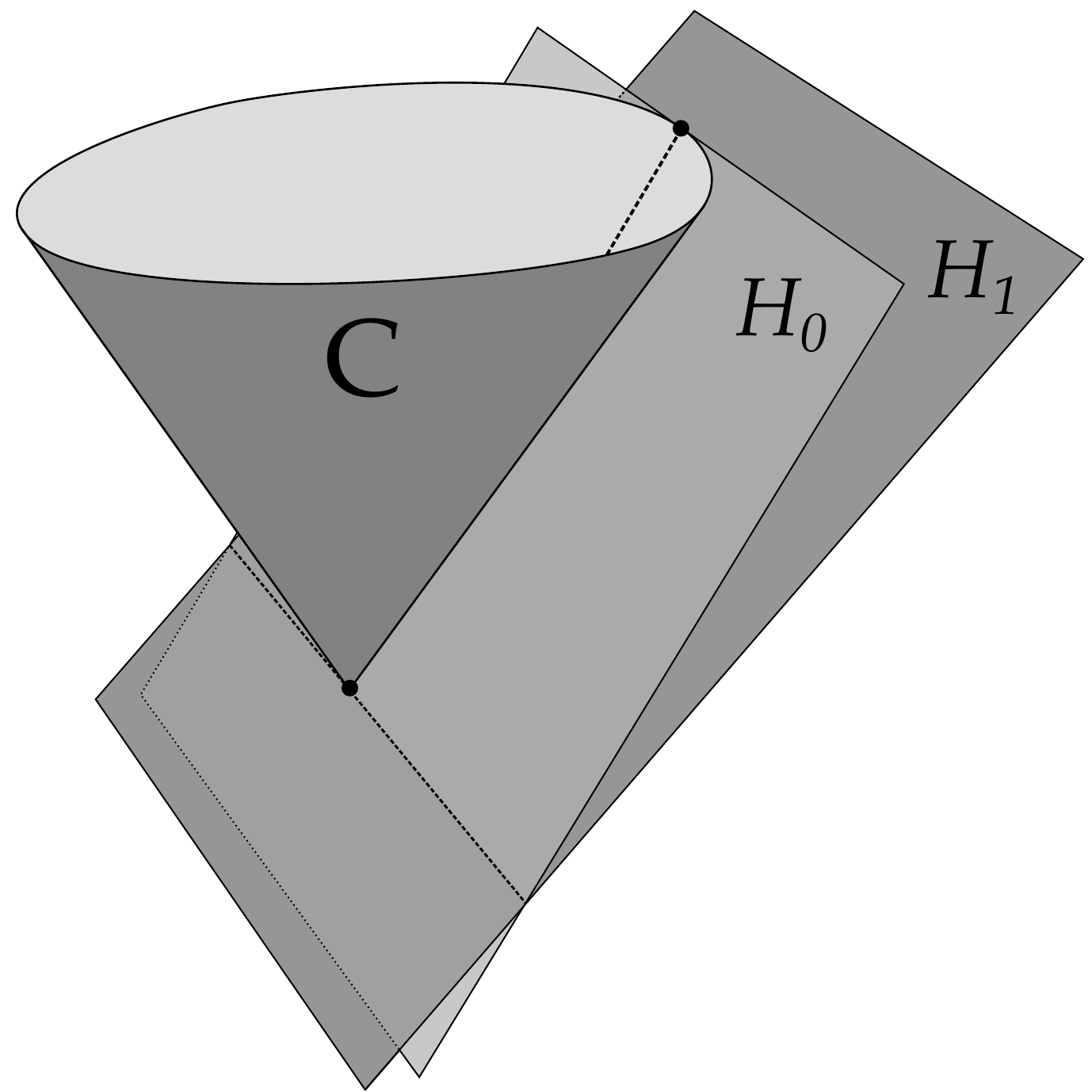}
		\caption{
			The cone $C$ and the boundary hyperplanes of some $H_0\in\Ha_\V^0$ and $H_1\in\Ha_\V^1$. The half-spaces $H_0$ and $H_1$ are the parts of $\V$ above the respective hyperplanes.}
		\label{figure_planes}
	\end{figure}
	\item Let $\Ha_\A^0(C)$ and $\Ha_\A^1(C)$ denote the pre-images of $\Ha_\V^0(C)$ and $\Ha_\V^1(C)$ in $\Ha_\A$, respectively.
\end{itemize}
Adapting terminologies from Minkowski geometry, we call elements of $\Ha_\A^0(C)$ and $\Ha_\A^1(C)$ \emph{$C$-null} and \emph{$C$-spacelike} half-spaces, respectively, and call an affine hyperplane $C$-null/$C$-spacelike if it is the boundary of a $C$-null/$C$-spacelike half-space.
When $C$ is the future light cone $C_0$ in the Minkowski space $\R^{n,1}$ (see Example \ref{example_hyperboloid}), these are just null/spacelike hyperplanes in the classical sense.  

The following notion also arises from the Minkowski setting: 
\begin{definition}[\textbf{$C$-regular domains}]\label{def_cregular}
	Given a proper convex cone $C\subset\V$, a \emph{$C$-regular domain} is by definition a subset  $D\subset\A$ of the form
	$$
	D=\interior\left(\bigcap_{H\in \mathcal{H}}\overline{H}\right),\ \mathcal{H}\subset \Ha_\A^0(C).
	$$
Namely, $D$ is the interior of the intersection of a collection $\mathcal{H}$ of \emph{closed} $C$-null half-spaces. When $\mathcal{H}$ is the set of all $C$-null half-spaces containing some subset $S$ of $\A$,  $D$ is called the $C$-regular domain \emph{generated by $S$}. 
\end{definition}
\begin{remark}
For the future light cone $C_0\subset\R^{n,1}$, such domains first appeared in mathematical relativity as \emph{domains of dependence}, then the name ``regular domain'' is introduced in \cite{MR2170277}. Our definition is slightly wider even for $C=C_0$, in that regular domains in \cite{MR2170277} correspond to \emph{proper} $C_0$-regular domains under our definition. We refer to the papers cited in the introduction for the role of such domains in the study of globally hyperbolic flat spacetimes.
\end{remark}

Note that $C$-regular domains are convex domains, and the simplest examples include the empty set $\emptyset$, the whole $\A$ (corresponding to $\mathcal{H}=\emptyset$), single $C$-null half-spaces, and translations of the cone $C$ itself.

A convex hypersurface $\Sigma\subset\A$ is by definition a nonempty open subset of the boundary of some  convex domain, and a supporting half-space of the domain at a point $p\in\Sigma$ is referred to as a supporting half-space of $\Sigma$ at $p$.  We will study the following particular type of convex hypersurfaces,  known as \emph{future-convex spacelike} hypersurfaces when $C=C_0\subset\R^{n,1}$:
\begin{definition}[\textbf{$C$-convex hypersurfaces}]\label{def_cconvex}
	Given a proper convex cone $C\subset\V$, a \emph{$C$-convex} hypersurface is a convex hypersurface $\Sigma\subset\A$ whose supporting half-spaces are all $C$-spacelike. $\Sigma$ is said to be \emph{complete} if it is the entire boundary of some convex domain. 
\end{definition}
\begin{remark}\label{remark_completeness}
	For a  \emph{locally} convex immersed hypersurface, there is a nontrivial relation between the notion of completeness defined above and the completeness of the geodesic metric on the hypersurface induced by an ambient Euclidean metric, see \cite{heij}. However, since we only consider embedded globally convex hypersurfaces, these notions are the same.
\end{remark}

As an example, for the affine sphere $\Sigma_C\subset C$ given by Theorem \ref{thm_chengyau1}, using the fact that $\Sigma_C$ is equi-affine and asymptotic to $\pa C$, it can be shown that any scaling $\lambda\Sigma_C$ is a complete $C$-convex hypersurface generating the convex cone $C$. 

In Section \ref{sec_legendre}, we will identify all complete $C$-convex hypersurfaces in $\A$ as the entire graphs of a specific class of convex functions on $\R^n$, and identify $C$-regular domains as strict epigraphs of specific convex functions as well. We will also see that if a complete $C$-convex hypersurface $\Sigma$ is asymptotic to the boundary of a $C$-regular domain $D$ (in the sense defined in Section \ref{subsec_completehyperbolic}), then $\Sigma$ generates $D$, whereas the converse is not true.


\subsection{Convex hypersurfaces with Constant Affine Gaussian Curvature}\label{subsec_cagc}
With the definitions from Section \ref{sec_affinedg} in mind, by a hypersurface in $\A$ with \emph{constant affine Gaussian curvature} (CAGC), we mean a non-degenerate smooth embedded hypersurface $\Sigma\subset\A$ such that the shape operator of $\Sigma$ with respect to an affine normal field has constant determinant. Since there are two affine normal fields opposite to each other, the precise value of affine Gaussian curvature (\ie the determinant) has sign ambiguity when $n$ is odd. When $\Sigma$ is locally convex, the ambiguity is eliminated by picking the affine normals pointing towards the convex side as mentioned in Section \ref{subsec_affinenormal}. 

The following result characterizes CAGC hypersurfaces by affine normal mappings and singles out a subclass of these hypersurfaces which we study later on:
\begin{proposition}\label{prop_cagc}
	Let $\Sigma\subset\A$ be a non-degenerate smooth  hypersurface with affine normal mapping $N:\Sigma\rightarrow\V$. Let $S$ be the shape operator of $\Sigma$ with respect to $N$ and suppose $\det(S)\neq0$ on $\Sigma$, so that $N$ is an immersion (see Lemma \ref{lemma_equiaffine}). Then 
	\begin{enumerate}
		\item\label{item_cagc1}
		$\det(S)$ is constant if and only if $N(\Sigma)$ is a proper affine sphere in $\V$. In this case, an affine normal field of $N(\Sigma)$ is given by $|\det(S)|^{-\frac{1}{n+2}}$ times its position vectors.
		\item\label{item_cagc2} 
		Further assume that $\Sigma$ is locally strongly convex and $N$ points towards the convex side of $\Sigma$. Then $N(\Sigma)$ is a hyperbolic affine sphere if and only if $\det(S)$ is a constant and the eigenvalues of $S$ are all positive at every point of $\Sigma$. 
	\end{enumerate}
\end{proposition}
\begin{proof}
	(\ref{item_cagc1})
	Let $(\nu, \nabla,h,S)$ be the intrinsic data of $(\Sigma,N)$, where $\nu=\dif\vol_h$ is the volume form of $h$ since $N$ is an affine normal field. 
	By Lemma \ref{lemma_data}, the induced volume form $\nu'$ and  affine metric $h'$ of the centro-affine immersion $N$ with respect to its position vector field are given by $\nu'=\det(S)\nu$ and $h'(\cdot,\cdot)=h(\cdot,S(\cdot))$.  The volume form of $h'$ is 
	$$
	\dif\vol_{h'}=|\det(S)|^\frac{1}{2}\dif\vol_h=|\det(S)|^\frac{1}{2}\nu=\pm |\det(S)|^{-\frac{1}{2}}\nu'.
	$$ 
  In particular, $\det(S)$ is a constant if and only if $\dif\vol_{h'}$ is a constant times $\nu'$. But Lemma \ref{lemma_scale} implies that $\dif\vol_{h'}$ is a constant $a\neq0$ times $\nu'$ if and only $|a|^{\frac{2}{n+2}}$ times the position vector field of $N(\Sigma)$ is an affine normal field, hence the required statements follows.
	
	(\ref{item_cagc2})
The additional assumption is equivalent to the condition that $h$ is positive definite (see Section \ref{subsec_intrinsic}). In this case, the eigenvalues of $S$ are all positive if and only if $h'$ is positive definite as well. But the positive definiteness of $h'$ is equivalent to the condition that $N(\Sigma)$ is a locally convex centro-affine hypersurface with the position vector of each point pointing towards the convex side (or equivalently, $0\in\V$ lies on the concave side). When $N(\Sigma)$ is a proper affine sphere, this condition means exactly that $N(\Sigma)$ is actually a hyperbolic affine sphere. The first statement then follows from Part (\ref{item_cagc1}). 
\end{proof}
With the same proof, one can show a similar statement as Part (\ref{item_cagc2}) with ``hyperbolic'' and ``positive'' replaced by ``elliptic'' and ``negative'', respectively. But  we are mainly interested in the situation where $N(\Sigma)$ is not merely a hyperbolic affine sphere, but also part of a \emph{complete} one discussed in Section \ref{subsec_completehyperbolic}: 
\begin{definition}[\textbf{Affine $(C,k)$-hypersurfaces}]\label{def_ck}
	Let $C\subset\V$ be a proper convex cone and let $k>0$. An \emph{affine $(C,k)$-hypersurface} in $\A$ is a locally strongly convex smooth hypersurface $\Sigma$ with CAGC $k$ such that the affine normal mapping $N:\Sigma\rightarrow\V$ has image in a complete hyperbolic affine sphere generating $C$.
\end{definition}
Given an affine $(C,k)$-hypersurface $\Sigma\subset\A$, by Propositions \ref{prop_scaling} and \ref{prop_cagc}, the affine normal mapping $N:\Sigma\to\V$ and conormal mapping $N^*:\Sigma\to\V^*$ have images in the following specific scaling of the Cheng-Yau affine spheres $\Sigma_C$ and $\Sigma_{C^*}$ (see Sections \ref{subsec_completehyperbolic} and \ref{subsec_conormal}), respectively:
$$
N(\Sigma)\subset k^\frac{1}{2(n+1)}\Sigma_C,\quad N^*(\Sigma)\subset k^{-\frac{1}{2(n+1)}}\Sigma_{C^*}.
$$
The former is the complete hyperbolic affine sphere generating $C$ with affine shape operator $k^{-\frac{1}{n+2}}\id$ by Proposition \ref{prop_scaling}, and the latter is dual to the former.
Also note that $\Sigma$ is $C$-convex because by Lemma \ref{lemma_equiaffine} and the fact that $N$ points towards the convex side of $\Sigma$, the supporting half-space of $\Sigma$ at a point $p$ coincides with that of $N(\Sigma)$ at $N(p)$ up to translation, while $k^\frac{1}{2(n+1)}\Sigma_C$ and hence $N(\Sigma)$ are $C$-convex.

\subsection{The Minkowski case}\label{subsec_minkowski}
We view the Minkowski space $\R^{n,1}$ (see Example \ref{example_hyperboloid} for the basic definitions) both as a vector space endowed with a bilinear form $\bii{\cdot}{\cdot}$ of signature $(n,1)$ (hence also endowed with the resulting volume form) and as the affine space modeled on this vector space, which is a flat Lorentzian manifold.

A hypersurface $\Sigma\subset\R^{n,1}$ is said to be \emph{spacelike} if its tangent hyperplanes are spacelike (\cf Section \ref{subsec_cregular}). When $\Sigma$ is $\C^1$, it is the case if and only if the ambient Lorentzian metric restricts to a Riemannian metric $g$ on $\Sigma$, which is called the \emph{first fundamental form} of $\Sigma$. Note that $C_0$-convex hypersurfaces (see Definition \ref{def_cconvex}) are in particular spacelike.

From an affine-differential-geometric point of view, there is a natural equi-affine transversal vector field to be considered for a spacelike hypersurfaces $\Sigma$: the unit normal field, taking values in the hyperboloid $\mathbb{H}$ (see  Example \ref{example_hyperboloid}).  The resulting intrinsic data $(\nu, \nabla, h, S)$ have the following relations with the first fundamental form $g$:
\begin{itemize}
	\item The induced volume form $\nu$ and induced affine connection $\nabla$ are the volume form and Levi-Civita connection of $g$, respectively.
	\item The affine metric $h$, known as the \emph{second fundamental form} of $\Sigma$, is related to the shape operator $S$ through the relation $h(\cdot,\cdot)=g(\cdot,S(\cdot))$.
\end{itemize}
We call the determinant $\det(S):\Sigma\to\R$ the \emph{classical} Gaussian curvature of $\Sigma$ in order to distinguish with the affine Gaussian curvature. Note that when $n=2$, the \emph{Gauss equation} for $\Sigma$, similar to the one for surfaces in the Euclidean $3$-space, states that the classical Gaussian curvature is opposite to the intrinsic curvature of the first fundamental form. 

As the main result of this section, we characterize affine $(C_0,k)$-hypersurface by classical Gaussian curvature:
\begin{proposition}\label{prop_classicalcurvature}
	Given $k>0$, affine $(C_0,k)$-hypersurfaces in the Minkowski space $\R^{n,1}$ are exactly $C_0$-convex hypersurfaces with constant classical Gaussian curvature $k^\frac{n+2}{2n+2}$.
\end{proposition}
\begin{proof}
	Let $\Sigma\subset\R^{n,1}$ be an affine $(C_0,k)$-hypersurface and $N:\Sigma\to\R^{n,1}$ be its affine normal mapping. Then $\Sigma$ is $C_0$-convex and $N$ has image in the $k^{\frac{1}{2(n+1)}}$-scaling of  $\Sigma_{C_0}=\mathbb{H}$ (see Section \ref{subsec_cagc}). Thus, $N':=k^{-\frac{1}{2(n+1)}}N$ has images in $\mathbb{H}$. By Lemma \ref{lemma_equiaffine}, $N'$ is a local diffeomorphism from $\Sigma$ to $\mathbb{H}$ and for every $p\in\Sigma$, the tangent space $\T_p\Sigma$ coincides with $\dif N'_p(\T_p\Sigma)=\T_{N'(p)}\mathbb{H}$ after translation. But the orthogonal complement of the vector $N'(p)$ is exactly the subspace $\T_{N'(p)}\mathbb{H}\subset\T_p\R^{n,1}\cong\R^{n,1}$. Therefore, $N'$ is the unit orthogonal normal field. The shape operator $S'$ of $\Sigma$ with respect to $N'$ is $S'=k^{-\frac{1}{2(n+1)}}S$, hence the classical Gaussian curvature is
	$\det(S')=k^{-\frac{n}{2(n+1)}}\det(S)=k^\frac{n+2}{2n+2}$, as required.
	
	Conversely, let $\Sigma\subset\R^{n,1}$ be a $C_0$-convex hypersurface with first fundamental form $g$ and unit orthogonal normal field $N':\Sigma\to\mathbb{H}$, and let $(\nu',\nabla',h',S')$ be the intrinsic data of $(\Sigma,N')$. Assume the classical Gaussian curvature $\det(S')$ is a constant $k'>0$. Since $\nu'=\dif\vol_g$ and $h'(\cdot,\cdot)=g(\cdot,S'(\cdot))$, the volume form of $h'$ is 
	$$
	\dif\vol_{h'}=k'^{\frac{1}{2}}\dif\vol_g=k'^{\frac{1}{2}}\nu'.
	$$
	By Lemma \ref{lemma_scale}, $N=k'^{\frac{1}{n+2}}N'$ is the affine normal field of $\Sigma$, hence the affine Gaussian curvature is $\det(S)=\det(k'^{\frac{1}{n+2}}S')=k'^{\frac{2n+2}{n+2}}=:k$. Thus, $\Sigma$ is an affine $(C_0,k)$-hypersurface, as required.
\end{proof}
\begin{remark}
	One can consider affine-differential-geometric properties of general (not necessarily locally convex) spacelike hypersurfaces $\Sigma\subset\R^{n,1}$, or hypersurfaces in the Euclidean space $\mathbb{E}^{n+1}$.
	In particular, refining the argument in the second part of the above proof, it can be shown that for such a $\Sigma$ in $\R^{n,1}$ or $\mathbb{E}^{n+1}$, a unit normal vector field $N$ (with values in $\mathbb{H}$ and the unit sphere $\mathbb{S}$ in the Minkowski and Euclidean cases, respectively)	
	scaled by some function $a:\Sigma\to\R_+$ gives an affine normal field if and only if the classical Gaussian curvature of $\Sigma$ is a nonzero constant, and in this case $a$ must be a constant. Therefore, if $\Sigma$ has constant classical Gaussian curvature then it also has CAGC.	
\end{remark}

\section{Preliminaries on convex analysis}\label{sec_convex_prel}
In this section, we introduce definitions, notations and results from convex analysis that we will use in the following sections. 
\subsection{Convex functions}\label{subsec_convexfunction}
We consider lower semicontinuous convex functions on $\R^n$ with values in $\Rp$ and denote the space of them, excluding the constant function $+\infty$, by
$$
\LC(\R^n):=\{u:\R^n\rightarrow\Rp\mid \mbox{$u$ is lower semicontinuous and convex, }u\not\equiv+\infty\}.
$$
Also, it is sometimes convenient to consider the space
$$
\GLC(\R^n):=\LC(\R^n)\cup\{\pm\infty\},
$$
where $+\infty$ and $-\infty$ are understood as constant functions. It has the property that the pointwise supremum of any family of functions $F\subset\GLC(\R^n)$ is still in $\GLC(\R^n)$ (the supremum of the empty set is $-\infty$ by convention). Members in $\GLC(\R^n)$ and $\LC(\R^n)$ are called ``closed convex functions'' and ``closed proper convex functions'', respectively, in literatures on convex analysis such as \cite{rockafellar}.  Here ``closed'' refers to the closedness of the epigraph
$$
\epi{u}:=\{(x,\xi)\in\R^n\times\R\mid \xi\geq u(x)\}.
$$
In fact, $u$ is convex/lower semicontinuous if and only if $\epi{u}$ is convex/closed.

Given $u\in\LC(\R^n)$, the \emph{effective domain}
$$
\dom{u}:=\{x\in\R^{n}\mid u(x)<+\infty\}
$$ 
is a nonempty convex subset of $\R^n$. It is a basic fact that any $\R$-valued convex function on an open subset of $\R^n$ is continuous (actually Lipschitz, see \cite[Lemma 1.1.6]{gutierrez}), so $u$ is continuous in the interior $U:=\interior\dom{u}$ of its effective domain. 
At a boundary point $x_0\in\pa U$, $u$ is continuous at least along line segments in $U$ in the sense that
$$
\lim_{t\to0^+}u((1-t)x_0+tx_1)=u(x_0)\,\text{ for any }x_0\in\pa U,x_1\in U.
$$
Indeed, by lower semicontinuity we have $\liminf_{t\to0^+}u((1-t)x_0+tx_1)\geq u(x_0)$, while by convexity we have $u((1-t)x_0+tx_1)\leq (1-t)u(x_0)+tu(x_1)$, hence $\limsup_{t\to0^+}u((1-t)x_0+tx_1)\leq u(x_0)$. Note that $u(x_0)=+\infty$ is allowed here. More generally, a similar argument shows that the restriction of $u$ to every simplex in $\dom{u}$ is continuous (see \cite[Theorem 10.2]{rockafellar}).


By virtue of these continuity properties, if $U$ is nonempty, then $u$ is determined by its restriction to $U$:
\begin{proposition}\label{prop_extension}
	Given a  convex domain $U\subset\R^n$ and a convex function $u:U\rightarrow\R$, the function $\env{u}:\R^n\to\Rp$ defined by
	$$
	\env{u}(x):=
	\begin{cases}
	u(x)&\text{ if }x\in U\\
	\liminf_{y\to x}u(y)&\text{ if }x\in\pa U\\
	+\infty&\text{ if }x\in\R^n\setminus\overline{U}
	\end{cases}
	$$
   is the unique element in $\LC(\R^n)$ extending $u$ with effective domain contained in $\overline{U}$.
\end{proposition}
The extension $\env{u}$ is a particular instance of \emph{convex envelopes}, introduced in the next section for more general $u$.
\begin{proof}
It is elementary to check that the epigraph $\epi{\env{u}}$ is exactly the closure of $\epi{u}\subset U\times\R$ in $\R^n\times\R$. Since  $\epi{u}$ is convex,  $\epi{\env{u}}$ is a closed convex set, hence $\env{u}\in\LC(\R^n)$. To prove the uniqueness, it is sufficient to show that given $u\in\LC(\R^n)$ with $U:=\interior\dom{u}$ nonempty, the value of $u$ at any $x_0\in\pa U$ coincides with the liminf of $u(x)$ as $x\in U$ tends to $x_0$. This is given by
$$
\liminf_{x\in U,\,x\to x_0} u(x)\geq u(x_0)=\lim_{t\to0^+}u((1-t)x_0+tx_1)\geq\liminf_{x\in U,\,x\to x_0}u(x),
$$
where the first inequality is because $u$ is lower semicontinuous.
\end{proof}
For general $u\in\LC(\R^n)$ with $\interior\dom{u}$ possibly empty, there is a unique affine subspace $A\subset\R^n$ containing $\dom{u}$ such that $\dom{u}$ has nonempty interior as a subset of $A$, and one can study the restriction $u|_A\in\LC(A)$ instead. The interior of $\dom{u}$ as a subset of $A$ is called the \emph{relative interior} and denoted by $\ri \dom{u}$ (see \cite[Section 6]{rockafellar}). 

Proposition \ref{prop_extension} assigns to every convex function $u:U\to\R$ its \emph{boundary values}, namely, the restriction $\env{u}|_{\pa U}$. This is fundamental in the study of Monge-Amp\`ere equations on $U$ because the Dirichlet problem of such equations asks for a convex function with prescribed boundary values. A basic fact is that $u$ extends continuously to $\overline{U}$ if and only if its boundary values are continuous. This follows from:
\begin{proposition}\label{prop_boundarycontinuous}
Let $u\in\LC(\R^n)$ with $U:=\interior\dom{u}$ nonempty and let $x_0\in\pa U$ be such that $u(x_0)<+\infty$ and the restriction of $u$ to $\pa U$ is continuous at $x_0$. Then the restriction of $u$ to $\overline{U}$ is continuous at $x_0$ as well.
\end{proposition}
\begin{proof}
Suppose $u|_{\overline{U}}$ is not continuous at $x_0$ and pick a point $y_0\in U$.
Adding an affine function $\R^n\to\R$ to $u$ does not affect the statement, so we may suppose without loss of generality that $u(x_0)=u(y_0)=0$. Since $u$ is lower semicontinuous, $u|_{\overline{U}}$ being discontinuous at $x_0$ means there is $\eps>0$ and a sequence $(y_i)_{i\geq1}$ in $\overline{U}$ converging to $x_0$ such that $u(y_i)\geq\eps$ for every $i$. Let $(x_i)$ be the sequence on $\pa U$ such that $y_i$ lies on the line segment joining $y_0$ and $x_i$. Since $u$ is convex on that segment with $u(y_0)=0$ and $u(y_i)\geq0$, we have $u(x_i)\geq u(y_i)\geq\eps$. Since $x_i$ converges to $x_0$ from $\pa U$, this shows that $u|_{\pa U}$ is discontinuous at $x_0$.
\end{proof}

\subsection{Convex envelope}\label{subsec_envelope}
An \emph{affine function} on $\R^n$ is a function of the form 
$$
a:\R^n\to\R,\ a(x)=x\cdot y+\eta,
$$ 
where $\eta\in\R$ and $x\cdot y:=x^1y^1+\cdots+ x^ny^n$ is the standard inner product ($x^i$ denotes the $i^\mathrm{th}$ coordinate of $x$). We call $y$ the \emph{linear part} of $a$. 

Clearly, affine functions belong to $\LC(\R^n)$, hence a function given by pointwise supremum of a set of affine functions is in $\GLC(\R^n)$. We can thus introduce:
\begin{definition}[\textbf{Convex envelop}]\label{def_env}
	Given  a subset $E$ of $\R^n\times\R$,
	the \emph{convex envelope} of $E$ is the function in $\GLC(\R^n)$ given by the pointwise supremum of all affine functions with epigraphs containing $E$. Given a set $S\subset\R^n$ and a function $\phi:S\to\Rpm$, the convex envelope of $\phi$ is defined as the convex envelope of the epigraph of $\phi$ and is denoted by $\env{\phi}$. Equivalently,
	$$
	\env{\phi}(x):=\sup\{a(x)\mid \text{ $a:\R^n\to\R$ is an affine function with $a|_S\leq \phi$}\}
	$$
\end{definition}
This is related to the well known notion of \emph{convex hull} of a set $E\subset\R^d$, namely the intersection of all convex subsets of $\R^d$ containing $E$, which we denote by $\chull{E}$. We have the following important characterizations of $\chull{E}$ and its closure:
\begin{itemize}
	\item (See \cite[Theorem 17.1]{rockafellar}) A point $x\in\R^d$ is in $\chull{E}$ if and only if $x$ is a convex combinations of $d+1$ points in $E$, \ie 
	$x=t_0x_0+\cdots+t_dx_d$ for some $x_i\in E$ and $t_i\in[0,1]$ with $t_0+\cdots+t_d=1$. 
	\item (See \cite[Theorem 11.5]{rockafellar}) The closure $\overline{\chull{E}}$ equals the intersection of all closed half-spaces of $\R^d$ containing $E$.
\end{itemize}
Using these facts, we can show:
\begin{proposition}\label{prop_envelop}
Let $S$ be a subset of $\R^n$ and $\phi:S\to\Rp$ be a function. 
\begin{enumerate}
	\item\label{item_envelop1} If the convex envelope $\env{\phi}$ is not constantly $-\infty$, then 
	its epigraph $\epi{\env{\phi}}$ is the closure of the convex hull of $\epi{\phi}\subset S\times\R$ in $\R^n\times\R$.
	\item\label{item_envelop2} If $S$ is compact and $\phi$ is lower semicontinuous, then the convex hull of $\epi{\phi}$ is closed, hence equals $\epi{\env{\phi}}$.
\end{enumerate}
\end{proposition}
The assumption $\env{\phi}\not\equiv-\infty$ in Part (\ref{item_envelop1}) means there exists an affine function majorized by $\phi$ on $S$, which is clearly true under the assumption of Part (\ref{item_envelop2}). Also note that the assumption of Part (\ref{item_envelop2}) implies $\epi{\phi}$ is a closed subset of $\R^{n+1}$, but the convex hull of a general unbounded closed set is not necessarily closed.

%
\begin{proof}
(\ref{item_envelop1}) Given a closed half-space $H\subset\R^n\times\R\cong\R^{n+1}$, let us call  $H$ \emph{vertical} if it contains a vertical line $\{x\}\times\R$, and call $H$ an \emph{upper half-space} if $H$ is not vertical and contains a vertical upper half-line $\{x\}\times[\xi,+\infty)$. Thus, upper half-spaces are exactly epigraphs of affine functions. 

If $\phi\equiv+\infty$, then $\epi{\phi}$ and $\epi{\env{\phi}}$ are empty and the required conclusion holds. Otherwise, $\epi{\phi}$ contains a vertical upper half-line, hence every half-space containing $\epi{\phi}$ is either vertical or an upper half-space, and there is at least one upper half-space $H_0$ containing $\epi{\phi}$. For any vertical half-space $H$, the intersection $H\cap H_0$ equals the intersection of all upper half-spaces containing $H\cap H_0$. Therefore, the intersection of all closed half-spaces containing $\epi{\phi}$ coincides with the intersection of the upper half-spaces among them. This proves the required statement because the former intersection is the closure of $\chull{\epi{\phi}}$ and the latter is $\epi{\env{\phi}}$ by definition of $\env{\phi}$.

(\ref{item_envelop2})  Let $(x_i,\xi_i)$ be a sequence in $\chull{\epi{\phi}}$ converging to $(x_\infty,\xi_\infty)\in\R^n\times\R$. We need to show that $(x_\infty,\xi_\infty)$ belongs to $\chull{\epi{\phi}}$. By the first bullet point above, there are $t_i^{(k)}\in[0,1]$ and $(x_i^{(k)},\xi_i^{(k)})\in \epi{\phi}$ ($k=1,\cdots,n+2$) such that 
$$
\sum_{k=1}^{n+2}t_i^{(k)}=1,\quad (x_i,\xi_i)=\sum_{k=1}^{n+2}t_i^{(k)}\cdot (x_i^{(k)},\xi_i^{(k)})
$$ 
for every $i$. By restricting to a subsequence, we may assume $x_i^{(k)}\to x_\infty^{(k)}\in S$ and $t_i^{(k)}\to t_\infty^{(k)}\in[0,1]$ (with $\sum_kt_\infty^{(k)}=1$), so that $x_\infty=\sum_kt_\infty^{(k)}x_\infty^{(k)}$ and 
\begin{equation}\label{eqn_proofenvelope}
\xi_\infty=\lim_{i\to\infty}\sum_kt_i^{(k)}\xi_i^{(k)}\geq \liminf_{i\to\infty}\sum_kt_i^{(k)}\phi(x_i^{(k)})\geq\sum_k t_\infty^{(k)}\phi(x_\infty^{(k)}),
\end{equation}
where the second inequality follows from super-additivity of liminf and the inequality $\liminf\phi(x_i^{(k)})\geq\phi(x_\infty^{(k)})$ is implied by the lower semicontinuity of $\phi$. 

Since $\epi{\phi}+(0,t)\subset\epi{\phi}$ for all $t\geq 0$, we have 
$$
\chull{\epi{\phi}}+(0,t)=\chull{\epi{\phi}+(0,t)}\subset\chull{\epi{\phi}}.
$$
This means if $(x,\xi)$ belongs to $\chull{\epi{\phi}}$ then so does $(x,\xi+t)$. Therefore, in view of  (\ref{eqn_proofenvelope}) and the fact that $\left(x_\infty,\sum_k t_\infty^{(k)}\phi(x_\infty^{(k)})\right)\in\chull{\epi{\phi}}$, we have $(x_\infty,\xi_\infty)\in\chull{\epi{\phi}}$, as required.
\end{proof}

As a consequence, $\env{\phi}$ is the maximum element in $\GLC(\R^n)$ majorized by $\phi$:
\begin{corollary}\label{coro_supremum}
For any subset $S\subset\R^n$ and function $\phi: S\to\Rpm$, the convex envelope $\env{\phi}$ is the pointwise supremum of all $u\in\GLC(\R^n)$ with $u|_S\leq \phi$. In particular, the convex envelope of $u\in\GLC(\R^n)$ is $u$ itself.
\end{corollary}
\begin{proof}
We prove the second statement first. The statement is trivial when $u\equiv\pm\infty$, so we may suppose $u\in\LC(\R^n)$. By convexity and lower semicontinuity, $\epi{u}$ is a closed convex subset of $\R^{n+1}$ contained in some upper half-space (see the proof of Proposition \ref{prop_envelop}), so Proposition \ref{prop_envelop} (\ref{item_envelop1}) implies $\epi{\env{u}}=\epi{u}$, hence $\env{u}=u$. To prove the first statement, we let $\Phi$ be the set of all $u\in\GLC(\R^n)$ with $u|_S\leq\phi$ and $\mathcal{A}\subset\LC(\R^n)$ be the set of all affine functions. Then for all $x\in\R^n$ we have $\sup_{u\in\Phi}u(x)\geq \sup_{a\in\mathcal{A}\cap\Phi}a(x)=:\env{\phi}(x)$, and conversely
$$
\sup_{u\in\Phi}u(x)=\sup_{u\in\Phi}\sup_{a\in\mathcal{A},a\leq u}a(x)\leq \sup_{a\in\mathcal{A}\cap\Phi}a(x)=\env{\phi}(x),
$$
where the first equality is given by the statement just proved, and the inequality is because $\sup_{a\in\mathcal{A},a\leq u}a(x)\leq\sup_{a\in\mathcal{A}\cap\Phi}a(x)$ for every $u\in\Phi$. This proves the first statement, namely $\env{\phi}(x)=\sup_{u\in\Phi}u(x)$.

\end{proof}

\subsection{Restrictions to the boundary of a convex domain}\label{subsec_adm}
Given $S\subset\R^n$, let $\LC(S)$ denote the space of functions $S\rightarrow\Rp$ restricted from functions in $\LC(\R^n)$:
$$
\LC(S):=\{u|_S\  \mid u\in\LC(\R^n)\}.
$$ 
For later use, we shall determine $\LC(\pa\Omega)$ for the boundary $\pa\Omega$ of a bounded convex domain $\Omega\subset\R^n$. This will be a consequence of:
\begin{lemma}\label{lemma_adm}
Let $\Omega\subset\R^n$ be a bounded convex domain. If $\phi:\pa\Omega\rightarrow\Rp$ is lower semicontinuous and restricts to a convex function on every line segment in $\pa\Omega$, then $\phi$ coincides with its convex envelope $\env{\phi}$ (see Definition \ref{def_env}) on $\pa\Omega$. 
\end{lemma}

\begin{proof}
	We extend $\phi$ to a lower semicontinuous function on the whole $\R^n$ by setting $\phi=+\infty$ on $\R^n\setminus\pa\Omega$. By definition of $\env{\phi}$, it is sufficient to show, for every $x_0\in\pa\Omega$:
	\begin{itemize}
		\item
		if $\phi(x_0)\in\R$, then for any $\eps>0$ there is an affine function $a:\R^n\rightarrow \R$ such that $a\leq\phi$ and $a(x_0)>\phi(x_0)-\eps$;
		\item
		if $\phi(x_0)=+\infty$, then for any $M>0$ there is an affine function $a:\R^n\rightarrow \R$  such that  $a\leq\phi$ and $a(x_0)>M$.
	\end{itemize}
We only give below a proof of the first statement since the second one is similar.

Suppose without loss of generality that $x_0$ is the origin and $\Omega$ is contained in the half-space $\R^n_+:=\{x\in\R^n\mid x^n>0\}$, where $x^i$ denotes the $i^\mathrm{th}$ coordinate of $x$. 
The assumptions imply that the restriction $\phi|_{\R^{n-1}}$ of $\phi$ to the subspace $\R^{n-1}:=\pa\R^n_+$ is a lower semicontinuous convex function. By Corollary \ref{coro_supremum},  there is an affine function $\hat a$ on $\R^{n-1}$ majorized by $\phi|_{\R^{n-1}}$ such that 
\begin{equation}\label{eqn_propadm1}
\hat a(0)\geq\phi(0)-\eps/2.
\end{equation}

Denote $\hat{x}:=(x^1,\cdots x^{n-1})$ for every $x\in\R^n$. The lower semicontinuous function $x\mapsto\phi(x)-\hat{a}(\hat{x})$ takes nonnegative values on the compact set $\R^{n-1}\cap\pa\Omega$, so there is $\delta>0$ such that 
\begin{equation}\label{eqn_propadm2}
\phi(x)-\hat{a}(\hat{x})>-\eps/2\ \mbox{ if } x^n<\delta.
\end{equation}
On the other hand, the function is constantly $+\infty$ outside of the compact set $\pa\Omega$, so there is $\Lambda>0$ such that
\begin{equation}\label{eqn_propadm3}
\phi(x)-\hat{a}(\hat{x})>-\Lambda\ \mbox{ for all } x\in\R^n.
\end{equation}
We can then check that the affine function 
$$
a(x):=\hat a(\hat x)-\frac{\Lambda}{\delta} x^n-\frac{\eps}{2}
$$ 
fulfills the requirements. In fact, (\ref{eqn_propadm1}) implies $a(0)>\phi(0)-\eps$,  (\ref{eqn_propadm2}) and (\ref{eqn_propadm3}) imply $a(x)\leq\phi(x)$ when $0\leq x^n<\delta$ and $x^n\geq\delta$, respectively, while $a(x)<\phi(x)=+\infty$ when $x^n<0$.
\end{proof}

\begin{corollary}\label{coro_adm}
For a bounded convex domain $\Omega\subset\R^n$, $\LC(\pa\Omega)$ consists of all lower semicontinuous functions $\phi:\pa\Omega\rightarrow\Rp$ convex on every line segment in $\pa\Omega$.
\end{corollary}
\begin{proof}
For any $u\in\LC(\R^n)$, the restriction $u|_{\pa\Omega}$ clearly belongs to class of functions described in the statement. Conversely, every function $\phi$ in the class is the restriction of $\env{\phi}\in\LC(\R^n)$ by Lemma \ref{lemma_adm}.
\end{proof}

In the sequel, we often need to consider the effective domain of the convex envelope $\env{\phi}$ for $\phi\in\LC(\pa\Omega)$. It actually coincides with the convex hull of $\dom{\phi}:=\{x\in\pa\Omega\mid \phi(x)<+\infty\}$ in $\R^n$: 
\begin{proposition}\label{prop_domain}
Let $\Omega\subset\R^n$ be a bounded convex domain. Then for any $\phi\in\LC(\pa\Omega)$, we have $\dom{\env{\phi}}=\chull{\dom{\phi}}$.
\end{proposition}
\begin{proof}
By Proposition \ref{prop_envelop} (\ref{item_envelop2}), the epigraph $\epi{\env{\phi}}$ is the convex hull of $\epi{\phi}$, while $\dom{\env{\phi}}$ and $\dom{\phi}$ are the projections of the two epigraphs, respectively, from $\R^n\times\R$ to $\R^n$. It follows from the first bullet point in Section \ref{subsec_envelope} that the operation of taking convex hulls commutes with the projection, hence the required statement.
\end{proof}
We will also need the following result on the piecewise linear structure of $\env{\phi}$:
\begin{lemma}\label{lemma_pleated}
Let $\Omega\subset\R^n$ be a bounded convex domain. Given $\phi\in\LC(\pa\Omega)$ and an affine function $a:\R^n\to\R$ with $a|_{\pa\Omega}\leq\phi$, the set $\{x\in\R^n\mid \env{\phi}(x)=a(x)\}$ is the convex hull of $\{x\in\pa\Omega\mid \phi(x)=a(x)\}$ in $\R^n$. 
\end{lemma}
\begin{proof}
Since $\phi$ is lower continuous, for every $\eps\geq0$,
$$
F_\eps:=\{x\in\pa\Omega\mid \phi(x)-a(x)\leq\eps\}
$$
is a closed subset of $\pa\Omega$. Denoting $E:=\{x\in\R^n\mid \env{\phi}(x)=a(x)\}$, we have $\bigcap_{\eps>0}F_\eps=F_0=E\cap\pa\Omega$ (the last equality follows from Lemma \ref{lemma_adm}). $E$ contains $\chull{F_0}$ since it is convex. So we only need to prove the inclusion $E\subset\chull{F_0}$.

Using the properties of convex hulls in Section \ref{subsec_envelope}, one can show $\chull{F_0}=\bigcap_{\eps>0}\chull{F_\eps}$. Therefore, if the required inclusion does not hold, there exists  $x_0\in E\setminus\chull{F_\eps}$ for some $\eps>0$. We can then find a closed half-space $H\subset\R^n$ containing $F_\eps$ but not $x_0$, and an affine function $b$ on $\R^n$ such that $b\leq 0$ on $H$ and $0<b\leq\eps$ on $\overline{\Omega}\setminus H$. The affine function $\tilde{a}=a+b$ satisfies 
$\tilde{a}|_{\pa\Omega}\leq\phi$, because $\tilde{a}\leq a\leq\phi$ on $\pa\Omega\cap H$ and $\tilde{a}\leq a+\eps<\phi$ on $\pa\Omega\setminus H\subset\pa\Omega\setminus F_\eps$. But on the other hand we have  $\tilde{a}(x_0)=\env{\phi}(x_0)+b(x_0)>\env{\phi}(x_0)$, a contradiction.
\end{proof}
When $n=2$, the lemma implies that the graph of $\env{\phi}$ has the structure of ``pleated surface''. In \cite{bs}, a link between $\LC(\pa\Omega)$ and measured geodesic laminations on $\Omega$ is built based on this structure (for $\Omega$ the unit disk).	
	
\subsection{Subgradients}\label{subsec_subdiff}
We review in this section some useful properties of \emph{subgradients} of convex functions, defined as follows:
\begin{definition}\label{def_subdiff}
A \emph{supporting affine function} of $u\in\LC(\R^n)$ at $x_0\in\R^n$ is an affine function $a:\R^n\to\R$ such that $a\leq u$ on $\R^n$ and $a(x_0)=u(x_0)$. We let $\D u(x_0)\subset\R^n$ denote the set of linear parts (see Section \ref{subsec_envelope}) of all supporting affine functions of $u$ at $x_0$ and call its elements the \emph{subgradients} of $u$ at $x_0$. Denote the set of points where $u$ admits subgradients (or equivalently, supporting affine functions) by
$$
\dom{\D u}:=\{x\in\R^n\mid \D u(x)\neq\emptyset\}.
$$
\end{definition}
Note that the definition can be written in a concise way as
$$
y\in \D u(x_0)\,\overset{\text{def.}}\Longleftrightarrow\,u(x)\geq u(x_0)+(x-x_0)\cdot y\ \text{ for all }x\in\R^n.
$$
We have the following basic existence results:
\begin{lemma}\label{lemma_subdiff}
Suppose $u\in\LC(\R^n)$.
\begin{enumerate}
	\item\label{item_lemmasubdiff1} We have $\ri\dom{u}\subset\dom{\D u}$, \ie $u$ admits a subgradient at every relative interior point (see Section \ref{subsec_convexfunction}) of $\dom{u}$.
	\item\label{item_lemmasubdiff2} $\D u(x)$ has exactly one point if and only if $x\in\interior\dom{u}$ and $u$ is differentiable at $x$. In this case, the subgradient is exactly the gradient of $u$ at $x$.
\end{enumerate}
\end{lemma}
See Theorems 23.4 and 25.1 in \cite{rockafellar} for proofs of Statements (\ref{item_lemmasubdiff1}) and (\ref{item_lemmasubdiff2}), respectively. Note that $u$ is differentiable almost everywhere in the interior of $\dom{u}$ because it is Lipschitz (see \cite[Lemma 1.1.7]{gutierrez}). For a point $x$ where $u$ is differentiable, we also understand $\D u(x)$ as the gradient itself, not distinguishing it with the one-point subset of $\R^n$ formed by the gradient.


Subgradients of a convex function $u$ at a point $x_0$ is closely related to the directional derivatives of $u$ at $x_0$, see \cite[Section 23]{rockafellar} for general discussions. We only need the following notion on finiteness of directional derivatives at a boundary point of essential domain:
\begin{definition}\label{def_inner}
Given $u\in\LC(\R^n)$ with $U:=\interior\dom{u}$ nonempty, $u$ is said to have \emph{finite inner derivatives} at $x_0\in\pa U$ if $u(x_0)<+\infty$ and
\begin{equation}
\label{eqn_innerderivative}
\lim_{t\rightarrow0^+}\frac{u(x_0+t(x_1-x_0))-u(x_0)}{t}>-\infty
\end{equation}
for some $x_1\in U$. Otherwise, $u$ is said to have \emph{infinite inner derivatives} at $x_0$.
\end{definition}
Note that since $t\mapsto u(x_1+t(x_1-x_0))$ is a convex function on $[0,1]$, the difference quotient in (\ref{eqn_innerderivative}) is increasing in $t$, hence the limit exists in $[-\infty,+\infty)$. 
Condition (\ref{eqn_innerderivative}) is actually independent of the choice of $x_1$ and is related to the existence of subgradient at $x_0$ and the boundedness of gradient in $U$:
\begin{proposition}\label{prop_derivative}
Given $u\in\LC(\R^n)$ with $U:=\interior\dom{u}\neq\emptyset$ and a point $x_0\in\pa U$ with $u(x_0)<+\infty$, the following conditions are equivalent to each other:
\begin{enumerate}[label=(\roman*)]
	\item\label{item_innerderiv1} $u$ has finite inner derivatives at $x_0$;
	\item\label{item_innerderiv2} Condition (\ref{eqn_innerderivative}) is satisfied by every $x_1\in U$;
	\item\label{item_innerderiv3} $u$ admits a subgradient at $x_0$;
	\item\label{item_innerderiv4} there exists a sequence $(x_i)$ of points in $U$ tending to $x_0$ such that $u$ is differentiable at $x_i$ and $|\D u(x_i)|$ does not tend to $+\infty$.  
\end{enumerate}
\end{proposition}
\begin{proof}
``\ref{item_innerderiv1}$\Leftrightarrow$\ref{item_innerderiv2}''.  For every vector $v$ in the \emph{tangent cone}  
	 $$
	 C_{x_0}U:=\{v\in \R^n\mid \text{ there exits $t>0$ such that }x_0+tv\in U\},
	 $$ 
	we define the \emph{directional derivative} of $u$ along $v$ as
	$$
	\pa_v u(x_0):=\lim_{t\rightarrow0^+}\frac{u(x_0+tv)-u(x_0)}{t}\in[-\infty,+\infty).
	$$
Then Condition \ref{item_innerderiv1} (resp. \ref{item_innerderiv2}) is equivalent to $\pa_v u(x_0)>-\infty$ for some  (resp. all) $v\in C_{x_0}U$. One can check (see \cite[Theorem 23.1]{rockafellar}) that $v\mapsto\pa_v u(x_0)$ is a homogeneous convex function on $C_{x_0}U$.Therefore, the equivalence  ``\ref{item_innerderiv1}$\Leftrightarrow$\ref{item_innerderiv2}'' follows from the elementary fact that if a $[-\infty,+\infty)$-valued convex function $f$ on a convex domain attends the value $-\infty$ at some point then $f\equiv-\infty$.  
	
	``\ref{item_innerderiv2}$\Leftrightarrow$\ref{item_innerderiv3}''.
 One can check (see \cite[Theorem 23.2]{rockafellar}) that $y\in\R^n$ is a subgradient of $u$ at $x_0$ if and only if 
	\begin{equation}\label{eqn_proofinner}
	\pa_vu(x_0)\geq y\cdot v\ \text{ for all }v\in C_{x_0}U.
	\end{equation}
This implies \ref{item_innerderiv2} by the above construction. Conversely, to prove ``\ref{item_innerderiv2}$\Rightarrow$\ref{item_innerderiv3}'', we take a hyperplane $P\subset \R^n$ disjoint from the origin such that $P$ meets every ray $\R_{\geq 0}\cdot v$ with $v\in C_{x_0}U$. If \ref{item_innerderiv2} holds, then $v\mapsto\pa_vu(x_0)$ restricts to a $\R$-valued convex function on $P\cap C_{x_0}U$. Let $a$ be an affine function on $P$ majorized by the restriction and  $y\in\R^n$ be the vector such that $x\cdot y=a(x)$ for all $x\in P$. We have $\pa_vu(x_0)\geq y\cdot v$ for $v\in P\cap C_{x_0}U$, hence (\ref{eqn_proofinner}) follows by homogeneity and implies \ref{item_innerderiv3}.

``\ref{item_innerderiv3}$\Leftrightarrow$\ref{item_innerderiv4}'' let $\mathcal{S}_u(x_0)\subset\R^n$ denote the set of all limit points of all sequences of the form $(\D u(x_i))$, such that $(x_i)_{i=1,2,\cdots}$ is a sequence in $\dom{u}$ converging to $x_0$ and $u$ is differentiable at $x_i$. Then Condition \ref{item_innerderiv4} just says  $\mathcal{S}_u(x_0)\neq\emptyset$, which is equivalent to \ref{item_innerderiv3} by \cite[Theorem 25.6]{rockafellar}.
\end{proof}

\subsection{Legendre transformation}\label{subsec_legendre1}
A fundamental tool of this paper is the Legendre transform of a convex function, also known as the \emph{conjugate} convex function.
We first recall its definition, which makes sense for non-convex functions:
\begin{definition}\label{def_legendre}
The \emph{Legendre transform} of a function $u:\mathbb{R}^n\rightarrow\mathbb{R}\cup\{\pm\infty\}$ is the function $u^*:\mathbb{R}^n\rightarrow\mathbb{R}\cup\{\pm\infty\}$ defined by
$$
u^*(y):=\sup_{x\in\mathbb{R}^n}(x\cdot y-u(x)).
$$
\end{definition}
Note that $u^*$ belongs to $\GLC(\R^n)$ (see Section \ref{subsec_convexfunction}) because it is a supremum of affine functions. The most important property of Legendre transforms is:
\begin{theorem}\label{thm_involution}
For any function $u:\R^n\to\Rpm$, the repeated Legendre transform $u^{**}$ coincides with the convex envelope $\env{u}$.  As a consequence (see Corollary \ref{coro_supremum}), the Legendre transformation $u\mapsto u^*$ is an involution on $\GLC(\R^n)$.
\end{theorem}
Since the constant functions $+\infty,-\infty\in\GLC(\R^n)$ are Legendre transforms of each other, $u\mapsto u^*$ is also an involution on $\LC(\R^n)$. We give here a proof of the theorem because its ingredients will be used later on.
\begin{proof}
On one hand, the definition of Legendre transform can be rewritten as
\begin{equation}\label{eqn_defleg}
u^*(y):=\sup_{(x,\xi)\in \gra{u}}(x\cdot y-\xi)=\sup_{(x,\xi)\in \epi{u}}(x\cdot y-\xi).
\end{equation}
On the other hand, the definition also implies that $\epi{u^*}$ is the set of all $(y,\eta)\in\R^n\times\R$ such that the affine function $x\mapsto x\cdot y-\eta$ is majorized by $u$: In fact,
$$
(y,\eta)\in \epi{u^*}\Longleftrightarrow \eta\geq x\cdot y-u(x),\forall x\in\R^n\Longleftrightarrow u(x)\geq x\cdot y-\eta,\forall x\in\R^n.
$$

Replacing $u$ with $u^*$ in (\ref{eqn_defleg}), we get $u^{**}(x)=\sup_{(y,\eta)\in\epi{u^*}}(x\cdot y-\eta)$. Using the above interpretation of $\epi{u^*}$, we conclude that $u^{**}$ is the supremum of affine functions majorized by $u$.
\end{proof}

The following lemma describes the Legendre transform of a convex envelope (see Definition \ref{def_env}):
\begin{lemma}\label{lemma_legendre2}
	Given $E\subset\mathbb{R}^n\times\R$, let $u\in\GLC(\R^n)$ be the convex envelope of $E$. Then
	$$
	u^*(y)= \sup_{(x,\xi)\in E}(x\cdot y-\xi).
	$$ 
\end{lemma}
\begin{proof}
Put $\tilde{u}(y):=\sup_{(x,\xi)\in E}(x\cdot y-\xi)$. We showed in the proof of Theorem \ref{thm_involution} that the epigraph of $u^*$ is the set of points $(y,\eta)\in\R^n\times \R$ such that the affine function $x\mapsto x\cdot y-\eta$ is majorized by $u$. On the other hand, essentially the same argument shows that the epigraph of $\tilde{u}$ is the set of $(y,\eta)\in\R^n\times \R$ such that the epigraph of $x\mapsto x\cdot y-\eta$ contains $E$. But by definition of convex envelopes, an affine functions is majorized by $u$ if and only if its epigraph contains $E$. Therefore, $\tilde{u}$ and $u^*$ have the same epigraphs, hence are equal.
\end{proof}
Legendre transformation is related to subgradients through the following proposition, which is basically a reformulation of the definitions:
\begin{proposition}\label{prop_young}
	For  any $u\in\LC(\R^n)$ and $x,y\in\mathbb{R}^n$, we have
	\begin{equation}\label{eqn_young}
	u(x)+u^*(y)\geq x\cdot y.
	\end{equation}
	Moreover, the following conditions are equivalent to each other:
	\begin{enumerate}[label=(\roman*)]
		\item\label{item_young1}
		The equality in (\ref{eqn_young}) is achieved;
		\item\label{item_young2} 
	    $y$ is a subgradient of $u$ at $x$;
		\item\label{item_young3}
		$x$ is a subgradient of $u^*$ at $y$.
	\end{enumerate}
\end{proposition}
An immediate consequence of the equivalence between Conditions \ref{item_young2} and \ref{item_young3} is that the subgradient map of $u$ and that of $u^*$ are inverse to each other:
\begin{corollary}\label{coro_subdiff}
Given $u\in\LC(\R^n)$, the set-valued maps $x\mapsto\D u(x)$ and $y\mapsto\D u^*(y)$ are inverse to each other in the sense that $y\in\D u(x)\Leftrightarrow x\in \D u^*(y)$.
In particular, we have $\D u(\R^n):=\bigcup_{x\in\R^n}\D u(x)=\dom{\D u^*}$.
\end{corollary}

Proposition \ref{prop_young} also gives the values of $u^*$ on $\D u(\R^n)=\dom{\D u^*}$ as
$$
u^*(y)=x\cdot y-u(x),\ \forall x\in \R^n,\, y\in \D u(x).
$$
In particular, if $u$ is differentiable on a set $U\subset\R^n$, then the graph of $u^*$ over $\D u(U)$ is parametrized by $U$ itself through the following map, which is useful in affine differential geometry (see Proposition \ref{prop_computations} below):
\begin{definition}\label{def_legendremap}
	Given a differentiable function $u$ on an open set $U\subset\mathbb{R}^{n}$, the map
	$$
	U\rightarrow\mathbb{R}^n\times\R,\ x\mapsto\big(\D u(x),\,x\cdot \D u(x)-u(x)\big)
	$$
	 is called the \emph{Legendre map} of $u$.
\end{definition}

\section{Legendre transforms of $C$-regular domains and $C$-convex hypersurfaces}\label{sec_legendre}
In this section we prove three fundamental results about $C$-regular domains and $C$-convex hypersurfaces (see Section \ref{subsec_cregular}): the first, Theorem \ref{thm_graph}, identifies them as  strict epigraphs and graphs, respectively, of certain classes of convex functions $\R^n\to\R$, characterized through Legendre transformation; the second result, Theorem \ref{thm_asymp}, gives a necessary and sufficient condition for a complete $C$-convex hypersurface to be asymptotic to the boundary of a $C$-regular domain; finally we give in Theorem \ref{thm_foliation} a necessary and sufficient condition for a family of $C$-convex hypersurfaces generating the same $C$-regular domain to be the level surfaces of a convex function on the domain.

\subsection{The affine section $\Omega$ and the spaces $\S(\Omega)$ and $\S_0(\Omega)$}\label{subsec_analytic}
Given a proper convex cone $C\subset\R^{n+1}$, we write a point in $\mathbb{R}^{n+1}$ as $\tilde{x}=(x,\xi)$ with $x=(x^1,\cdots, x^n)\in\mathbb{R}^n$ and $\xi\in\mathbb{R}$, considering $\xi$ as the ``vertical coordinate'' and always assume the coordinates are so chosen that $C$ has the form
$$
C=\{(t x,t)\mid x\in C_0,\, t>0\}
$$ 
for a bounded convex domain $C_0\subset\mathbb{R}^n$ containing the origin. We identify $\mathbb{R}^{n+1}$ with its dual vector space $\mathbb{R}^{*(n+1)}$ through the standard inner product
$(x,\xi)\cdot(y,\eta):=x\cdot y+\xi\eta$,
so that the \emph{dual cone} $C^*\subset\R^{*(n+1)}$, formed by linear forms on $\mathbb{R}^{n+1}$ taking positive values on $\overline{C}\setminus\{0\}$, can be written as $C^*=\{(tx,t)\mid x\in C^*_0,\,t>0\}$ with
$$
C^*_0:=\{x\in\mathbb{R}^n\mid x\cdot y+1>0,\,\forall y\in\overline{C}_0\}.
$$ 
Note that $C_0^*$ is also a bounded convex domain containing the origin. It is actually the dual of $C_0$ in the sense of Sasaki \cite{sasaki}. The opposite domain
$$
\Omega:=-C_0^*
$$
is important for us due to the following Lemma, which parametrizes the spaces of $C$-null and $C$-spacelike half-spaces in $\R^{n+1}$ by $\pa\Omega\times\R$ and $\Omega\times\R$, respectively:

\begin{lemma}\label{lemma_halfspace}
 Let $C$ and $\Omega$ be as above. Then $C$-null (resp. $C$-spacelike) half-spaces in $\R^{n+1}$ are exactly strict epigraphs of affine functions on $\R^n$ of the form $y\mapsto x\cdot y-\xi$ with $(x,\xi)\in\pa\Omega\times\R$ (resp. $(x,\xi)\in\Omega\times\R$).
\end{lemma}
Here, the \emph{strict} epigraph of a function $u:\R^n\rightarrow\Rpm$ is the set
$$
\sepi{u}:=\{(x,\xi)\in\R^{n+1}\mid \xi>u(x)\}=\epi{u}\setminus\gra{u}.
$$
The proof of the lemma is elementary and we omit it. We remark that $\Omega$ is the section of the opposite dual cone $-C^*$ by the affine hyperplane $\{(x,-1)\mid x\in\R^n\}$. In particular, $\Omega$ can be identified projectively with the convex domain $\mathbb{P}(C^*)$ in $\mathbb{RP}^n$.

As mentioned, a main result of this section, Theorem \ref{thm_graph} below, represents $C$-regular domains and complete $C$-convex hypersurfaces as strict epigraphs and graphs of the Legendre transforms of certain classes of convex functions. Roughly speaking, the function tells us which $C$-null/spacelike half-spaces are used to cut out a given domain/hypersurface. 

The functions corresponding to $C$-convex hypersurfaces are as follows:
\begin{itemize}
    \item Let $\S(\Omega)$ denote the space of all $u\in\LC(\R^n)$ such that $u$ does not admit subgradients at any point outside of $\Omega$ (see Definition \ref{def_subdiff}). Namely,
    $$
    \S(\Omega):=\{u\in\LC(\R^n)\mid \dom{\D u}\subset\Omega\}.
    $$ 
	\item Let $\S_0(\Omega)$ denote the space of all $u\in\LC(\R^n)$ satisfying:
	\begin{itemize}
	\item $U:=\interior\dom{u}$ is nonempty and contained in $\Omega$;
	\item $u$ is smooth and locally strongly convex (see Section \ref{subsec_intrinsic}) in $U$;
	\item $u$ has infinite inner derivatives at every boundary point of $U$.
\end{itemize} 
\end{itemize}
By Proposition \ref{prop_derivative}, the last condition is equivalent to the gradient blowup condition $\lim_{x\to\pa U}|\D u(x)|\to+\infty$, or the condition that $u$ does not admit subgradients at any point of $\pa U$. The latter implies that $\S_0(\Omega)$ is a subset of $\S(\Omega)$. Also note that the effective domain of any $u\in\S(\Omega)$ is contained in $\overline{\Omega}$ because we have $\dom{u}\subset\overline{\dom{\D u}}$ by  Lemma \ref{lemma_subdiff}, so $\S(\Omega)$ can be viewed as a space of $\Rp$-valued functions on $\overline{\Omega}$. 

\subsection{Characterization of $C$-regular domains and $C$-convex hypersurfaces}
Recall from Section \ref{sec_convex_prel} that $\LC(\pa\Omega)$ is the space of lower semicontinuous functions $\phi:\pa\Omega\rightarrow\Rp$ convex on line segments, $\env{\phi}$ denotes the convex envelope of $\phi$ and the asterisk stands for Legendre transformation.
\begin{theorem}\label{thm_graph}
Given $C$ and $\Omega$ as above, the following statements hold.
\begin{enumerate}
	\item\label{item_thmgraph1} The assignment 
	$\phi\mapsto D=\sepi{\env{\phi}^*}$
	gives a bijection from $\LC(\pa\Omega)$ to the space of nonempty $C$-regular domains in $\R^{n+1}$. Moreover, $D$ is proper if and only if $\chull{\dom{\phi}}$ has nonempty interior. 
	\item\label{item_thmgraph2} 
The assignment $u\mapsto\Sigma=\gra{u^*}$ gives a bijection from $\S(\Omega)$ to the space of complete $C$-convex hypersurfaces in $\R^{n+1}$. Moreover, $\Sigma$ is smooth and locally strongly convex if and only if $u\in \S_0(\Omega)$.	
\item\label{item_thmgraph3} 
Let $u\in\S(\Omega)$ and put $\phi:=u|_{\pa\Omega}$. Then $\sepi{\env{\phi}^*}$ is the $C$-regular domain generated by the $C$-convex hypersurface $\gra{u^*}$. 
\end{enumerate}
\end{theorem}
Let us establish some lemmas before giving the proof. We first note that by the following lemma, the Legendre transforms $\env{\phi}^*$ and $u^*$ in the theorem are $\R$-valued on the whole $\R^n$, hence the strictly epigraphs and graphs in question are \emph{entire}.
\begin{lemma}\label{lemma_entire}
Let $u\in\LC(\R^n)$ be such that $\dom{u}$ is bounded and nonempty. Then $\dom{u^*}=\R^n$.
\end{lemma}
\begin{proof}
For any $y\in\R^n$, the lower semicontinuous convex function $x\mapsto u(x)-x\cdot y$ is $+\infty$ outside of the bounded set $\dom{u}$, hence attends its minimum at some $x_0\in\dom{u}$. This means $x\mapsto (x-x_0)\cdot y+u(x_0)$ is a supporting affine function of $u$ at $x_0$, hence $y\in\D u(x_0)$. Therefore, we have $\D u(\R^n)=\R^n$ and it follows from Corollary \ref{coro_subdiff} that $\dom{u^*}=\R^n$.
\end{proof}

For the assertion in Part (\ref{item_thmgraph1}) of the theorem about properness of $D$, we need: 
\begin{lemma}\label{lemma_affineatom}
For $k=0,1,\cdots, n$, let $\mathcal{A}_k$ denote the set of all $u\in\LC(\R^n)$ such that $\dom{u}$ is a $k$-dimensional affine subspace of $\R^n$ and $u$ restricts to an affine function on $\dom{u}$. Then the Legendre transformation $u\mapsto u^*$ sends $\mathcal{A}_k$ to $\mathcal{A}_{n-k}$.
\end{lemma}
Note that a $0$-dimensional affine subspace is by convention a single point.
\begin{proof}
One can check that if $\tilde{u}\in\LC(\R^n)$ is transformed from $u\in\LC(\R^n)$ by applying an affine transformation of $\R^n$ and adding an affine function (namely,
$\tilde{u}(x)=u(g(x))+a(x)$
for an affine transformation $g:\R^n\to\R^n$ and an affine function $a:\R^n\to\R$), then the Legendre transforms $\tilde{u}^*$ and $u^*$ are related to each other by a transformation of the same type. Therefore, the required assertion follows from the fact that every function in $\mathcal{A}_k$ can be transformed into 
$$
u_k(x):=
\begin{cases}
0 &\text{ if }x^{k+1}=\cdots=x^n=0\\
+\infty&\text{ otherwise }
\end{cases}
~,
$$
and that the Legendre transform $u_k^*$ belongs to $\mathcal{A}_{n-k}$ because it has the expression
$$
u_k^*(x)=
\begin{cases}
0 &\text{ if }x^1=\cdots=x^k=0\\
+\infty&\text{ otherwise }
\end{cases}
.
$$
\end{proof}

Finally, the last statement in Part (\ref{item_thmgraph2}) of the theorem relies on the following duality between continuous differentiability of a convex function and  strictly convexity of its Legendre transform: 
\begin{lemma}[{\cite[Theorem 26.3]{rockafellar}}]\label{lemma_strictlyconvex}
Let $u\in\LC(\R^n)$. Then the Legendre transform $u^*$ is strictly convex on every line segment in $\dom{\D u^*}$ if and only if $u$ satisfies:
	\begin{itemize}
		\item $\dom{u}$ has nonempty interior;
		\item $u$ is $\C^1$ in $\interior\dom{u}$;
		\item $u$ does not admit subgradients at any boundary point of $\dom{u}$.
	\end{itemize} 
\end{lemma}
Note that since the set $\dom{\D u^*}$ is not necessarily convex (see \cite[P.253]{rockafellar}), it makes sense to talk about strict convexity of $u^*$ on line segments in $\dom{\D u^*}$ rather than on the entire set.

\begin{proof}[Proof of Theorem \ref{thm_graph}]
(\ref{item_thmgraph1})	
In view of Lemma \ref{lemma_halfspace}, the definition for a subset $D\subset\R^{n+1}$ to be a $C$-regular domain (Definition \ref{def_cregular}) is equivalent to
$$
D=\interior\Big(\bigcap_{(x,\xi)\in E}\epi{y\mapsto x\cdot y-\xi}\Big),\ E\subset\pa\Omega\times\R.
$$
Since the intersection of epigraphs of a family of functions equals the epigraph of their pointwise supremum, the above equality is equivalent to
$$
D=\interior\epi{w},\,\text{ where } w(y)=\sup_{(x,\xi)\in E}(x\cdot y-\xi),\ E\subset\pa\Omega\times\R.
$$
By Lemma \ref{lemma_legendre2}, we can write $w=u^*$ where $u\in\LC(\R^n)$ is the convex envelop of $E$.
If $u\not\equiv\pm\infty$, Lemma \ref{lemma_entire} implies that $w$ is continuous on $\R^n$, hence $D=\interior\epi{w}$ coincides with the strict epigraph $\sepi{w}$. 
In the particular cases $u=-\infty$ and $u=+\infty$ we have $\epi{w}=\sepi{w}=\emptyset$ and $\epi{w}=\sepi{w}=\R^{n+1}$, respectively, hence $\interior\epi{w}=\sepi{w}$ still holds. Therefore, we conclude that $C$-regular domains in $\R^{n+1}$ are exactly subsets of the form $\sepi{u^*}$ with $u\in\LC(\R^n)$ the convex envelope of some subset of $\pa\Omega\times\R$. But one can check from the definitions that such a $u$ coincides with the convex envelope of the restriction $u|_{\pa\Omega}$, which is contained in $\LC(\pa\Omega)$ with the only exception $u\equiv-\infty$, corresponding to $D=\emptyset$. It follows that nonempty $C$-regular domains are exactly subsets of the form $\sepi{\env{\phi}^*}$ with $\phi\in\LC(\pa\Omega)$. This proves the first statement.

To prove the second statement, we note that $D=\sepi{\env{\phi}^*}$ is improper (\ie contains a line in $\R^{n+1}$) if and only if
$$
\env{\phi}^*\leq f \,\text{ for some }f\in\mathcal{A}_1,
$$  
where $\mathcal{A}_k$ is defined in Lemma \ref{lemma_affineatom}. Using the definition and involutive property of Legendre transformation, it can be shown that $u_1\leq u_2\Leftrightarrow u_1^*\geq u_2^*$  for $u_1,u_2\in\LC(\R^n)$. Therefore, by virtue of Lemma \ref{lemma_affineatom}, the above condition is equivalent to
$$
\env{\phi}\geq h\,\text{ for some $h\in\mathcal{A}_{n-1}$}.
$$
This is in turn equivalent to the condition that $\dom{\env{\phi}}$ has empty interior, while Proposition \ref{prop_domain} asserts $\dom{\env{\phi}}=\chull{\dom{\phi}}$. The required statement follows.

(\ref{item_thmgraph2}) $\Sigma\subset\R^{n+1}$ being a complete $C$-convex hypersurface means $\Sigma=\pa G$ for a convex domain $G\subset\R^{n+1}$ whose supporting half-spaces are $C$-spacelike. Since the boundary of a $C$-spacelike half-space cannot be vertical by our coordinate assumptions, we infer that $G=\gra{w}$ for some convex function $w:\R^n\to\R$. The supporting half-spaces of $G$ are exactly the strict epigraphs of supporting affine functions of $w$. By Lemma \ref{lemma_halfspace}, all supporting half-space are $C$-spacelike if and only if all supporting affine function of $w$ has linear part contained in $\Omega$, which just means  $\D w(\R^n)\subset\Omega$ (see Definition \ref{def_subdiff}).
Therefore, complete $C$-convex hypersurfaces are exactly graphs of convex functions $w:\R^n\to\R$ with $\D w(\R^n)\subset\Omega$.  By Corollary \ref{coro_subdiff} and Lemma \ref{lemma_entire}, such $w$'s are exactly the Legendre transforms of functions in $\S(\Omega)$. The first statement follows.

For the second statement, we pick $u\in\S(\Omega)$ and need to show that the Legendre transform $u^*$ is smooth and locally strongly convex if and only if $u\in\S_0(\Omega)$. To this end, we first apply Lemma \ref{lemma_strictlyconvex} (to both $u$ and $u^*$) and conclude that $u^*$ is $\C^1$ and strictly convex if and only if 
\begin{itemize}
	\item $U:=\interior\dom{u}$ is nonempty and contained in $\Omega$;
	\item $u$ is $\C^1$ and strictly convex in $U$;
	\item $u$ does not admit subgradients at any boundary point of $U$.
\end{itemize}
When these conditions hold, $u$ is smooth and locally strongly convex in $U$ if and only if its gradient map $x\mapsto\D u(x)$ is smooth  in $U$ without critical points, and similarly for smoothness and locally strongly convexity of $u^*$ on $\R^n$. Therefore, the second statement follows from the fact that the gradient maps of $u$ and $u^*$ are inverse to each other (see Corollary \ref{coro_subdiff}).

(\ref{item_thmgraph3}) 
	Let $D=\sepi{\env{\psi}^*}$ (where $\psi\in\LC(\pa\Omega)$) be the $C$-regular domain generated by $\Sigma=\gra{u^*}$. In view of Lemma \ref{lemma_halfspace}, the condition that $D$ generates $\Sigma$ means $\env{\psi}^*$ is the pointwise supremum of all affine functions $y\mapsto x\cdot y-\xi$ majorized by $u^*$ with $(x,\xi)\in\pa\Omega\times\R$. 
	
	On the other hand, the epigraph $\epi{u}$ consists of all $(x,\xi)\in\R^{n+1}$ such that $y\mapsto x\cdot y-\xi$ is majorized by $u^*$ 	(see the proof of Theorem \ref{thm_involution}). Therefore, $\env{\psi}^*$ is actually the pointwise supremum of $y\mapsto x\cdot y-\xi$ for $(x,\xi)$ running over $\epi{u|_{\pa\Omega}}=\epi{\phi}$, hence Lemma \ref{lemma_legendre2} implies $\env{\psi}=\env{\phi}$. Therefore, we have $D=\sepi{\env{\phi}^*}$ and the proof is completed.
\end{proof}

Given a $C$-regular domain $D=\sepi{\env{\phi}^*}$, the following proposition essentially says that $D$ admits $C$-null supporting half-spaces at any boundary point $(y,\env{\phi}^*(y))$, and every supporting half-space at that point is obtained from the $C$-null ones:
\begin{proposition}\label{prop_supporting}
		Let $\Omega\subset\R^n$ be a bounded convex domain and let $\phi\in\LC(\pa\Omega)$. Then for every $y\in\R^n$ there exists a closed subset $E_y$ of $\pa\Omega$ such that the set $\D\env{\phi}^*(y)$ of subgradients of $\env{\phi}^*$ at $y$ is the convex hull of $E_y$ in $\R^n$. 
\end{proposition}
\begin{proof}
By Proposition \ref{prop_young} and the definition of subgradients, $x\in\D\env{\phi}^*(y)$ if and only if there exists an affine function $a$ on $\R^n$ with linear part $y$ which supports $\env{\phi}$ at $x$. This $a$ must be the same for all $x\in\D\env{\phi}^*(y)$, because there does not exists two different affine functions with the same linear part both supporting a given convex function. Therefore, we can write $\D\env{\phi}^*(y)=\{x\in\R^n \mid \env{\phi}(x)=a(x)\}$. By Lemma \ref{lemma_pleated}, the latter set is the convex hull of a closed subset of $\pa\Omega$.
\end{proof}
This allows us to show the following basic fact: 
\begin{corollary}\label{coro_contained}
Let $\Sigma$ be a $C$-convex hypersurface and let $D$ be the $C$-regular domain generated by $\Sigma$. Then $\Sigma$ is contained in $D$.
\end{corollary}
\begin{proof}
$\Sigma$ is contained in $\overline{D}$ by definition. If there exists $p\in\Sigma\cap\pa D$, then every supporting half-space of $D$ at $p$ is also a supporting half-space of $\Sigma$. But Proposition \ref{prop_supporting} and Lemma \ref{lemma_halfspace} implies that $D$ has a least one $C$-null supporting half-space at $p$, while supporting half-spaces of $\Sigma$ are all $C$-spacelike, a contradiction. 
\end{proof}

\begin{example}\label{example_graph}
	(1) Every $C$-spacelike hyperplane is a complete $C$-convex hypersurface. It is the graph of Legendre transform of a function which is constantly $+\infty$ except at a single point in $\Omega$. The $C$-regular domain generated by it is the whole $\R^{n+1}$, corresponding to the constant function $+\infty$ on $\pa\Omega$.

(2) The cone $C$ itself is a $C$-regular domain and corresponds to the constant function $0$ on $\pa\Omega$. Complete $C$-convex hypersurfaces generating $C$ are in one-to-one correspondence with continuous convex functions on $\overline{\Omega}$ with zero boundary values (\cf Proposition \ref{prop_boundarycontinuous}) and infinite inner derivatives at every boundary point. 
\end{example}

\subsection{Asymptoticity to the boundary}\label{subsec_asymp}
We give in this section a necessary and sufficient condition for a complete $C$-convex hypersurface $\Sigma$ to be \emph{asymptotic} to the boundary $\pa D$ of a $C$-regular domain $D$, by which we mean the distance from $x\in\Sigma$ to $\pa D$ with respect to an ambient Euclidean metric tends to $0$ as $x$ tends to infinity  in $\Sigma$. This will be deduced from the following general result:
\begin{proposition}\label{prop_asymp}
Let $u_1,u_2\in\LC(\R^n)$ be such that the effective domains $\dom{u_1}$ and $\dom{u_2}$ are bounded subsets of $\R^n$. Then the Legendre transforms $u^*_1$ and $u^*_2$ satisfy
$$
\lim_{|y|\to+\infty}\left(u_1^*(y)-u_2^*(y)\right)=0
$$
if and only if $u_1$ and $u_2$ satisfy the following conditions:
\begin{enumerate}[label=(\roman*)]
\item\label{item_asymp1}
$\dom{u_1}=\dom{u_2}$.
\item\label{item_asymp2}
$u_1=u_2$ on the boundary of $\dom{u_1}$.
\item\label{item_asymp3}
If $U:=\interior{\dom{u_1}}$ is nonempty, then $u_1(x)-u_2(x)\to0$ as $x\in U$ tends to $\pa U$.
\end{enumerate}
\end{proposition}
\begin{remark}\label{remark_asymp}
When $U$ is empty, we have $\dom{u_1}=\pa\dom{u_1}$, so the combination of conditions \ref{item_asymp1} and \ref{item_asymp2} means $u_1=u_2$ on the entire $\R^n$. 
On the other hand, when $U\neq\emptyset$ and  \ref{item_asymp1} holds, Condition \ref{item_asymp3} actually implies \ref{item_asymp2} because for any $x_0\in\pa U$ and $x_1\in U$ we have $u_i(x_0)=\lim_{t\to0^+}u_i(x_0+t(x_1-x_0))$ ($i=1,2$) (see Section \ref{subsec_convexfunction}). 
\end{remark}

Our main tool for the proof of Proposition \ref{prop_asymp} is the following lemma:
\begin{lemma}\label{lemma_u1u2}
Given $u_1,u_2\in\LC(\R^n)$, $y\in\R^n$ and $c\in\R$, we have
\begin{equation}\label{eqn_u1u21}
u_1^*(y)-u_2^*(y)> c
\end{equation}
if and only if there exists an affine function $a$ on $\R^n$ with linear part $y$ (\ie $a(x)=x\cdot y-\eta$ for some $\eta\in\R$) such that
\begin{equation}\label{eqn_u1u22}
\sup_{x\in\R^n}\left(a(x)-u_1(x)\right)> c,\quad a\leq u_2\,\text{ on }\R^n.
\end{equation}
\end{lemma}
Note that in order to rule out the undefined difference $+\infty-(+\infty)$, one should understand (\ref{eqn_u1u21}) as $u^*_1(y)>u^*_2(y)+c$, which implies $u_2^*(y)<+\infty$.
\begin{proof}
We may assume $u_2^*(y)<+\infty$ because it is implied by both conditions. If (\ref{eqn_u1u22}) holds, one can deduce  (\ref{eqn_u1u21}) through
\begin{align*}
u^*_1(y)-u^*_2(y)&=\sup_{x\in\mathbb{R}^n}(x\cdot y-u_1(x))-\sup_{x\in\mathbb{R}^n}(x\cdot y-u_2(x))\\
&=\sup_{x\in\mathbb{R}^n}(a(x)-u_1(x))-\sup_{x\in\mathbb{R}^n}(a(x)-u_2(x))> c
\end{align*}
Conversely, if (\ref{eqn_u1u21}) holds, one easily checks (\ref{eqn_u1u22}) with $a(x):=x\cdot y-u_2^*(y)$.
\end{proof}

The following property of subgradients is also used in the proof:
\begin{lemma}[{\cite[Theorem 24.7]{rockafellar}}]\label{lemma_subdiffproper}
	Given $u\in\LC(\R^n)$ and a nonempty compact subset $E$ of $U:=\interior\dom{u}$, the set of subgradients $\D u(E):=\bigcup_{x\in E}\D u(x)$ is also nonempty and compact.
\end{lemma}

\begin{proof}[Proof of Proposition \ref{prop_asymp}]
To prove the ``if'' part, we pick a sequence  $(y_i)$ in $\R^n$ tending to $\infty$ and need to show that $u^*_1(y_i)-u^*_2(y_i)\to 0$ under the assumptions \ref{item_asymp1}, \ref{item_asymp2} and \ref{item_asymp3}. We may assume $U$ is nonempty, otherwise $u_1$ and $u_2$ are identical by  \ref{item_asymp1} and \ref{item_asymp2} (see Remark \ref{remark_asymp}) and the required conclusion is trivial.

Take a subgradient $x_i\in\D u^*_1(y_i)$ for each $i$. We have $y_i\in\D u_1(x_i)$ by Corollary \ref{coro_subdiff}. It follows that the sequence $(x_i)\subset\dom{\D u_1}\subset\dom{u_1}$ must leave any compact subset of $U$, because otherwise a subsequence of $(y_i)$ would be bounded by Lemma \ref{lemma_subdiffproper}. Therefore, from Conditions \ref{item_asymp2} and \ref{item_asymp3} we get $u_1(x_i)-u_2(x_i)\to0$. Since $u^*_1(y_i)=x_i\cdot y_i-u_1(x_i)$ (see Proposition \ref{prop_young}), we have
\begin{align*}
u^*_1(y_i)-u^*_2(y_i)&=x_i\cdot y_i-u_1(x_i)-\sup_{x\in\mathbb{R}^n}(x\cdot y_i-u_2(x))\\
&\leq \left(x_i\cdot y_i-u_1(x_i)\right)-\left(x_i\cdot y_i-u_2(x_i)\right)=u_2(x_i)-u_1(x_i).
\end{align*}
It follows that $\limsup(u^*_1(y_i)-u^*_2(y_i))\leq\lim(u_2(x_i)-u_1(x_i))=0$. Switching the roles of $u_1$ and $u_2$, we get the required limit  $\lim(u^*_1(y_i)-u^*_2(y_i))=0$.

For the ``only if'' part, we assume $u^*_1(y)-u^*_2(y)\to0$ as $y\to\infty$ in $\R^n$ and first show that $\dom{u_1}$ and $\dom{u_2}$ have the same closure. Suppose by contradiction that it is not the case. Switching the roles of $u_1$ and $u_2$ if necessary, we can find a point $x_0$ in $\dom{u_1}$ but not in the closure $\overline{\dom{u_2}}$. Since $\overline{\dom{u_2}}$ is convex, there is an affine function $b:\R^n\to\R$ such that $b(x_0)>0$ and $b\leq 0$ on $\overline{\dom{u_2}}$. Fix any affine function $a:\R^n\to\R$ with $a\leq u_2$ on $\R^n$. Then for any $t>0$ the affine function $a_t:=a+tb$ still satisfies $a_t\leq u_2$. Moreover, given $\eps>0$, there is $t_1>0$ such that $a_t(x_0)-u_1(x_0)>\eps$ for all $t\geq t_1$. Letting $y_t\in\R^n$ denote the linear part of $a_t$, we deduce from Lemma \ref{lemma_u1u2} that $u^*_1(y_t)-u^*_2(y_t)\geq\eps$ for all $t\geq t_1$. This contradicts the assumption because $y_t\to\infty$ as $t\to+\infty$, hence proves  $\overline{\dom{u_1}}=\overline{\dom{u_2}}$.

Property \ref{item_asymp2} can be proved by contradiction in a similar way: Suppose for example $u_1(x_0)<u_2(x_0)$ at same boundary point $x_0$ of $\overline{\dom{u_1}}=\overline{\dom{u_2}}$ and pick $c\in\R$ such that $u_1(x_0)<c<u_2(x_0)$. By Corollary \ref{coro_supremum}, there exists an affine function $a\leq u_2$ on $\R^n$ with $a(x_0)\geq c$. On the other hand, there is a nontrivial affine function $b$ such that $b(x_0)=0$ and $b\leq 0$ on $\overline{\dom{u_2}}$. For any $t>0$, the affine function $a_t:=a+tb$ satisfies $a_t\leq u_2$ and $a_t(x_0)-u_1(x_0)\geq c-u_1(x_0)=:\eps$. Letting $y_t$ be the linear part of $a_t$, we deduce from Lemma \ref{lemma_u1u2} that $u^*_1(y_t)-u^*_2(y_t)>\eps$, a contradiction.

Since we have shown that $\overline{\dom{u_1}}=\overline{\dom{u_2}}$, Property \ref{item_asymp1} now follows from  \ref{item_asymp2}. Thus, it remains to show \ref{item_asymp3}. Since the roles of $u_1$ and $u_2$ are exchangeable, we only need to show $\limsup(u_2(x_i)-u_1(x_i))\leq 0$ for every sequence $(x_i)$ in $U$ approaching $\pa U$.  Suppose by contradiction that it is not the case. Then there exist $\eps>0$ and a sequence $(x_i)$ in $U$ converging to some $x_0\in\pa U$ such that $u_2(x_i)-u_1(x_i)>\eps$ for every $i$. 
Pick a subgradient $y_i\in\D u_2(x_i)$ for every $i$. By definition of subgradients, the affine function 
$$
a_i(x):=u_2(x_i)+(x-x_i)\cdot y_i
$$ 
is majorized by $u_2$. Since $a_i(x_i)-u_1(x_i)=u_2(x_i)-u_1(x_i)>\eps$, by Lemma \ref{lemma_u1u2} again, we have $u^*_1(y_i)-u^*_2(y_i)>\eps$. 

This contradicts the assumption if the sequence $(y_i)$ has a subsequence tending to $\infty$. Otherwise, $(y_i)$ is bounded and we deduce a contradiction in the following way instead: Pick a nontrivial affine function $b$ vanishing at $x_0$ such that $b\leq 0$ on $\overline{U}$ and $b(x)\geq -\frac{\eps}{2}|x_0-x|$ for all $x\in\R^n$. Then the affine function
$$
\tilde{a}_i:=a_i+\tfrac{1}{|x_0-x_i|}b
$$ 
satisfies $\tilde{a}_i\leq u_2$ on $\R^n$ and $\tilde{a}_i(x_i)-u_1(x_i)=u_2(x_i)-u_1(x_i)+\frac{1}{|x_0-x_i|}b(x_i)\geq \frac{\eps}{2}$. Letting $\tilde{y}_i$ be the linear part of $\tilde{a}_i$ and using Lemma \ref{lemma_u1u2} again, we obtain 
$$
u^*_1(\tilde{y}_i)-u^*_2(\tilde{y}_i)\geq\tfrac{\eps}{2}.
$$ 
But we have $\tilde{y}_i\to\infty$ because the linear part $y_i$ of $a_i$ is bounded while that of $\tfrac{1}{|x_0-x_i|}b$ goes to $\infty$. Therefore, the above inequality contradicts the assumption and concludes the proof.
\end{proof}

As a consequence of Proposition \ref{prop_asymp}, we obtain:
\begin{theorem}\label{thm_asymp}
Let $C$ and $\Omega$ be as in Section \ref{subsec_analytic}. Pick $u\in\S(\Omega)$ and $\phi\in\LC(\pa\Omega)$, so that $\Sigma:=\gra{u^*}$ is a complete $C$-convex hypersurface and $D:=\sepi{\env{\phi}^*}$ a $C$-regular domain. Then $\Sigma$ is asymptotic to $\pa D$ if and only if $u$ and $\phi$ satisfy the following conditions:
\begin{itemize}
	\item  the effective domains $\dom{u}$ and $\dom{\env{\phi}}$ coincide and have nonempty interior $U$;
	\item  $\env{\phi}(x)-u(x)\to 0$ as $x\in U$ tends to $\pa U$.
\end{itemize}
Moreover, these conditions imply that $u=\env{\phi}$ on $\pa U$ and that $u|_{\pa\Omega}=\phi$. In particular, $D$ is the $C$-regular domain generated by $\Sigma$ (see Theorem \ref{thm_graph} (\ref{item_thmgraph3})).

\end{theorem}
\begin{proof}
The condition that $\Sigma$ is asymptotic to $\pa D$ is equivalent to $u^*(y)-\env{\phi}^*(y)\to0$  ($y\to\infty$), which is in turn equivalent to Conditions \ref{item_asymp1} -- \ref{item_asymp3} from Proposition \ref{prop_asymp} for $u_1=u$ and $u_2=\env{\phi}$. 
As mentioned in Remark \ref{remark_asymp}, if $U$ is empty then \ref{item_asymp1} and \ref{item_asymp2} implies that $u$ and $\env{\phi}$ are identical, which is impossible (by Corollary \ref{coro_contained} for example). So \ref{item_asymp1} -- \ref{item_asymp3} altogether are equivalent to the two bullet points. This proves the first statement. The assertion $u=\env{\phi}$ on $\pa U$ follows from \ref{item_asymp1} and \ref{item_asymp2} as said, hence we have $u=\env{\phi}$ on $\R^n\setminus U$ and in particular, $u|_{\pa\Omega}=\phi$.
\end{proof}
Let us point out two remarkable consequences of Theorems \ref{thm_graph} and \ref{thm_asymp}. First, if a $C$-regular domain $D$ is given by a continuous function $\phi\in\LC(\pa\Omega)\cap\C^0(\pa\Omega)$, so that $\env{\phi}$ and $u$ are continuous functions on $\overline{\Omega}$ with the same boundary values (see Proposition \ref{prop_boundarycontinuous}), then every complete $C$-convex hypersurface generating $D$ is asymptotic to $\pa D$. A particular instance is when $D$ is the cone $C$ itself, as discussed in Example \ref{example_graph}. Second, a necessary condition for a $C$-regular $D$ to admit a complete $C$-convex hypersurface asymptotic to its boundary is that $D$ is proper. We give below an example of improper $C$-regular domain with complete $C$-convex hypersurfaces generating it, and examine the corresponding $\phi$ and $u$.

\begin{example}
Given a proper convex cone $C\subset\R^3$, the ``trough domain'' formed by intersecting two $C$-null half-spaces is an improper $C$-regular domain. For the sake of clarity, we consider the cone $C_0=\{(x,y,z)\in\R^3\mid z^2>x^2+y^2, z>0\}$, whose corresponding convex domain $\Omega$ under the setting of Section \ref{subsec_analytic} is the unit disk $\mathbb{D}\subset\R^2$. Let $\phi:\pa\mathbb{D}\to\Rp$ be the function vanishing at $(\pm1,0)$ and taking the value $+\infty$ everywhere else, so that $D=\sepi{\env{\phi}^*}$ is the intersection of the two $C_0$-null half-spaces whose boundaries meet along the $y$-axis.

The hyperboloid $\mathbb{H}:=\{z^2-x^2-y^2=1,z>0\}$ is of course a complete $C_0$-convex surface asymptotic to $\pa C_0$. But for any $\lambda>1$, the $\lambda$-dilation of $\mathbb{H}$ along the $y$-axis, denoted by $\mathbb{H}_\lambda:=\{z^2-x^2-(y/\lambda)^2=1,z>0\}$, generates $D$ instead. See Figure \ref{figure_trough}. At the limit, we obtain the trough surface $\mathbb{H}_\infty=\{z^2-x^2=1, z>0\}$ from the second picture of Figure \ref{figure_trough}, which generates $D$ as well.
 \begin{figure}[h]
		\includegraphics[width=5.1in]{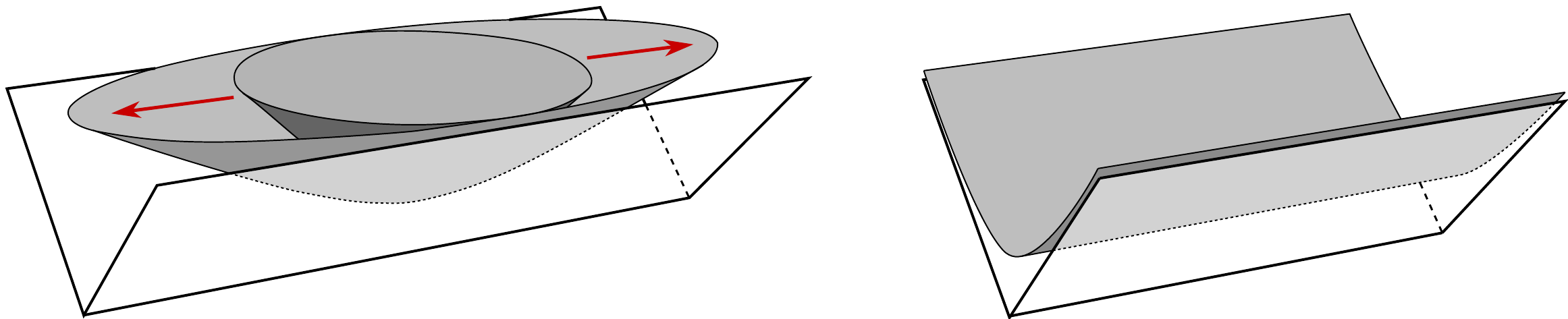}
		\caption{Trough domain and $C$-convex hypersurfaces generating it.}
		\label{figure_trough}
\end{figure}

One can check that $\mathbb{H}$ is the graph of the Legendre transform of
$$
u(x,y):=
\begin{cases}
-\sqrt{1-(x^2+y^2)}&\text{ if }(x,y)\in\overline{\mathbb{D}},\\
+\infty &\text{ otherwise},
\end{cases}
$$
while $\mathbb{H}_\lambda$ and $\mathbb{H}_\infty$ are the graphs of Legendre transforms of the functions $u_\lambda, u_\infty\in\S(\mathbb{D})$ defined by
$$
u_\lambda(x,y):=u(x,\lambda y)\,\text{ for }\lambda\in(1,+\infty),\quad u_\infty(x,y):=
\begin{cases}
u(x,y)&\text{ if }y=0,\\
+\infty &\text{ if }y\neq0.\\
\end{cases}
$$
Thus, $u_\lambda$ and $u_\infty$ restrict to $\env{\phi}$ on $\pa\mathbb{D}$, but they violate the conditions in Theorem \ref{thm_asymp}. Also note that $u_\lambda$ belongs to $\S_0(\mathbb{D})$ but $u_\infty$ does not, which reflects the fact that $\mathbb{H}_\lambda$ is strictly convex and $\mathbb{H}_\infty$ is not (see Theorem \ref{thm_graph} (\ref{item_thmgraph2})).
\end{example}

\subsection{Convex foliations by $C$-convex hypersurfaces}
Given a family of complete $C$-convex hypersurfaces $(\Sigma_t)_{t\in\R}$ in $\R^{n+1}$ generating the same $C$-regular domain $D$, we give in the next theorem a necessary and sufficient condition for $(\Sigma_t)$ to be a \emph{convex} foliation of $D$ in the sense that there is convex function $D\to\R$ for which $\Sigma_t$ is the level hypersurface with value $t$.
\begin{theorem}\label{thm_foliation}
	Let $\phi\in\LC(\pa\Omega)$ and let $(u_t)_{t\in\R}$ be a one-parameter family in $\S(\Omega)$ with $u_t|_{\pa\Omega}=\phi$, so that every $\Sigma_t:=\gra{u_t^*}$ is a complete $C$-convex hypersurface generating the $C$-regular domain $D:=\sepi{\env{\phi}^*}$. Suppose $u_t$ is pointwise nondecreasing in $t$. Then there is a convex function $K:D\to\R$ such that $\Sigma_t=K^{-1}(t)$ for every $t\in\R$  if and only if $(u_t)$ further satisfies the following conditions:
	\begin{enumerate}[label=(\roman*)]
		\item\label{item_foliation1}  for any $x\in\R^n$ and $t_1<t_2$, if $u_{t_1}$ admits a subgradient at $x$, then the strict inequality $u_{t_1}(x)<u_{t_2}(x)$ holds;
		\item\label{item_foliation2} $u_t$ is not uniformly bounded from below and  $\lim_{t\rightarrow+\infty}u_t=\env{\phi}$ pointwise;
		\item\label{item_foliation3} $u_t(x)$ is a concave function in $t$ for every fixed $x\in\R^n$.
	\end{enumerate}
\end{theorem}
In the proof, we make use of the following general result on the Legendre transform of the pointwise infimum of a family of functions:
\begin{lemma}\label{lemma_family}
	Let $F$ be a subset of $\LC(\R^n)$ and $f\in\GLC(\R^n)$ be the convex envelope of the function $x\mapsto\inf_{u\in F}u(x)$. Then the Legendre transform of $f$ is given by
	$$
	f^*(y)=\sup_{u\in F} u^*(y).
	$$
\end{lemma}
\begin{proof}
	Theorem \ref{thm_involution} implies the general fact that any function $u:\R^n\to \Rpm$ and its convex envelop $\env{u}\in\LC(\R^n)$ have the same Legendre transform. Therefore, 
	\begin{align*}
	f^*(y)&=\sup_{x\in\R^n}(x\cdot y-f(x))=\sup_{x\in\R^n}\Big(x\cdot y-\inf_{u\in F}u(x)\Big)=\sup_{x\in\R^n}\sup_{u\in F}(x\cdot y-u(x))\\
	&=\sup_{u\in F}\sup_{x\in\R^n}(x\cdot y-u(x))=\sup_{u\in F}u^*(y).
	\end{align*}
\end{proof}


\begin{proof}[Proof of Theorem \ref{thm_foliation}]
	Given $t_1<t_2$, since $u_{t_1}\leq u_{t_2}$ on $\R^n$, we have $u^*_{t_1}\geq u^*_{t_2}$ by definition of Legendre transformation. Put
	$$
	h:\R^n\times\R\to\R,\quad h(y,t):=u^*_t(y)=\sup_{x\in\R^n}\left(x\cdot y-u_t(x)\right).
	$$
	We shall first show that Conditions \ref{item_foliation1}, \ref{item_foliation2} and \ref{item_foliation3} are equivalent to the following conditions, respectively:
	\begin{enumerate}[label=(\roman**)]
		\item\label{item_foliation1'} $u_t^*(y)$ is strictly decreasing in $t$ for every $y\in\R^n$;
		\item\label{item_foliation2'} $\lim_{t\to-\infty}u_t^*=+\infty$ and $\lim_{t\to+\infty}u_t^*=\env{\phi}^*$ pointwise;
		\item\label{item_foliation3'} $h(y,t)$ is a convex function on $\R^n\times\R$.
	\end{enumerate}
	
	``\ref{item_foliation1} $\Leftrightarrow$ \ref{item_foliation1'}'': Suppose \ref{item_foliation1} fails, namely $u_{t_1}(x)=u_{t_2}(x)=\xi\in\R$ for some $t_1<t_2$ and $x\in\R^n$ such that $u_{t_1}$ admits a subgradient $y$ at $x$. Then $y$ is also a subgradient of $u_{t_2}$ at $x$ and Proposition \ref{prop_young} gives $u_{t_1}^*(y)=u_{t_2}^*(y)=x\cdot y-\xi$, so \ref{item_foliation1'} fails as well.
	The converse is proved in the same way.
	
	``\ref{item_foliation2} $\Leftrightarrow$ \ref{item_foliation2'}'':
	Define $u_\pm\in\GLC(\R^n)$ as follows: first put
	$$
	\tilde{u}_-(x):=\inf_{t\in\R} u_t(x)=\lim_{t\rightarrow-\infty}u_t(x),\quad u_+(x):=\sup_{t\in\R} u_t(x)=\lim_{t\rightarrow+\infty}u_t(x)
	$$
	(so $u_+\in\LC(\R^n)$ but $\tilde{u}_-$ is not necessarily lower semicontinuous), then let $u_-$ be the convex envelope of $\tilde{u}_-$. We claim that 
	the Legendre transforms of $u_\pm$ are 
	$$
	u^*_-(y)=\sup_{t\in\R}u^*_t(y)=\lim_{t\rightarrow-\infty}u^*_t(y),\quad u^*_+(y)=\inf_{t\in\R}u^*_t(y)=\lim_{t\rightarrow+\infty}u^*_t(y).
	$$
	The former follows immediately from Lemma \ref{lemma_family}. For the latter, we note that since $u_t^*:\R^n\to\R$ is decreasing in $t$ and minorized by $\env{\phi}^*$, the pointwise limit $y\mapsto\inf_{t\in\R}u^*_t(y)$ is an $\R$-valued convex function on the entire $\R^n$, hence its convex envelope is itself and Lemma \ref{lemma_family} implies that its Legendre transform is $u_+$. Therefore, using the involutive property of Legendre transformation, we establish the claim. 
	
	With these definitions and expressions of $u_\pm$ and $u_\pm^*$, Conditions \ref{item_foliation2} and  \ref{item_foliation2'} can be reformulated as $u_-=-\infty$, $u_+=\env{\phi}$ and $u_-^*=+\infty$, $u_+^*=\env{\phi}^*$, respectively, hence they are equivalent to each other. 
	
	``\ref{item_foliation3} $\Leftrightarrow$ \ref{item_foliation3'}'': If $u_t(x)$ is concave in $t$ then 
	$(y,t)\mapsto x\cdot y-u_t(x)$ is a convex function on $\R^n\times\R$ for every $x\in\R^n$, hence the pointwise supremum $h(y,t)$
	is convex. Conversely, suppose $h$ is convex.
	It is a general fact (see \cite[Theorem 5.7]{rockafellar}) that if $A:\R^m\to\R^l$ is a linear map and $f:\R^m\to\Rp$ is a convex function, then $y\mapsto \inf_{x\in A^{-1}(y)}f(x)$ is a convex function on $\R^l$. Fixing $x\in\R^n$ and applying this to the projection $A:\R^n\times\R\to\R$ and the function $f(y,t)=h(y,t)-x\cdot y$, we conclude that $$
	\inf_{y\in\R^n}(h(y,t)-x\cdot y)=-\sup_{y\in\R^n}(x\cdot y-h(y,t))=-u^{**}_t(x)=-u_t(x)
	$$ 
	is convex in $t$. Thus, $u_t(x)$ is concave in $t$. This finishes the proof of the equivalences between the conditions \ref{item_foliation1}, \ref{item_foliation2}, \ref{item_foliation3} and \ref{item_foliation1'}, \ref{item_foliation2'}, \ref{item_foliation3'}.
	
	To prove the theorem, it is sufficient to show that \ref{item_foliation1'}, \ref{item_foliation2'} and \ref{item_foliation3'} hold altogether if and only if there is a convex function $K$ on $D:=\sepi{\env{\phi}^*}$ with $K^{-1}(t)=\Sigma_t:=\gra{u_t^*}$ for every $t\in\R$.
	
	If such $K$ exists, then $(\Sigma_t)$ is a foliation of $D$, hence Conditions \ref{item_foliation1'} and \ref{item_foliation2'} hold because $u_t^*$ is non-increasing in $t$ from the beginning. To show  \ref{item_foliation3'}, we note that the defining property $K^{-1}(t)=\gra{u^*_t}$ of $K$ can be rewritten as $h(y,K(y,s))=s$ for all $(y,s)\in\R^{n+1}$. It follows that the epigraph $\epi{K}\subset\R^n\times\R\times\R$ is obtained from the epigraph of $h$ by switching the last two $\R$-components: in fact,
	$$
	(y,s,t)\in\epi{K} \Leftrightarrow t\geq K(y,s)\Leftrightarrow h(y,t)\leq h(y,K(y,s))=s\Leftrightarrow (y,t,s)\in\epi{h},
	$$
	where the second equivalence is because $h(y,t)=u_t^*(y)$ is strictly decreasing in $t$ (Condition \ref{item_foliation2'}). Therefore, the convexity of $K$ implies Condition \ref{item_foliation3'}, namely convexity of $h$. This proves the ``if'' part.
	
	Finally, suppose \ref{item_foliation1'}, \ref{item_foliation2'} and \ref{item_foliation3'} hold. Then $u_t^*(y)$ strictly decreases from $+\infty$ to $\env{\phi}^*(y)$ as $t$ goes from $-\infty$ to $+\infty$ and is convex in $t$. Since convex functions $\R\to\R$ are continuous, $t\mapsto u_t^*(y)$ is a bijection from $\R$ to $(\env{\phi}^*(y),+\infty)$. It follows that $(\Sigma_t)$ is a foliation of $D$, hence we can define $K:D\to\R$ by $K^{-1}(t)=\Sigma_t$. The above relation between $\epi{K}$ and $\epi{h}$ still holds and implies that $K$ is convex as $h$ is. This proves the ``only if'' part and completes the proof of the theorem.
\end{proof}

We conclude this section by some remarks and examples about Theorem \ref{thm_foliation}. Assuming $(u_t)$ satisfies the conditions in the theorem, we first note that $\dom{u_t}$ is actually the same for all $t\in\R$ and contains $\dom{\env{\phi}}$, because $u_t(x)$ is a nondecreasing concave function in $t$ majorized by $\env{\phi}(x)$, and every $\Rp$-valued concave function on $\R$ is either $\R$-valued or constantly $+\infty$. However,  $\dom{u_t}$ does not necessarily coincide with $\dom{\env{\phi}}$. For example, when $\dom{u_t}$ is a single point in $\Omega$, we have $\dom{\env{\phi}}=\emptyset$ (see Example \ref{example_graph} (1)) and the corresponding foliation is that of $D=\R^{n+1}$ by $C$-spacelike hyperplanes parallel to each other. 

On the other hand, if every $\Sigma_t$ is \emph{asymptotic} to $\pa D$ (see Section \ref{subsec_asymp}), we have $\dom{u_t}=\dom{\env{\phi}}$ by Theorem \ref{thm_asymp}. In case, $u_t$ can be viewed as a family of $\R$-valued convex functions on the convex domain $U=\interior\dom{\env{\phi}}$ with the same boundary values (see Section \ref{subsec_convexfunction}). As a simple example, for the cone $C$ itself as a $C$-regular domain, given any $C$-convex hypersurface $\Sigma=\gra{u^*}$ generating it (see Example \ref{example_graph} (2)), all the scalings $\Sigma_t:=e^{-t}\Sigma$ ($t\in\R$) form a convex foliation of $C$ with the corresponding $(u_t)$ given by $u_t=e^{-t}u$, which has zero boundary value. A particular case is when $\Sigma=\Sigma_C$ is the Cheng-Yau affine sphere and $u=w_\Omega$ is the Cheng-Yau support function (see Theorems \ref{thm_chengyau1} and \ref{thm_chengyau2}). The main results of this paper can be viewed as generalizing this case to arbitrary $\phi\in\LC(\pa\Omega)$.

\section{Preliminaries on Monge-Amp\`ere equations}\label{subsec_preliminaries}
We review in this section some basic notions and results from the theory of real Monge-Amp\`ere equations used in the next sections.

Let $\Omega\subset\R^n$ be a convex domain and $u:\Omega\to\R$ be a function. A foundational fact of the theory is that the determinant of hessian $\det\D^2u$ can be defined not only when $u$ is $\C^2$ but also in the generalized sense when $u$ is merely convex (hence Lipschitz but not necessary $\C^1$). In the latter situation, $\det\D^2u:\Omega\to[0,+\infty]$ is by definition the density of the \emph{Monge-Amp\`ere measure} $\MA{u}$ of $u$, defined through subgradients (see Section \ref{subsec_subdiff}) by
	$$
	\int_{E}\det\D^2u\,\dif\Leb:=\MA{u}(E):=\Leb(\D u(E)) \text{ for any Borel subset $E\subset\Omega$}.
	$$
where $\Leb$ denotes the Lebesgue measure. Therefore, if this density equals a prescribed measurable function $g$ on $\Omega$, we call $u$ a \emph{generalized solution} to $\det\D^2u=g$. 

The Monge-Amp\`ere measure has the following stability property:
\begin{lemma}[{\cite[Lemma 2.2]{trudwang}}]\label{convergence of solutions}
	Let $\Omega\subset\R^n$ be a convex domain and  $u_i:\Omega\to\R$ be a sequence of convex functions converging on compact sets to $u_\infty$. Then the Monge-Amp\`ere measures $\MA{u_i}$ converges weakly to the Monge-Amp\`ere measure $\MA{u_\infty}$. 
\end{lemma}
The following super-additive property will also be important:
\begin{lemma}[{\cite[Lemma 2.9]{figalli}}] \label{lemma ma measure of sum}
	Let $\Omega\subset\R^n$ be a convex domain and $u_1,u_2:\Omega\to\R$ be convex functions. Then 
	$$
	\det\D^2(u_1+u_2)\geq \det\D^2 u_1+\det\D^2u_2.
	$$
\end{lemma}
Here and below, inequalities between Borel measure densities such as the above one are understood via evaluation on Borel sets. Note that if $u_1$ and $u_2$ are $\C^2$ then the above inequality follows from the fact that $\det(A+B)\geq\det(A)+\det(B)$ for all positive definite symmetric matrices $A$ and $B$.

Our main tool for estimating solutions of Monge-Amp\`ere equations is: 
\begin{lemma}[Maximum Principle]\label{lemma_maximum}
		Let $U\subset\R^n$ be a bounded convex domain and $u_+, u_-:U\rightarrow\R$ be  convex functions such that $\det \D^2u_+\leq \det \D^2 u_-$ in $U$. Then
		$$
		\inf_{U}(u_+-u_-)=\liminf_{x\rightarrow\pa U}\left(u_+(x)-u_-(x)\right)
		$$
\end{lemma}
Here, $\liminf_{x\rightarrow\pa U}f(x)$ is defined as the supremum of $\inf_{U\setminus E}f$ as $E$ runs over all compact subsets of $U$. A more standard version of the lemma (see \eg \cite[Theorem 1.4.6]{gutierrez}) deals with continuous convex functions $u_\pm\in\C^0(\overline{U})$, for which the conclusion amounts to $\min_{\overline{U}}(u_+-u_-)=\min_{\pa U}(u_+-u_-)$. We need the above simple generalization to cover the case where $u_\pm\rightarrow+\infty$ near $\pa U$ in Section \ref{subsec_existence}. 
\begin{proof}
Put $a:=\inf_U(u_+-u_-)$ and $b:=\liminf_{x\rightarrow\pa U}\left(u_+(x)-u_-(x)\right)$. Assume by contradiction that $a<b$ and fix $c\in (a,b)$. Then there is a compact set $E\subset U$ such that $u_+-u_-\geq c$ on $U\setminus E$. As a consequence, there is $x_0\in E$ such that
$$
a=\min_E(u_+-u_-)=u_+(x_0)-u_-(x_0).
$$
The rest of the proof is standard (see \cite[Theorem 1.4.6]{gutierrez}): fix $c'\in (a,c)$ and pick $\delta>0$ small enough such that
$w:=u_-+\delta|x-x_0|^2+c'$ satisfies 
$$
u_+-w=u_+-u_--\delta|x-x_0|^2-c'\geq c-c'-\delta|x-x_0|^2>0\,\text{ on }U\setminus E.
$$
Then $G:=\{x\in U\mid u_+(x)\leq w(x)\}$ is a compact subset of $E$ and is nonempty because $u_+(x_0)-w(x_0)=a-c'<0$. Since $u_+=w$ on $\pa G$ by continuity, we have $\D w(G)\subset \D u_+(G)$ and hence
$$
\int_G\det\D^2w\,\dif\Leb=\Leb(\D w(G))\leq \Leb(\D u_+(G))=\int_G\det\D^2u_+\,\dif\Leb.
$$
But on the other hand, by Lemma \ref{lemma ma measure of sum}, we have
$$
\det\D^2w\geq \det\D^2u_-+\det\D^2(\delta|x-x_0|^2)\geq\det\D^2u_++(2\delta)^n,
$$
a contradiction.
\end{proof}

Lemma \ref{lemma_maximum} implies the fundamental Comparison Principle for Monge-Amp\`ere equations: if $u_\pm\in\C^0(\overline{U})$ are convex functions with $\det\D^2u_+\leq\det\D^2u_-$ in $U$ and $u_+\geq u_-$ on $\pa U$, then $u_+\geq u_-$ throughout $\overline{U}$. The following generalization, essentially proved in \cite[Proposition 3.11]{bon_smillie_seppi}, allows us to get the same conclusion without comparing $u_+$ and $u_-$ at certain boundary points:

\begin{lemma}[Generalized Comparison Principle]\label{lemma_comparison}
	Let $U\subset\R^n$ be a bounded convex domain,
	$u_+:\overline{U}\rightarrow\Rp$ be a lower semicontinuous convex function taking finite values in $U$ and $u_-:\overline{U}\rightarrow\R$ be a continuous convex function such that
	\begin{itemize}
		\item $\det \D^2u_+\leq \det \D^2 u_-$ in $U$;
		\item $u_+(x)\geq u_-(x)$ for every $x\in\pa U$ where $u_+$ has finite inner derivative and every $x\in\pa U$ where $u_-$ has infinite inner derivative.
	\end{itemize}
Then we have $u_+\geq u_-$ throughout $\overline{U}$.
\end{lemma}
We refer to Definition \ref{def_inner} and Proposition \ref{prop_derivative} for the notion of finiteness of inner derivatives.
\begin{proof}
	Since  $\overline{U}$ and  $\pa U$  are compact and $u_+-u_-$ is lower semicontinuous, there are  $x_0\in\overline{U}$ and $x_1\in\pa U$ and such that 
	$$
	a:=\inf_{U}(u_+-u_-)=\min_{x\in\overline{U}} \left(u_+(x)-u_-(x)\right)=u_+(x_0)-u_-(x_0).
	$$
$$
b:=\liminf_{x\rightarrow\pa U}\left(u_+(x)-u_-(x)\right)=\min_{x\in\pa U} \left(u_+(x)-u_-(x)\right)=u_+(x_1)-u_-(x_1).
$$
Suppose by contradiction that $a<0$. Then we can show $a<b$ in the following two possible cases respectively:
\begin{itemize}
	\item If either $u_+$ has finite inner derivatives or $u_-$ has infinite inner derivatives at $x_1$, then $b=u_+(x_1)-u_-(x_1)\geq 0>a$ by assumption;
	\item Otherwise, $u_+$ has infinite inner derivatives and $u_-$ has finite inner derivatives at $x_1$. This means that given $x\in U$, we have
	$$
	\lim_{t\to0^+}\frac{u_+(x_1+t(x-x_1))-u_+(x_1)}{t}=-\infty,\ \lim_{t\to0^+}\frac{u_-(x_1+t(x-x_1))-u_-(x_1)}{t}\in\R.
	$$
	As a consequence, when $t>0$ is small enough, we have
	$$
	\left[u_+(x_1+t(x-x_1))-u_+(x_1)\right]-\left[u_-(x_1+t(x-x_1))-u_-(x_1)\right]<0,
	$$  
	hence $b>u_+(x_1+t(x-x_1))-u_-(x_1+t(x-x_1))\geq a$.
\end{itemize}
This contradicts Lemma \ref{lemma_maximum} and finishes the proof.
\end{proof}

We proceed to review results on \emph{regularity} of Monge-Amp\`ere equations. The following theorem follows from the Evans-Krylov estimate for more general nonlinear elliptic PDEs:
\begin{theorem}[{\cite[Theorem 3.1]{trudwang}}] \label{solution smooth}
	Let $\Omega\subset\R^n$ be a bounded convex domain, $u:\Omega\to\R$ be a convex function and $g$ be a positive smooth function on $\Omega$. If $u$ is a generalized solution of $\det\D^2 u=g$ and is strictly convex, then $u$ is smooth.
\end{theorem}
There is a simple criterion for strictly convexity  in dimension $2$ exclusively:
\begin{theorem}[Aleksandrov-Heinz, see {\cite[Remark 3.2]{trudwang}}]\label{solution strictly convex dimension 2}
	Let $\Omega\subset\R^2$ be a bounded convex domain and $u:\Omega\to\R$ be a convex function. If $u$ is a generalized solution of $\det\D^2 u=g$ with $g\geq c$ for a constant $c>0$, then $u$ is strictly convex.
\end{theorem}
Theorems \ref{solution strictly convex dimension 2} and \ref{solution smooth} provide the following regularity result in dimension $2$:
\begin{theorem} \label{solution smooth in dim 2}
	Let $\Omega\subset\R^2$ be a bounded convex domain and let $u:\Omega\to\R$ be a convex function. If $u$ is a generalized solution of $\det\D^2 u=g$ with $g$ positive and smooth, then $u$ is smooth.
\end{theorem}
 
 Finally, we collect some simple \emph{existence} results of generalized solutions to the Monge-Amp\`ere equation $\det\D^2u=g$. For $g=0$, it basically follows from Lemma \ref{lemma_pleated} that the convex envelop $\env{\phi}$ of any $\phi\in\LC(\pa\Omega)$ gives a solution in the interior of effective domain:
 \begin{theorem}[{\cite[Theorem 1.5.2]{gutierrez}}] \label{lemma convex envelope det=0}
 	Let $\Omega$ be a bounded convex domain in $\R^n$ and let $\varphi\in\LC(\pa\Omega)$. Then $\det\D^2\env{\phi}=0$ in the interior of $\dom{\env{\phi}}$.
 \end{theorem}
 For the Dirichlet problem with continuous boundary value, we have: 
 \begin{theorem}[{\cite[Theorem 1.6.2]{gutierrez}, \cite[Theorem 2.14]{figalli}}] \label{thm existence classical monge ampere}
 	Let $\Omega\subset\R^n$ be a bounded strictly convex domain and consider $\phi\in \C^0(\partial\Omega)$ and $g$ a nonnegative Lebesgue integrable function on $\Omega$. Then there exists a unique generalized solution $u\in \C^0(\overline\Omega)$ to the Dirichlet problem
 	$$
 	\begin{cases}
 	\det \D^2u=g\,\text{ in }\Omega, \\
 	u|_{\pa\Omega}=\phi.
 	\end{cases}
 	$$
 \end{theorem}
Note that the uniqueness assertion follows from Comparison Principle.

\section{Monge-Amp\`ere equation for affine $(C,k)$-hypersurfaces}\label{sec_cagc}
In this section, we first compute the affine normals and intrinsic data (see Sections \ref{subsec_intrinsic} and \ref{subsec_affinenormal}) of the Legendre map of a smooth function (Definition \ref{def_legendremap}), viewed as a hypersurface immersion, then we use the computation to express the condition for a complete $C$-convex hypersurface $\Sigma=\gra{u^*}$ (with $u\in\S(\Omega)$, see Theorem \ref{thm_graph}) to be an affine $(C,k)$-hypersurface as a Monge-Amp\`ere equation on $u$.

\subsection{Computations of intrinsic data}\label{subsec_charamonge}
Recall from Section \ref{subsec_legendre1} that given an open set $U\subset\R^n$ and a function $u\in\C^1(U)$, we define the \emph{Legendre map} of $u$ as
	$$
f:U\rightarrow\mathbb{R}^{n+1},\ f(x)= 
\begin{pmatrix}
\D u(x)\\[3pt]
x\cdot \D u(x)-u(x)
\end{pmatrix}.
$$
Here and below, we write points in $\R^n$ and $\R^{n+1}$ in coordinates as column vectors. In particular, the gradient is written in column as $\D u=\transp(\pa_1u,\cdots \pa_n u)$. Also view the hessian $\D^2u=(\pa_{ij}u)$
 as an $n\times n$ matrix of functions. 
 
Consider $\R^{n+1}$ as an affine space equipped with the volume form given by the standard determinant of $(n+1)\times(n+1)$ matrices.
 It turns out that $f$ is a non-degenerate hypersurface immersion (see Section \ref{subsec_intrinsic}) if and only if $\D^2u$ is non-degenerate. The following proposition computes some of the intrinsic data of $f$:
\begin{proposition}\label{prop_computations}
Let $U\subset\mathbb{R}^n$ be an open set, $u:U\rightarrow\mathbb{R}$ be a smooth function with $\det\D^2u>0$ and $f: U\rightarrow\mathbb{R}^{n+1}$
be the Legendre map of $u$, viewed as a non-degenerate hypersurface immersion. Then 
the Legendre map $N:U\to\R^{n+1}$ of the function 
$$
w:=-(\det\D^2u)^{-\frac{1}{n+2}}
$$
is an affine normal mapping of $f$. The resulting affine metric and shape operator are
$$
h=-\frac{1}{w}\D^2 u,\quad S=(\D^2u)^{-1}\D^2w.
$$
Moreover, the affine conormal mapping dual to $N$ is 
$$
N^*:U\to\R^{*(n+1)}\cong\mathbb{R}^{n+1},\quad N^*(x)=\frac{1}{w}\begin{pmatrix}
x\\
-1
\end{pmatrix}.
$$
\end{proposition}
Here, as in Section \ref{subsec_analytic}, we identify the dual vector space $\R^{*(n+1)}$ with $\R^{n+1}$ using the standard inner product.
\begin{proof}
Given any map $N:U\to\R^{n+1}$, we let $N_0:=\transp{(N^1,\cdots, N^n)}:U\to\R^n$ denote its first $n$ entries and put $w(x):=x\cdot N_0(x)-N^{n+1}(x)$. We shall prove the first statement of the proposition by showing that $N$ is an affine normal mapping of $f$ if and only if $w$ has the required expression up to sign and $N$ is the Legendre map of $w$.

Let $I$ denote the $n\times n$ identity matrix. A computation gives
 \begin{equation}\label{eqn_proofcomputations}
(\pa_1f,\cdots,\pa_nf, N)=
\begin{pmatrix}
I&N_0\\
\transp{x}&N^{n+1}
\end{pmatrix}
\begin{pmatrix}
\D^2 u&\\
&1
\end{pmatrix}.
\end{equation}
The determinant of the above matrix is $-w\det\D^2u$. So $N$ gives a transversal vector field of  $f$ if and only if $w\neq0$. In this case, the induced volume form is 
$$
\nu=\det(\pa_1f,\cdots,\pa_nf, N)\,\dx^1\wedge\cdots\wedge\dx^n=-w\det\D^2u\,\dx^1\wedge\cdots\wedge\dx^n,
$$
while the other intrinsic data are determined by Eq.(\ref{eqn_matrix}) in
Section \ref{subsec_centroaffine}. In order to obtain explicit expressions, we first compute the inverse matrix to (\ref{eqn_proofcomputations}) and get
$$
(\pa_1f,\cdots,\pa_nf,N)^{-1}=\frac{1}{w}
\begin{pmatrix}
(\D^2u)^{-1}&\\
&1
\end{pmatrix}
\begin{pmatrix}
w I-N_0\transp{x}&N_0\\
\transp{x}&-1
\end{pmatrix}.
$$
Multiplying it to both sides of Eq.(\ref{eqn_matrix}) and computing, we get
$$
\tau=\frac{1}{w}\left(\transp{x}\,\dif N_0-\dif N^{n+1}\right),\quad h=-\frac{1}{w}\D^2u.
$$
We also get the following expression of $S$ when $\tau=0$ particularly:
$$
S=(\D^2u)^{-1}(\pa_1N_0,\cdots,\pa_nN_0).
$$

By definition, $N$ is an affine normal field of $f$ if $\tau=0$ and the volume form 
$$
\dif\vol_h=|w|^{-\frac{n}{2}}|\det\D^2u|^\frac{1}{2}\dx^1\wedge\cdots\wedge \dx^n
$$ 
coincides with $\nu$. By the expression of $\nu$ obtained earlier, the latter condition is equivalent to the equality $w=\pm(\det\D^2u)^{-\frac{1}{n+2}}$. On the other hand, one can check from the definitions of $w$ and the expression of $\tau$ that $\tau=0$ if and only if $N$ is the Legendre map of $w$. This proves the first statement.  

While the required expression of $h$ is already obtained above, using the condition that $N$ is the Legendre map of $w$, one checks that the above expressions of $S$ coincides with the required one, and that the $N^*$ given by the expression in the statement of the proposition satisfies 
$\transp{N}^*(\pa_1f,\cdots,\pa_n f,N)=(0,\cdots,0,1)$, which means $N^*$ is the affine conormal mapping dual to $N$. So the proof is completed.
\end{proof}

\subsection{Equation of affine $(C,k)$-hypersurfaces}\label{subsec_equations}
Recall from Section \ref{subsec_analytic} that given a bounded convex domain $\Omega\subset\R^n$, $\S_0(\Omega)$ is the space of lower semicontinuous convex functions $u:\R^n\to\R$ 
such that $u$ is smooth, locally strongly convex and has the gradient blowup property $\lim_{x\to\pa U}|\D u(x)|=+\infty$ in some convex subdomain $U\subset\Omega$ with $u=+\infty$ outside of $\overline{U}$. 

We showed in Theorem \ref{thm_graph} that the complete  
$C$-convex hypersurfaces (see Definition \ref{def_cconvex}) in $\R^{n+1}$ which are smooth and locally strongly convex are exactly the entire graphs $\gra{u^*}$, $u\in\S_0(\Omega)$, where $u^*$ is the Legendre transforms $u$ and $\Omega$ is an affine section of the opposite dual cone $-C^*$.
The goal of this section is to deduce from Proposition \ref{prop_computations} conditions on $u\in\S_0(\Omega)$ for graph $\gra{u^*}$ to be an affine $(C,k)$-hypersurface. 

For the particular case of affine spheres, Proposition \ref{prop_computations} implies:
\begin{corollary}\label{coro_affinesphereequation}
Let $C$ and $\Omega$ be as in Section \ref{subsec_analytic}. 
Then $\Sigma\subset\R^{n+1}$ is a complete hyperbolic affine sphere generating $C$ with affine shape operator the identity if and only if $\Sigma=\gra{w^*}$ for $w\in\S_0(\Omega)$ satisfying
\begin{equation}\label{eqn_chengyau}
\begin{cases}
\det \D^2w=(-w)^{-n-2}\ \text{  in }\Omega,\\
w|_{\pa\Omega}=0.
\end{cases}
\end{equation}
In this case, the affine sphere $\Sigma^*$ dual to $\Sigma$ (see Section \ref{subsec_conormal}) is given by
$$
\Sigma^*=\left\{\frac{1}{w(x)}\begin{pmatrix}
x\\
-1
\end{pmatrix}\,\Big|\, x\in\Omega\right\}\subset\mathbb{R}^{*(n+1)}\cong\R^{n+1}.
$$
\end{corollary}
Therefore, the Cheng-Yau theorem on unique existence of affine spheres (Theorem \ref{thm_chengyau1}) is a consequence of the following unique solvability result on \eqref{eqn_chengyau}:
\begin{theorem}[Cheng-Yau \cite{chengyau1}, analytic version of Theorem \ref{thm_chengyau1}]\label{thm_chengyau2}
For every bounded convex domain $\Omega\subset\mathbb{R}^n$, there exists a unique convex function $w_\Omega\in\CC{\Omega}$ satisfying Eq.(\ref{eqn_chengyau}).
Moreover, $w_\Omega$ satisfies $\lim_{x\rightarrow x_0}|\nabla w_\Omega(x)|\rightarrow+\infty$ for all $x_0\in\pa\Omega$.
\end{theorem}
In the sequel, we refer to the function $w_\Omega$ as the \emph{Cheng-Yau support function} of $\Omega$. 
\begin{remark}\label{remark_dualscaling}
By Corollary \ref{coro_affinesphereequation}, the relation between $w_\Omega$ and affine spheres is twofold: on one hand, if we extend $w_\Omega$ to $\R^n$ by setting $w_\Omega=+\infty$ outside of $\overline{\Omega}$, then the graph of its Legendre transform $w_\Omega^*$ is the affine sphere $\Sigma_C$ from Theorem \ref{thm_chengyau1}; on the other hand, while the dual cone $C^*$ consists of all negative scalings of points $(x,-1)\in\R^{n+1}$ with $x\in\Omega$, the locus of the $\frac{1}{w_\Omega(x)}$-scaling is the dual affine sphere $\Sigma_{C^*}$ in $C^*$. The latter property can be used to determine $w_\Omega$ when $\Sigma_{C^*}$ is known.
\end{remark}

Complete affine $(C,k)$-hypersurfaces can now be characterized through Monge-Amp\`ere equation as follows: 
\begin{corollary}\label{coro_ckmongeampere}
Let $C$ and $\Omega$ be as in Section \ref{subsec_analytic} and $w_\Omega$ be the Cheng-Yau support function of $\Omega$. Then $\Sigma\subset\R^{n+1}$ is a complete affine $(C,k)$-hypersurface with $k>0$ if and only if $\Sigma$ is the graph of the Legendre transform of some $u\in\S_0(\Omega)$ satisfying
\begin{equation}\label{eqn_ck}
\det\D^2u=\frac{k^{-\frac{n+2}{2(n+1)}}}{(-w_\Omega)^{n+2}}\,\text{ in }U:=\interior\dom{u}.
\end{equation}
In this case, the image of the projectivized affine conormal mapping $\mathbb{P}\circ N^*:\Sigma\to \mathbb{P}(C^*)\cong\Omega$ is exactly $U$.
\end{corollary}
This result is essentially contained in \cite{lisimoncrelle} albeit in a local sense.
\begin{proof}
Affine $(C,k)$-hypersurfaces are smooth, locally strongly convex and $C$-convex (see Section \ref{subsec_cagc}). Therefore, by Theorem \ref{thm_graph}, a complete affine $(C,k)$-hypersurface is the graph of the Legendre transform $u^*$ of some $u\in\S_0(\Omega)$. We shall determine the condition on $u$ for the graph to be affine $(C,k)$. Since $u^*$ is a smooth convex function on the entire $\R^n$ (see Lemma \ref{lemma_entire} and Theorem \ref{thm_graph} (\ref{item_thmgraph2})), we have $\D u(U)=\D u(\R^n)=\dom{\D u^*}=\R^n$ (the first equality follows from the definition of $\S_0(\Omega)$ and the second from Corollary \ref{coro_subdiff}), hence $\Sigma$ is parametrized by the Legendre map of $u|_U\in\C^\infty(U)$ (see Section \ref{subsec_legendre1}). Proposition \ref{prop_computations} and Corollary \ref{coro_affinesphereequation} then imply that the affine normal mapping of $\Sigma$ has image in a scaling of the Cheng-Yau affine sphere $\Sigma_C$ if and only if $w:=-(\det\D^2u)^{-\frac{1}{n+2}}$ is a constant multiple of the Cheng-Yau support function $w_\Omega$, \ie 
$$
\det\D^2u=(-w)^{-n-2}=c(-w_\Omega)^{-n-2}\, \text{ in } U
$$
for some constant $c>0$. In this case, we have $w=c^{-\frac{1}{n+2}}w_\Omega$ and the affine Gaussian curvature of $\Sigma$ is 
$$
k=\det(S)=\det(\D^2u)^{-1}\det\D^2w=c^{-1}(-w_\Omega)^{n+2}c^{-\frac{n}{n+2}}\det\D^2w_\Omega=c^{-\frac{2(n+1)}{n+2}},
$$
which yields $c=k^{-\frac{n+2}{2(n+1)}}$. Therefore, $\Sigma=\gra{u^*}$ is an affine $(C,k)$-hypersurface if and only if $u$ satisfies (\ref{eqn_ck}). This proves the first statement. The second statement follows from the expression of affine conormal given in Proposition \ref{prop_computations}.
\end{proof}

\subsection{Properties of the Cheng-Yau support function}
In this section, we collection some results on the Cheng-Yau support function $w_\Omega$ from Theorem \ref{thm_chengyau2}.

We first give explicit expression of $w_\Omega$ when $\Omega$ is a ball or a simplex, which correspond to hyperboloids and \c{T}i\c{t}eica affine spheres from Example \ref{example_hyperboloid} and \ref{example_titeica}:
\begin{example}[\textbf{Balls and simplices}]\label{example_support}
Let $\mathbb{B}:=\{x\in\R^n\mid |x|<1\}$ be unit ball. Then
$$
w_\mathbb{B}(x)=-\sqrt{1-|x|^2}.
$$
For a general ball $B=B_R(x_1):=\{x\in\R^n\mid |x-x_1|<R\}$, we have
$$
w_B=-R^{-\frac{1}{n+1}}\sqrt{R^2-|x-x_1|^2}.
$$

On the other hand, for the simplex $\Delta\subset\R^n$ with vertices $x_0,\cdots,x_n\in\R^n$, given by
$\Delta:=\{t_0x_0+\cdots +t_nx_n\mid t_i>0,\, t_0+\cdots +t_n=1\}$, we have
$$
w_\Delta(x)=-\left(\frac{\vol(\Delta)}{\Lambda }\,t_0(x)\cdots t_n(x)\right)^\frac{1}{n+1},
$$
where $\vol(\Delta)$ is the volume of $\Delta$, $\Lambda=\Lambda_n$ is the constant mentioned in Example \ref{example_titeica} and the functions $t_i:\Delta\to(0,1)$ are determined by 
$$
t_0(x)x_0+\cdots+t_n(x)x_n=x,\ t_0(x)+\cdots t_n(x)=1\,\text{ for all }x\in\Delta.
$$ 

Let $P_i\subset\R^n$ denote the hyperplane spanned by the vertices of $\Delta$ other than $x_i$. It is elementary to check that after choosing a Euclidean metric on $\R^n$, we can write
$$
t_i(x)=\frac{\dist(x,P_i)}{\dist(x_i,P_i)}~,
$$
where ``$\dist$'' stands for the distance induced by the metric.
\end{example}
These expressions are obtained using Remark \ref{remark_dualscaling} and the expressions of affine spheres given in Example \ref{example_titeica}. The dimensional exponents appearing in the expressions can be justified from the fact that 
 $w_{g(\Omega)}(x)=\mathsf{Jac}(g)^{\frac{1}{n+1}}w_\Omega(g^{-1}(x))$
 for any affine transformation $g:\R^n\to\R^n$, where $\mathsf{Jac}(g)$ denotes the Jacobian of $g$.

Our main technique to deal with a general $w_\Omega$ is to compare it with the above special ones using the following lemma:
\begin{lemma}\label{lemma_potentialcomparison}
	Let $\Omega_1$ and $\Omega_2$ be bounded convex domains in $\mathbb{R}^n$ such that $\Omega_1\subset\Omega_2$. Then  we have $w_{\Omega_1}\geq w_{\Omega_2}$ on $\Omega_1$.
\end{lemma}
\begin{proof}
Assume by contradiction that the required inequality does not hold, so that the function $w_{\Omega_1}-w_{\Omega_2}$ on $\overline{\Omega}_1$, which takes nonnegative values on $\pa\Omega_1$,  achieves its negative minimum at some $x_0\in\Omega_1$. It follows that the hessian $\D^2(w_{\Omega_1}-w_{\Omega_2})(x_0)$ is positive semidefinite, hence 
$$
(-w_{\Omega_1})^{-n-2}=\det \D^2 w_{\Omega_1}\geq \det \D^2 w_{\Omega_2}=(-w_{\Omega_2})^{-n-2}
$$
at $x_0$. But this contradicts the fact that $w_{\Omega_1}(x_0)-w_{\Omega_2}(x_0)<0$.
\end{proof}

We proceed to show that as a sequence of convex domains approaches $\Omega$ from outside, the resulting sequence of Cheng-Yau support functions converges uniformly to $w_\Omega$ on $\Omega$:
\begin{proposition}\label{prop_epsilon}
	Given a bounded convex domain $\Omega\subset\R^n$ and $\eps>0$, there exists $\delta>0$ such that for any convex domain $\Omega'$ containing $\Omega$ and contained in a $\delta$-neighborhood of $\Omega$, we have
	$w_\Omega-\eps\leq w_{\Omega'}\leq w_\Omega$ in $\Omega$.
\end{proposition}
In the statement of the proposition and the proof below, we fix an auxiliary Euclidean metric on $\R^n$, under which the $\delta$-neighborhood is defined.
\begin{proof}
The second inequality is given by Lemma \ref{lemma_potentialcomparison}. To prove the first one, we let $\Omega_\delta$ denote the $\delta$-neighborhood of $\Omega$, which is itself a bounded convex domain. Let $w_\delta:=w_{\Omega_\delta}\in\C^0(\overline{\Omega}_\delta)$ denote the Cheng-Yau support function of $\Omega_\delta$. Key to the proof is to show that
\begin{equation}\label{eqn_proofepsilon}
\lim_{\delta\to 0}\min_{\pa\Omega}w_\delta=0.
\end{equation}

To this end, we claim that there is a constant $C$ only depending on $n$ such that for every bounded convex domain $U\subset\R^n$ and every $y\in\pa U$ there exists a simplex $\Delta\subset\R^n$ containing $U$ with $\vol(\Delta)\leq C\,\mathsf{diam}(U)^n$ and $y\in\pa\Delta$ (where ``$\mathsf{diam}$'' stands for diameter). To show this, we fix a simplex $\Delta_0$ containing the unit half-ball $B:=\{x\in\R^n\mid |x|<1,\, x^1>0\}$ such that the boundary $\pa\Delta_0$ contains the origin $0$. Given $U$ and $y$, there is an isometry $g:\R^n\rightarrow\R^n$ with $g(0)=y$ such that $U$ is contained in the half-ball $g\left(\mathsf{diam}(U)B\right)$. The simplex $\Delta=g(\mathsf{diam}(U)\Delta_0)$ then satisfies the requirements of the claim with $C=\vol(\Delta_0)$.

To prove (\ref{eqn_proofepsilon}), we fix $\delta>0$, a point $x_0\in\pa\Omega$ and a supporting hyperplane $H$ of $\Omega$ at $x_0$. Let $y\in\R^n$ be the point such that the vector $\overrightarrow{x_0y}$ is orthogonal to $H$, has length $\delta$ and points towards the exterior of $\Omega$. Then $y$ is on the boundary of $\Omega_\delta$ and the above claim yields a simplex $\Delta$ containing $\Omega_\delta$ such that $y\in\pa\Delta$ and 
$$
\vol(\Delta)\leq C\,\mathsf{diam}(\Omega_\delta)^n\leq C(\mathsf{diam}(\Omega)+2\delta)^n.
$$
Let $v_0,\cdots, v_n\in\pa\Delta$ be the vertices of $\Delta$ such that $y$ is contained in the face of $\Delta$ spanned by $v_1,\cdots v_n$. Let $P_0\subset\R^n$ be the hyperplane containing that face. Since $\Delta$ contains $\Omega$, we have
$$
\dist(v_0,P_0)\geq h_\Omega:=\sup\{h\mid \text{ $\Omega$ contains a ball with diameter $h$ }\}.
$$ 
Example \ref{example_support} gives 
$
w_\Delta(x_0)=-\left(\Lambda ^{-1}\vol(\Delta)t_0\cdots t_n\right)^\frac{1}{n+1}
$
with $t_0,\cdots t_n\in(0,1)$ and 
$t_0=\frac{\dist(x_0,P_0)}{\dist(v_0,P_0)}$. The above estimates then yield
\begin{align*}
|w_\Delta(x_0)|&\leq\left(\frac{\dist(x_0,P_0)}{\dist(v_0,P_0)}\cdot\frac{\vol(\Delta)}{\Lambda }\right)^\frac{1}{n+1}\leq \left(\frac{\delta}{\Lambda h_\Omega}\vol(\Delta)\right)^\frac{1}{n+1}\\
&\leq \left(\frac{C\delta}{\Lambda h_\Omega}(\mathsf{diam}(\Omega)+2\delta)^n\right)^\frac{1}{n+1}.
\end{align*} 
The last bound is independent of $x_0$ and goes to $0$ as $\delta\to0$, while Lemma \ref{lemma_potentialcomparison} implies $w_\Delta(x_0)\leq w_\delta(x_0)\leq 0$. The required limit (\ref{eqn_proofepsilon}) follows.

The required inequality $w_\Omega-\eps\leq w_{\Omega'}$ can be deduced from (\ref{eqn_proofepsilon}) using the classical Maximum Principle for Monge-Amp\`ere equations (see Section \ref{subsec_preliminaries}): Given $\eps>0$, on one hand, (\ref{eqn_proofepsilon}) yields $\delta>0$ such that 
$$
w_\Omega-\eps=-\eps\leq w_\delta\leq w_{\Omega'} \text{ on }\pa\Omega
$$ 
for any convex domain $\Omega'$ with $\Omega\subset\Omega'\subset\Omega_\delta$, where the last equality is provided by Lemma \ref{lemma_potentialcomparison}; on the other hand, the inequality $w_{\Omega'
}\leq w_\Omega\leq 0$ implies
$$
\det\D^2(w_\Omega-\eps)=(-w_\Omega)^{-n-2}\geq (-w_{\Omega'})^{-n-2}=\det\D^2w_{\Omega'}\,\text{ in }\Omega.
$$
Thus, we can apply Maximum Principle and obtain the required inequality in $\Omega$.
\end{proof}

Finally, 
the growth of $w_\Omega(x)$ near a boundary point $x_0\in\Omega$ can be estimated in terms of the distance function from $x_0$. We only give below the results in dimension $2$, which is needed in the next section. In this case, $\Omega\subset\R^2$ is said to satisfy the \emph{exterior circle condition} at a boundary point $x_0\in\pa\Omega$  if there is a disk $B\subset\R^2$ containing $\Omega$ such that $x_0\in\pa B$.
\begin{proposition}\label{prop_womega}
Let $\Omega$ be a bounded convex domain in $\R^2$.
\begin{enumerate}
	\item\label{item_womega1} If $\Omega$ satisfies the exterior circle condition at $x_0\in\pa\Omega$, then there is a constant $c>0$ such that
$$
w_\Omega(x)\geq -c\,|x-x_0|^\frac{1}{2}\,\text{ for all }x\in\Omega.
$$
 \item\label{item_womega2} Let $L\subset\R^2$ be a straight line such that $L\cap\pa\Omega$ is a segment and let $x_0$ be an interior point of that segment. Then there is a constant $c>0$ such that
$$
w_\Omega(x)\leq -c\,\dist(x,L)^\frac{1}{3}
$$
for all $x\in\Omega$ in some neighborhood of $x_0$.
\end{enumerate}
\end{proposition}
\begin{proof}
(\ref{item_womega1}) Let $B=\{x\in\R^2\mid |x-x_1|<R\}$ ($R=|x_0-x_1|$) be a disk such that $\Omega\subset B$ and $x_0\in\pa B$. By Lemma \ref{lemma_potentialcomparison}, we have $w_B(x)\leq w_\Omega(x)\leq 0$ for all $x\in \Omega$, where $w_B(x)$ is given in Example \ref{example_support} as
$$
w_B(x)=-R^{-\frac{1}{3}}\sqrt{R^2-|x-x_1|^2}.
$$
The required inequality then follows from the estimates
\begin{align*}
0&\leq R^2-|x-x_1|^2=(R+|x-x_1|)(R-|x-x_1|)\\
&\leq 2R(|x_0-x_1|-|x-x_1|)\leq 2R|x-x_0|.
\end{align*}

(\ref{item_womega2}) Let $\Delta\subset\Omega$ be a triangle with vertices $v_0$, $v_1$ and $v_2$ such that $v_1,v_2\in L$. By Lemma \ref{lemma_potentialcomparison}, we have 
$$
w_\Omega(x)\leq w_\Delta(x)=-c'\big(t_0(x)t_1(x)t_2(x)\big)^\frac{1}{3} \text{ for all }x\in\Delta,
$$ 
where the equality is from Example \ref{example_support}, with $c'>0$ a constant, $t_0,t_1,t_2$ function on $\Delta$ with values in $(0,1)$ and in particular 
$t_0(x)=\frac{\dist(x,L)}{\dist(v_0,L)}$. The required inequality follows.

\end{proof}

\section{Analysis of the Monge-Amp\`ere equation}\label{sec_analysis}
In this section, we solve the Monge-Amp\`ere problem mentioned in the introduction and deduce the results on affine $(C,k)$-surfaces. The main tool used to deal with infinite boundary values is the generalized maximum principle reviewed in Section \ref{subsec_preliminaries}.
 We also examine in Section \ref{subsec_triangularcone} triangular cones in $\R^3$ as $C$-regular domains, which provide a situation where the exterior circle condition is not fulfilled and the solution does not have the gradient blowup property.

\subsection{Statement of the main results}\label{subsec_statement}
Recall from Section \ref{sec_convex_prel} the following notations and facts:
\begin{itemize}
\item
Given a bounded convex domain $\Omega\subset \R^n$, $\LC(\pa\Omega)$ denotes the space of lower semicontinuous functions $\phi:\pa\Omega\rightarrow\Rp$ which are convex on every line segment in $\pa\Omega$, and $\env{\phi}$ denotes the convex envelope of $\phi$.
\item
The notation ``$\mathsf{dom}$'' stands for the \emph{effective domain} of a $\Rp$-valued function, namely the set of points in the domain where the function has values in $\R$. For $\phi\in\LC(\pa\Omega)$, the convex hull $\chull{\dom{\phi}}$ of $\dom{\phi}\subset\pa\Omega$ in $\R^n$ coincides with $\dom{\env{\phi}}$ (Proposition \ref{prop_domain}).
\item
$\LC(\R^n)$ denotes the space of lower semicontinuous convex functions $u:\R^n\rightarrow\Rp$ that is not constantly $+\infty$.  Such a $u$ is determined by its restriction to the interior of effective domain $U:=\interior \dom{u}$ if $U$ is nonempty (Proposition \ref{prop_extension}), and is said to have \emph{infinite inner derivatives} at $x_0\in\pa U$ if either $u(x_0)=+\infty$ or the graph of $u$ over any line segment from $x_0$ to a point in $U$ has infinite slope at $x_0$ (Definition \ref{def_inner}). This is equivalent to $\lim_{x\rightarrow x_0}|\D u(x)|\rightarrow+\infty$ if $u$ is differentiable in $U$ (Proposition \ref{prop_derivative}).
\end{itemize}
We can now state our main results on Monge-Amp\`ere equations as the following theorem, covering Theorem \ref{thm_intro1} from the introduction:
\begin{manualtheorem}{\ref*{thm_intro1}'}\label{thm_main}
	Let $\Omega\subset\R^2$ be a bounded convex domain, $c>0$ be a constant and $\phi\in\LC(\pa\Omega)$ be such that $\chull{\dom{\phi}}$ has nonempty interior. Suppose $u\in\LC(\R^2)$ satisfies 
\begin{equation}\label{eqn_main}\tag{$\star$}
\begin{cases}
U:=\interior\dom{u} \text{ is nonempty and contained in } \Omega,\\
\det\D^2 u=c\,w_\Omega^{-4}\ \text{ in }U,\\
u|_{\pa\Omega}=\phi,
\end{cases} 
\end{equation}
where $w_\Omega\in\CC{\Omega}$ is the Cheng-Yau support function of $\Omega$, namely the unique convex solution to (see Theorem \ref{thm_chengyau2})
$$
\begin{cases}
\det \D^2w=(-w)^{-n-2}\ \text{  in }\Omega,\\
w|_{\pa\Omega}=0.
\end{cases}
$$
Then the following statements hold.
\begin{enumerate}
  \item\label{item_thmmain1}  Given $x_0\in\pa U\cap\pa\Omega$, if $\Omega$ satisfies the exterior circle condition at $x_0$, then $u$ has infinite inner derivatives at $x_0$.
 \item\label{item_thmmain2}  The following conditions are equivalent to each other:
\begin{enumerate}[label=(\roman*)]
	\item\label{item_thmmain21}  $u$ has infinite inner derivatives at every point of $\pa U\cap\Omega$;
	\item\label{item_thmmain22}  $\env{\phi}(x)-u(x)\to0$ as $x\in U$ tends to $\pa U$.
\end{enumerate} 
The second condition implies $\dom{u}=\chull{\dom{\phi}}$ and $u=\env{\phi}$ on $\pa U$.
    \item\label{item_thmmain3}  There exists a unique $u$ satisfying (\ref{eqn_main}) and the conditions in Part (\ref{item_thmmain2}). 
    \item\label{item_thmmain4}  Given $t\in\R$, let $u_t$ denote the $u$ produced by Part (\ref{item_thmmain3}) with parameter $c=e^{-t}$ in (\ref{eqn_main}). Then $t\mapsto u_t(x)$ is a strictly increasing concave function with
 $$
 \lim_{t\rightarrow-\infty}u_t(x)=-\infty,\quad \lim_{t\rightarrow+\infty}u_t(x)=\env{\phi}(x)
 $$
 for every $x\in U=\interior\chull{\dom{\phi}}$.
\end{enumerate}
\end{manualtheorem}
Note that Part (\ref{item_thmmain1}) of the theorem is independent of other parts.
\begin{remark}[\textbf{Regularity}]
	\label{remark_smooth}
	
		A priori, $u$ is merely continuous in $U$ and the Monge-Amp\`ere equation in \eqref{eqn_main} is understood in the generalized sense (see Section \ref{sec_cagc}).
	But classical regularity results on Monge-Amp\`ere equations in dimension $2$ (see Theorem \ref{solution smooth in dim 2} in Section \ref{subsec_preliminaries}) ensures that $u$ is actually smooth and strictly convex in $U$, hence a classical solution.
	
For a higher-dimensional bounded convex domain $\Omega\subset\R^n$ ($n\geq 3$), there exists continuous $\phi:\pa\Omega\to\R$ such that \eqref{eqn_main} has a unique generalized solution $u\in\C^0(\overline{\Omega})$ which is not smooth and restricts to affine functions on some line segments in $\Omega$ joining boundary points (hence $u$ has finite inner derivatives at these points), even when $\Omega$ is the unit ball. In fact, Bonsante and Fillastre \cite{bonsante-fillastre} constructed examples that correspond to regular domains $D\subset\R^{n,1}$ ($n\geq 3$) such that $D$ is invariant under an affine deformation of a uniform lattice in $\SO(n,1)$ but there is no smooth hypersurface of constant Gauss-Kronecker curvature generating $D$. On the other hand,
if $\pa\Omega$ and $\phi$ are both smooth, it is showed in \cite{lisimoncrelle} that \eqref{eqn_main} has a unique solution $u\in\CC{\Omega}$ and it satisfies $\lim_{x\rightarrow\pa\Omega}|\D u(x)|\to+\infty$.
\end{remark}

\begin{remark}[\textbf{$u$ not satisfying the conditions in Part (\ref{item_thmmain2})}]\label{remark_conditions}
If $\dom{\phi}=\pa\Omega$, then $U=\Omega$ and Condition \ref{item_thmmain21}  in Part (\ref{item_thmmain2}) is trivial. Otherwise, there exists $u$ satisfying (\ref{eqn_main}) without fulfilling \ref{item_thmmain21}  or \ref{item_thmmain22} . For instance, if $\Omega$ is strictly convex and $A\subsetneq\pa\Omega$ is a nonempty closed subset, then the characteristic function $\phi=\chi_A$ (defined by $\chi_A=0$ on $A$ and $\chi_A=+\infty$ outside of $A$) belongs to $\LC(\pa\Omega)$ and its convex envelope $\env{\phi}$ is the characteristic function of $\chull{A}$ on $\R^2$. In this case, the convex function 
$$
u(x):=
\begin{cases}
w_\Omega(x)& \text{ if }x\in \chull{A}\\
+\infty&\text{ if }x\notin \chull{A}
\end{cases}
$$
satisfies (\ref{eqn_main}) (with $c=1$) but not \ref{item_thmmain21}  or \ref{item_thmmain22} .
\end{remark}

%

We prove Parts (\ref{item_thmmain1}) -- (\ref{item_thmmain4}) of Theorem \ref{thm_main} through Sections \ref{subsec_infiniteness} -- \ref{subsec_dependence} and discuss in Section \ref{subsec_triangularcone} the case where $\Omega$ is not strictly convex, hence fails the exterior circle condition.  In the rest of of this section, we deduce from Theorem \ref{thm_main} the results on affine $(C,k)$-surfaces mentioned in the introduction.



\begin{proof}[Proof of Theorems \ref{thm_intromain}, \ref{thm_introfoliation} and Corollary \ref{coro_intro} in Introduction]
Given the cone $C$ as in the assumption of Theorem \ref{thm_intromain},  in the way explained in Section \ref{subsec_analytic}, we choose coordinates of $\R^3$ and get a convex domain $\Omega\subset\R^2$ as an affine section of $-C^*$, which can be identified projectively with $\mathbb{P}(C^*)$ and satisfies the exterior circle condition. In particular, $\Omega$ is strictly convex, hence $\LC(\pa\Omega)$ consists exactly of lower semicontinuous functions $\phi:\pa\Omega\rightarrow\Rp$, and $\chull{\dom{\phi}}$ has nonempty interior exactly when $\dom{\phi}$ has at least three points.

By Theorem \ref{thm_graph} (\ref{item_thmgraph1}), a proper $C$-regular domain $D$ can be written as $D=\sepi{\env{\phi}^*}$ for some $\phi\in\LC(\pa\Omega)$ with $\interior\chull{\dom{\phi}}\neq\emptyset$, while Corollary \ref{coro_ckmongeampere} says that a complete affine $(C,k)$-surface generating $D$ is exactly the graph of the Legendre transform of some $u\in\S_0(\Omega)$ satisfying Eq.(\ref{eqn_main}) in Theorem \ref{thm_main}, with  $c=k^{-\frac{2}{3}}$ (see Section \ref{subsec_analytic} or \ref{subsec_equations} for the definition of $\S_0(\Omega)$). By Parts (\ref{item_thmmain1}) and (\ref{item_thmmain3}) of Theorem \ref{thm_main} and classical regularity results (see Remark \ref{remark_smooth}), there exists a unique such $u$, which proves Theorem \ref{thm_intromain}.  Theorem  \ref{thm_introfoliation} then follows from Theorem \ref{thm_foliation} and Theorem \ref{thm_main} (\ref{item_thmmain4}). The construction also provides the required correspondences  ``\ref{item_corointro3}$\leftrightarrow$\ref{item_corointro1}$\leftrightarrow$\ref{item_corointro2}'' in Corollary \ref{coro_intro} as
$$
	\phi\,\longleftrightarrow\, D=\sepi{\env{\phi}^*}\,\longleftrightarrow\,\Sigma=\gra{u^*},
$$
where $u\in\S_0(\Omega)$ is related to $\phi$ through (\ref{eqn_main}) with $c=k^{-\frac{2}{3}}$. 
Finally, the identification between the image of affine conormal map of the above $\Sigma=\gra{u}$ and $\interior\chull{\dom{\phi}}$ is contained in Corollary \ref{coro_ckmongeampere}.
\end{proof}

\subsection{Infiniteness of inner derivatives}\label{subsec_infiniteness}
In this section, we collect some results on (in)finiteness of inner derivatives (see Definition \ref{def_inner} and Proposition \ref{prop_derivative}) that will be used later on and deduce Part (\ref{item_thmmain1}) of Theorem \ref{thm_main}. Note that although Definition \ref{def_inner} is concerned with convex functions $\R^n\to\Rp$, it also makes sense for $\R$-valued convex functions on convex domains through Proposition \ref{prop_extension}.
\begin{lemma}\label{lemma_edge}
Let $U\subset \R^2$ be a convex domain such that $\pa U$ contains an open line segment $I$, and $u\in\C^0(\overline{U})$ be a convex function such that 
$\det\D^2 u\geq c$ in $U$ for a constant $c>0$ and that $u$ restricts to an affine function on $I$.
Then $u$ has infinite inner derivatives at every point of $I$.
\end{lemma}
\begin{proof}
For any open triangle $\Delta\subset\R^2$, let $u_\Delta\in\C^0(\overline{\Delta})$ be the generalized solution to
$$
\begin{cases}
\det\D^2u=1\,\text{ in }\Delta\\
u|_{\pa\Delta}=0
\end{cases}
$$
given by Theorem \ref{thm existence classical monge ampere}. Letting $x_\Delta\in\Delta$ denote the barycenter of $\Delta$, we note that there is a constant $C>0$ independent of $\Delta$ such that 
\begin{equation}\label{eqn_proofedge}
u_\Delta(x_\Delta)=-C\,\mathsf{Area}(\Delta) \text{ for any triangle }\Delta\subset\R^2.
\end{equation}
To see this, fix a triangle $\Delta_0$ with area $1$ and take an affine transformation $g(x)=a(x)+b$ (where $a$ is a linear transformation of $\R^2$ and $b\in\R^2$) such that $g(\Delta_0)=\Delta$. The convex function  $\tilde{u}(x):=u_{\Delta_0}(g^{-1}x)$ on $\overline{\Delta}$ satisfies
$$
\det\D^2\tilde{u}(x)=\det(a)^{-2}\det\D^2u_{\Delta_0}(g^{-1}x)=\mathsf{Area}(\Delta)^{-2}.
$$ 
It follows that $u_\Delta=\mathsf{Area}(\Delta)\tilde{u}$, which implies (\ref{eqn_proofedge}) with $C=-\tilde{u}(x_{\Delta})=-u_{\Delta_0}(x_{\Delta_0})$.

We claim that $u_\Delta$ has infinite inner derivatives at the midpoints of its edges.
To prove this, by applying a volume-preserving affine transformation (which does not change the Monge-Amp\`ere measure), we can suppose without loss of generality that $x_0=(0,0)$ and the vertices of $\Delta$ are $(0,\pm1)$ and $(t_0,0)$ with $t_0>0$.
Assume by contradiction that $u_\Delta$ has finite inner derivatives at $(0,0)$, so that $\pa_1u_\Delta(0,0):=\lim_{t\rightarrow0^+}\frac{u_\Delta(t,0)}{t}$ is a negative real number and 
$$
v(x):=u_\Delta(x)-\pa_1u_\Delta(0)x^1
$$ 
(where $x^i$ denotes the $i^\mathrm{th}$ coordinate of $x\in\R^2$) satisfies $\pa_1v(0,0)=0$. In particular, $v(t,0)\geq0$ for $0\leq t\leq t_0$.

Fix $0<t<t_0$ and consider the triangle $\Delta_t$ with vertices $(0,\pm1)$ and $(t,0)$. Since $v$ is convex in $\Delta_t$ and vanishes on the vertical edge, we have $v(x)\leq v(t,0)$ for all $x\in\overline{\Delta}_t$, hence
$$
v(x)\leq u_{\Delta_t}(x)+v(t,0)
$$
for $x$ on the boundary of $\Delta_t$. By Comparison Principle, the inequality also holds for all $x\in\Delta_t$.
Taking $x$ to be the barycenter $x_{\Delta_t}=(\frac{t}{3},0)$ and applying (\ref{eqn_proofedge}), we get
$$
v(\tfrac{t}{3},0)\leq  -C\mathsf{Area}(\Delta_t)+v(t,0)=-Ct+v(t,0).
$$ 
But the right-hand side is negative when $t$ is small enough because $\pa_1v(0,0):=\lim_{t\rightarrow0^+}\frac{v(t,0)}{t}=0$. This contradicts the fact that $v(t,0)\geq0$ for $0\leq t\leq 1$ and completes the proof of the claim.

Now, under the assumptions of the lemma, for every $x_0$ in the interior of the line segment $I\subset\pa\Omega$, we can take a triangle $\Delta\subset U$ such that an edge $I'$ of $\Delta$ is a sub-segment of $I$ and has $x_0$ as its midpoint, and that the opposing vertex $x_1$ is in $U$. Let $a:\R^2\to \R$ be the affine function with $a=u$ on $I'\cup\{x_1\}$. Then we have $a+\sqrt{c}\,u_\Delta=u$ on $I'\cup\{x_1\}$ and hence $a+\sqrt{c}\,u_\Delta\geq u$ on $\pa\Delta$ by convexity, while 
$$
\det\D^2(a+\sqrt{c}\,u_\Delta)=c\det \D^2u_\Delta=c\leq \det\D^2 u
$$
holds in $\Delta$. By Comparison Principle, we have $a+\sqrt{c}\,u_\Delta\geq u$ throughout $\overline{\Delta}$. Since $a+\sqrt{c}\,u_\Delta$ equals $u$ at $x_0\in\pa\Delta$ and has infinite inner derivatives at $x_0$ as shown above, we conclude that $u$ has infinite inner derivatives at $x_0$, as required.
\end{proof}

\begin{lemma}\label{lemma_angle}
Let $\Delta\subset\R^2$ be an open triangle, $x_0\in\pa\Delta$ be a vertex and $u\in\C^0(\overline{\Delta})$ be a convex function.
\begin{enumerate}
\item\label{item_lemmaangle1}
If there are constants $c>0$ and $\alpha>-2$ such that
$$
\begin{cases}
\det\D^2 u(x)\leq c\,|x-x_0|^\alpha \text{ for all }x\in \Delta,\\
u|_{\pa\Delta}=0,
\end{cases}
$$
then $u$ has finite inner derivatives at $x_0$.
\item\label{item_lemmaangle2}
If there is a constant $c>0$ such that
$$
\det\D^2 u(x)\geq c\,|x-x_0|^{-2}\ \text{ for all }x\in \Delta,
$$
then $u$ has infinite inner derivatives at $x_0$.
\end{enumerate}
\end{lemma}
\begin{proof}
(\ref{item_lemmaangle1}) Assuming $x_0=0$ without loss of generality, we only need to find a convex function $v\in\C^0(\overline{\Delta})$ with finite inner derivatives at $0$ such that
\begin{equation}\label{eqn_proofangle1}
\begin{cases}
\det\D^2 v(x)= c'\,|x|^\alpha\, \text{ for some $c'>0$ and all $x\in \Delta$},\\
v|_{\pa\Delta}\leq0,\ v(0)=0
\end{cases}
\end{equation}
Comparison Principle would then imply $u\geq (c/c')^\frac{1}{2}v$, hence $u$ has finite inner derivatives at $0$ as well.

To find $v$, we set $\beta:=\frac{\alpha+4}{2}>1$. The function $x\mapsto |x|^\beta$ is convex and we have
$$
\det \D^2|x|^\beta=\beta^2(\beta-1)|x|^{\alpha}.
$$
Let $l:\R^2\rightarrow\R$ be a linear form such that 
$l(x)\geq |x|^\beta$ for all $x$ on the edge of $\Delta$ opposing the vertex $0$.
By convexity, $l(x)\geq |x|^\beta$ holds for all $x\in\overline{\Delta}$. Therefore,
$$
v(x):=|x|^\beta-l(x)
$$
satisfies (\ref{eqn_proofangle1}) and proves the required statement.

(\ref{item_lemmaangle2}) By applying an affine transformation to $\Delta$ and subtracting from $u$ the affine function whose values coincide with $u$ at the  vertices of $\Delta$, we may assume without loss of generality that
\begin{itemize}
\item
$x_0=(0,0)$ and the other two vertices of $\Delta$ are $(1,\pm1)$;
\item
$u=0$ at the vertices of $\Delta$. This implies $u\leq 0$ on $\overline{\Delta}$.
\end{itemize}

Consider the smooth function $w$ on the band $B:=\{x\in\R^2\mid 0<x^1<1\}$ (where $x^i$ is the $i^\mathrm{th}$ coordinate of $x$) defined by
$$
w(x):=x^1(-\log x^1)^\frac{1}{2}\left(\left(\frac{x^2}{x^1}\right)^2-1\right).
$$
By some computations, one checks that
$$
\det\D^2w(x)=\frac{1}{(x^1)^2}\left[1+\frac{1}{2}(-\log x^1)^{-1}-\left(\frac{x^2}{x^1}\right)^2\left(1+\frac{3}{2}(-\log x^1)^{-1}\right)\right]
$$
and that the hessian $\D^2w$ is positive definite on the convex domain (see Figure \ref{figure_domainw})
$$
W:=\left\{x\in\R^2\,\Big|\, 0<x^1<1,\,\left(\frac{x^2}{x^1}\right)^2\leq \frac{\frac{1}{2}-\log x^1}{\frac{3}{2}-\log x^1}\right\}=\{x\in B\mid \det\D^2 w(x)>0 \}\subset\Delta.
$$
\begin{figure}[h]%
	\centering
\includegraphics[width=.43\linewidth]{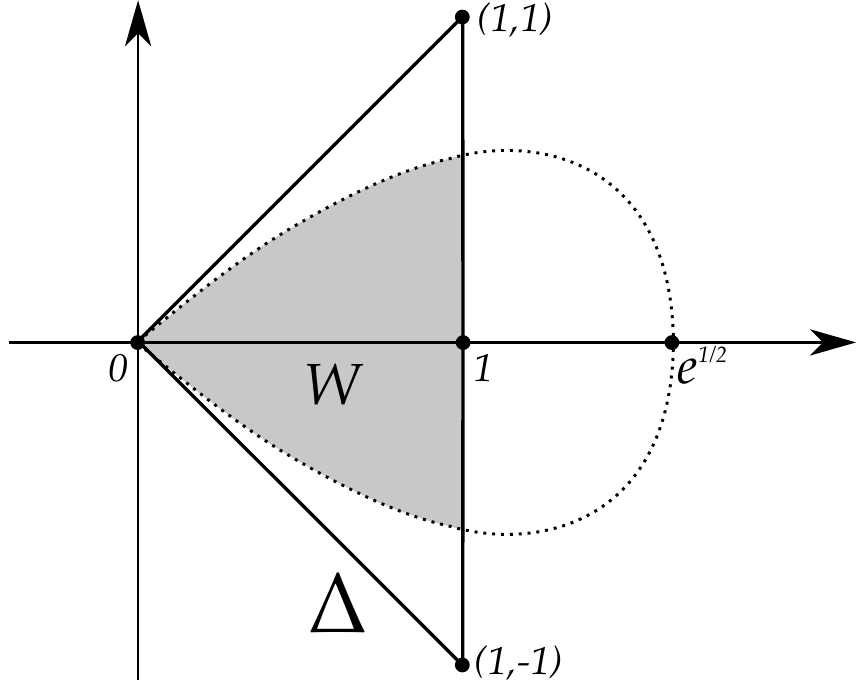}
\caption{The domain $W$. The dashed curve is $\big(\frac{x^2}{x^1}\big)^2= \frac{\frac{1}{2}-\log x^1}{\frac{3}{2}-\log x^1}$, which is tangent to both edges of the triangle $\Delta$ at the origin.}
	\label{figure_domainw}
\end{figure}

Fix $\lambda\in(0,1)$ and denote $V:=\{x\in W\mid x^1< \lambda\}$. Setting $w(0,0):=0$, we can view $w$ as a continuous convex function on $\overline{V}$ and verify the following properties:
\begin{itemize}
\item  by the above expression of $\det\D^2w$, there is a constant $c'>0$ such that
$$
 \det\D^2 w(x)\leq c'|x|^{-2}\,\text{ for all }x\in V;
$$ 
\item $w$ has infinite inner derivatives at $(0,0)\in\pa W$;
\item $\mu:=\inf_{x\in\pa V}w(x)/x^1$
$$
=\min\left\{-(-\log\lambda)^\frac{1}{2},\inf_{0<x^1\leq \lambda}(-\log x^1)^\frac{1}{2}\left(\frac{\frac{1}{2}-\log x^1}{\frac{3}{2}-\log x^1}-1\right)\right\}>-\infty.
 $$ 
 \end{itemize}
As a consequence, $v(x):=w(x)-\mu x^1$ is a continuous convex function on $\overline{V}$ with infinite inner derivatives $(0,0)$ and satisfies 
$$
\begin{cases}
\det\D^2 v(x)\leq c'|x|^{-2}\, \text{ for all $x\in V$},\\
v|_{\pa V}\geq 0.
\end{cases}
$$
We can thus apply Comparison Principle to get $u\geq(c/c')^\frac{1}{2}v$ on $\overline{V}$. Since $u(0,0)=v(0,0)=0$, we conclude that $u$ has infinite inner derivatives at $(0,0)$, as required.
\end{proof}

\begin{proof}[Proof of Theorem \ref{thm_main} (\ref{item_thmmain1})]  
Given $u$ and $x_0$ under the assumptions, we have
$$
\det\D^2u(x)=c\,w_\Omega(x)^{-4}\geq c'|x-x_0|^{-2} \text{ for all }x\in U,
$$
where the inequality follows from Proposition \ref{prop_womega}.
If $u(x_0)=+\infty$ then $u$ has infinite inner derivatives at $x_0$ by definition. 
Otherwise, take a triangle $\Delta$ with a vertex at $x_0$ and the other two vertices  in $U$. The restriction of $u$ to $\overline{\Delta}$ is continuous (see Section \ref{subsec_convexfunction}), so we can apply Lemma \ref{lemma_angle} (\ref{item_lemmaangle2}) to it and conclude that $u$ has infinite inner derivatives at $x_0$.
\end{proof}

\subsection{Equivalence between Conditions \ref{item_thmmain21}  and \ref{item_thmmain22} }\label{subsec_equivalence}

\begin{lemma}\label{lemma_uestimate}
If $u\in\LC(\R^2)$ satisfies (\ref{eqn_main}) and Condition \ref{item_thmmain21}  in Theorem \ref{thm_main}, then 
$$
\env{\phi}+\sqrt{c}\,w_U\leq u\leq \env{\phi}\ \text{ on }\overline{U},
$$
where $w_U$ is the Cheng-Yau support function of $U$.
\end{lemma}
\begin{proof}
The second inequality follows from the condition $u|_{\pa\Omega}=\phi$ by Corollary \ref{coro_supremum}. 

To prove the first inequality, we consider, for every $\delta>0$, the $\delta$-neighborhood of $U$ in $\Omega$:
$$
U_\delta:=\{x\in\Omega\mid \dist(x,U)<\delta\},
$$
where ``$\dist$'' stands for the distance defined with respect to a Euclidean metric on $\R^2$. Let $w_\delta:=w_{U_\delta}\in\C^0(\overline{U}_\delta)\cap\C^\infty(U_\delta)$ be the Cheng-Yau support function of $U_\delta$.

Recall that $\env{\phi}$ is the pointwise supremum of affine functions $a:\R^2\rightarrow\R$ with $a|_{\pa\Omega}\leq \phi$ (see Definition \ref{def_env}). For any such $a$, we shall show that
\begin{equation}\label{eqn_proofequivalence1}
u\geq a+\sqrt{c}\,w_\delta\,\text{ in } U
\end{equation}
using Generalized Comparison Principle (Lemma \ref{lemma_comparison}). By Lemma \ref{lemma_potentialcomparison}, we have $w_\Omega\leq w_\delta\leq 0$ in $U_\delta$, hence the required comparison of Monge-Amp\`ere measures
$$
\det\D^2u=c\,w_\Omega^{-4}\leq c\,w_\delta^{-4}=\det\D^2(a+\sqrt{c}\,w_\delta)~.
$$
On the other hand, we have
$$
u(x)=\env{\phi}(x)\geq a(x)=a(x)+\sqrt{c}\,w_\delta(x)\, \text{ for all } x\in\pa U\cap\pa\Omega.
$$
This gives the required comparison of boundary values in Lemma \ref{lemma_comparison}, because by Condition \ref{item_thmmain21}  and the fact that $a+\sqrt{c}\,w_\delta$ is smooth in $U_\delta$, any boundary point $x$ of $U$ such that either $u$ has finite inner derivatives at $x$ or $a+\sqrt{c}\,w_\delta$ has infinite inner derivatives at $x$ must be on the boundary of $\Omega$. Thus, Lemma \ref{lemma_comparison} implies (\ref{eqn_proofequivalence1}).

In view of Proposition \ref{prop_epsilon}, taking the pointwise supremum of the right-hand side of (\ref{eqn_proofequivalence1}) for all affine functions $a$ with $a|_{\pa\Omega}\leq \phi$ and all $\delta>0$, we obtain the first inequality and completes the proof of the lemma.
\end{proof}

\begin{proof}[Proof of Theorem \ref{thm_main} (\ref{item_thmmain2})] 
Let us first prove the second statement, namely, Condition \ref{item_thmmain22} implies $\dom{u}=\chull{\dom{\phi}}$ and $u=\env{\phi}$ on $\pa U$. Note that $\chull{\dom{\phi}}=\dom{\env{\phi}}$ by Proposition \ref{prop_domain}. Assuming \ref{item_thmmain22}, we only need to show that the convex sets $\dom{u}$ and $\dom{\env{\phi}}$ have the same closure and $u=\env{\phi}$ on their common boundary. By Corollary \ref{coro_supremum}, we have $u\leq\env{\phi}$ on $\R^2$, hence $\dom{u}$ contains $\dom{\env{\phi}}$. If the closures are not the same, then $U\setminus\dom{\env{\phi}}$ is a nonempty set touching the boundary of $U$. But we have $\env{\phi}-u=+\infty$ on this set, contradicting Condition \ref{item_thmmain22}. This proves $\overline{\dom{u}}=\overline{\dom{\env{\phi}}}$. For any $x_0\in\pa U=\pa\dom{u}=\pa\dom{\env{\phi}}$ and $x_1\in U$, we have $\lim_{t\to0^+}u((1-t)x_0+tx_1)=u(x_0)$ (see Section \ref{subsec_convexfunction}) and the same for $\env{\phi}$, hence Condition \ref{item_thmmain22} implies $u=\env{\phi}$ at $x_0$, as required.

The implication ``\ref{item_thmmain21} $\Rightarrow$ \ref{item_thmmain22}'' follows from Lemma \ref{lemma_uestimate}. Conversely, assuming \ref{item_thmmain22}, in order to show that $u$ has infinite inner derivatives at $x\in\pa U\cap\Omega$, we let $I$ be the connected component of $\pa U\cap\Omega$ containing $x$. Since $\overline{U}$ is the convex envelop of the subset $\dom{\phi}$ of $\pa\Omega$, $I$ is an open line segment with endpoints $x_1,x_2\in\pa\Omega$. If $\phi$ takes finite values on both $x_1$ and $x_2$, then $u|_I=\env{\phi}|_I$ is the affine function interpolating $\phi(x_1)$ and $\phi(x_2)$, and $u$ has infinite inner derivatives at every $x\in I$ by Lemma \ref{lemma_edge}. Otherwise, we have $u=\env{\phi}=+\infty$ on $I$ and $u$ has infinite inner derivatives by convention. This establishes ``\ref{item_thmmain22} $\Rightarrow$ \ref{item_thmmain21}'' and completes the proof.
\end{proof}

\subsection{Existence and uniqueness}\label{subsec_existence}
\begin{proof}[Proof of the existence part of Theorem \ref{thm_main} (\ref{item_thmmain3})]
 We shall construct a convex function $u:V:=\interior\chull{\dom{\phi}}\rightarrow\R$ such that 
 \begin{equation}\label{eqn_proofexistence}
 \det\D^2u=c\,w_\Omega^{-4},\quad \env{\phi}+\sqrt{c}\,w_V\leq u \leq\env{\phi}\,\text{ in }V,
 \end{equation}
 where $w_V$ is the Cheng-Yau support function of $V$. The extension of $u$ to $\R^2$ given by Proposition \ref{prop_extension} then satisfies (\ref{eqn_main}) and Condition \ref{item_thmmain22}, as required. The construction is standard and goes though the following steps.
 
 \begin{steps}
		\item Let $V_1\subset V_2\subset\cdots \subset V$ be an exhaustion of $V$ by strictly convex domains.
		Using Theorem \ref{thm existence classical monge ampere}, we find a convex function $u_i\in\C^0(\overline{V}_i)$ satisfying
		$$
		\begin{cases}
		\det\D^2 u_i=c\,w_\Omega^{-4}&\text{in }V_i, \\
		u_i=\env{\phi}&\text{on }\pa V_i.
		\end{cases}
		$$
		Each $u_i$ is strictly convex and smooth by Theorems \ref{solution strictly convex dimension 2} and \ref{solution smooth in dim 2}. 
		\item We show that 
		\begin{equation} \label{eq: two-sided bound}
		\env{\phi}+\sqrt{c}\,w_V\leq u_i\leq \env{\phi} \text{ in } V_i.
		\end{equation}
        Since $w_V<0$ in $V$, these inequalities hold on $\pa V_i$ by construction. On the other hand, we have $\det\D^2\env{\phi}=0$ by Lemma \ref{lemma convex envelope det=0}. Using Lemmas \ref{lemma ma measure of sum} and \ref{lemma_potentialcomparison}, we obtain
        $$
        \det\D^2\left(\env{\phi}+\sqrt{c}\,w_V\right)\geq \det\D^2\env{\phi}+c\,\det\D^2w_V=c\,w_V^{-4}
        \geq c\,w_\Omega^{-4}
        =\det\D^2 u_i.
        $$
        Therefore, the respective Monge-Amp\`ere measure densities of the three functions in (\ref{eq: two-sided bound}) satisfy the reversed inequalities. Inequality (\ref{eq: two-sided bound}) then follows from the classical Comparison Principle (see Section \ref{subsec_preliminaries}).

		\item By (\ref{eq: two-sided bound}), the convex functions $u_i$ are uniformly bounded. On each $V_k$, applying Arzela-Ascoli to the the sequence of functions $(u_i)_{i\geq k}$, one can extract a convergent subsequence. Moreover, we can assume that the subsequence for $V_{k+1}$ is a subsequence of the subsequence for $V_k$. Let $u:V\to\R$ be the limit obtained by this diagonal argument. It follows from Lemma \ref{convergence of solutions} and Inequality \eqref{eq: two-sided bound} that $u$ is a convex function fulfilling the requirement (\ref{eqn_proofexistence}).
	\end{steps}
\end{proof}

\begin{proof}[Proof of the uniqueness part of Theorem \ref{thm_main} (\ref{item_thmmain3})] 
Suppose $u_1$ and $u_2$ both satisfy (\ref{eqn_main}) and Condition \ref{item_thmmain21} .  Lemma \ref{lemma_uestimate} implies 
$|u_1-u_2|\leq -\sqrt{c}\,w_U$
in $U=\interior\dom{u_1}=\interior\dom{u_2}=\chull{\dom{\phi}}$, hence
$$
\lim_{x\rightarrow\pa U}\left(u_1(x)-u_2(x)\right)=0
$$
because $w_U\in\C^0(\overline{U})$ vanishes on $\pa U$. Maximum Principle (Lemma \ref{lemma_maximum}) then gives $\inf_U(u_1-u_2)=\inf_U(u_2-u_1)=0$, \ie $u_1=u_2$.
\end{proof}

\subsection{Dependence on the parameter}\label{subsec_dependence}
\begin{proof}[Proof of Theorem \ref{thm_main} (\ref{item_thmmain4})]
We first show that $t\mapsto u_t(x)$ is nondecreasing with an argument similar to the above proof of uniqueness. Given $t_1<t_2$, we have 
	$$
	\det \D^2u_{t_1}=e^{-t_1}w_\Omega^{-4}>e^{-t_2}w_\Omega^{-4}=\det \D^2u_{t_2}
	$$
in $U=\interior\dom{u_t}=\interior\chull{\dom{\phi}}$. On the other hand, Lemma \ref{lemma_uestimate} gives
\begin{equation}\label{eqn_proofdepence}
\env{\phi}+e^{-\frac{t}{2}}w_U\leq u_{t}\leq \env{\phi}\, \text{ in }\overline{U} \text{ for all }t\in\R
\end{equation}
Since $w_U\in\C^0(\overline{U})$ vanishes on $\pa U$, it follows that $\lim_{x\to\pa U}\left(u_{t_2}(x)-u_{t_1}(x)\right)=0$. Using Maximum Principle (Lemma \ref{lemma_maximum}), we get $\inf_V(u_{t_2}-u_{t_1})=0$, or  $u_{t_1}\leq u_{t_2}$.
	
Next, let us establish the concavity of $t\mapsto u_t(x)$. Using the fact that $\log \det$ is a concave function on the space of positive definite matrices, we get
	\begin{align*}
	&\log\det \D^2\left(\frac{u_{t_1}+u_{t_2}}{2}\right)=
	\log\det\left(\frac{\D^2 u_{t_1}+\D^2u_{t_2}}{2}\right)\\
	&\geq \frac{\log\det\D^2u_{t_1}+\log\det\D^2u_{t_2}}{2}=-\frac{t_1+t_2}{2}+\log w_\Omega^{-4}=\log\det\D^2u_{\frac{t_1+t_2}{2}}
	\end{align*}
in $U$.	Using Maximum Principle in the same way as above, we obtain
	$\frac{u_{t_1}+u_{t_2}}{2}\leq u_{\tfrac{t_1+t_2}{2}}$ in $U$, which means $t\mapsto u_t(x)$ is concave.
	
	The required limit $\lim_{t\rightarrow+\infty}u_t(x_0)=\env{\phi}(x_0)$ ($x_0\in U$) follows from (\ref{eqn_proofdepence}). The other limit $\lim_{t\to-\infty}u_t(x_0)=-\infty$ can be proved by taking a disk $D=D(x_0,\eps)$ with closure contained in $U$ and using Comparison Principle to control $u_t$ from above on the disk by the function
	$$
	v_t(x)=a\,e^{-\frac{t}{2}}\left(|x-x_0|^2-\eps^2\right)+b,
	$$
where the constants $a,b>0$ are chosen to ensure that 
$$
\det\D^2 v_t=4a^2e^{-t}\leq e^{-t}(-w_\Omega)^{-4}=\det\D^2u_t
$$ 
in $D$ and $v_t\geq \env{\phi}\geq u_t$ on $\pa D$.
	
	Finally, given $x_0\in U$, in order to show that $u_t(x_0)$ is \emph{strictly} increasing in $t$, we first establish the strict inequality $u_t(x_0)<\env{\phi}(x_0)$. Assume by contradiction that $u_t(x_0)=\env{\phi}(x_0)$ for some $t$ and consider the supporting affine function $a:\R^n\to\R$ of $u_{t}$ at $x_0$, namely
	$$
	a(x)=(x-x_0)\cdot \D u_{t}(x_0)+u_{t}(x_0).
	$$
The locus of $a=u_t$ is the single point $x_0$ because $u_t$ is strictly convex in $U$ (see Remark \ref{remark_smooth}). Since $u_t\leq\env{\phi}$ with equality at $x_0$, we have $a\leq \env{\phi}$ with equality only at $x_0$. This contradicts Lemma \ref{lemma_pleated} (the locus of $a=\env{\phi}$ is the convex hull of a subset of $\pa\Omega$), hence proves $u_t(x_0)<\env{\phi}(x_0)$.
Now, since $t\mapsto u_t(x_0)$ is concave, nondecreasing and tends to $\env{\phi}(x_0)$ as $t\to+\infty$, if $u_{t_1}(x_0)=u_{t_2}(x_0)$ for $t_1<t_2$ then $u_{t}(x_0)=\env{\phi}(x_0)$ for all $t\geq t_2$, contradicting what we just proved. This shows that $u_t(x)$ is strictly increasing in $t$ and completes the proof.	
\end{proof}

\subsection{Triangular cone as $C$-regular domain}\label{subsec_triangularcone}
Given a proper convex cone $C\subset\R^3$, we consider in this section a triangular cone $T\subset\R^3$ circumscribed to $C$ as shown in Figure \ref{figure_tc} \subref{figure_tc3} from the introduction, which is a $C$-regular domain. 

We show in Corollary \ref{coro_trianglecone} that if the projections of $C$ and $T$ on $\mathbb{RP}^2$ look like Figure \ref{figure_tc} \subref{figure_tc1}, then $T$ is uniquely foliated by affine $(C,k)$-surfaces generating $T$; whereas in case of  Figure \ref{figure_tc} \subref{figure_tc2}, $T$ is not generated by any affine $(C,k)$-surface.  A particular instance of the former case is when $C=C_0$ is the future light cone in the Minkowski space, which is computed explicitly in \cite[Section 5]{bon_smillie_seppi}.

The analytic counterpart is Proposition \ref{prop_introimproper} from the introduction, which deals with the Monge-Amp\`ere problem \eqref{eqn_main} studied in the previous part of this section in the particular case where $\phi\in\LC(\pa\Omega)$ vanishes at three points and take value $+\infty$ everywhere else. We give a more precise statement of the proposition as follows:

\begin{manualprop}{\ref*{prop_introimproper}'}\label{prop_imrpoper}
	Let $\Omega\subset \R^2$ be a bounded convex domain, $\Delta\subset\Omega$ be an open triangle with vertices on $\pa\Omega$ and $\phi$ be the function on $\pa\Omega$ vanishing at the vertices of $\Delta$ with $\phi=+\infty$ everywhere else.
\begin{enumerate}[label=(\alph*)]
	\item\label{item_improper1} If $\Omega$ satisfies the exterior circle condition at every vertex of $\Delta$ (see Figure \ref{figure_improper} \subref{figure_improper1}), then there exists a unique $u\in\S_0(\Omega)$ (see Section \ref{subsec_analytic}) satisfying \eqref{eqn_main}. Moreover, we have $\dom{u}=\overline{\Delta}$, and $u$ is continuous on $\overline{\Delta}$ with vanishing boundary value.
	\item\label{item_improper2} If $\pa\Omega$ contains an open line segment meeting $\pa\Delta$ exactly at a vertex (see Figure \ref{figure_improper} \subref{figure_improper2}), then there does not exist $u\in\S_0(\Omega)$ satisfying \eqref{eqn_main}. 
\end{enumerate}
\end{manualprop}
\begin{proof}
	Part (\ref{item_thmmain3}) of Theorem \ref{thm_main} yields a unique $u\in\LC(\R^2)$ satisfying \eqref{eqn_main} and the following conditions, which are equivalent to each other by Part (\ref{item_thmmain2}): 
	\begin{itemize}
		\item $u$ has infinite inner derivatives at every boundary point of $\dom{u}$ that is not on $\pa\Omega$;
		\item $\dom{u}=\overline{\Delta}$, and $u$ is continuous on $\overline{\Delta}$ with vanishing boundary value.
	\end{itemize}
Under the assumption of Statement \ref{item_improper1}, Part (\ref{item_thmmain2}) of Theorem \ref{thm_main} implies that $u$ has infinite inner derivatives everywhere on the boundary of $\Delta$. In view of the classical regularity results (see Remark \ref{remark_smooth}), we conclude that $u\in\S_0(\Omega)$, proving \ref{item_improper1}.
	
	To prove Part \ref{item_improper2}, it is sufficient to show that the unique $u$ given above has finite inner derivatives at the vertices of $\Delta$. To this end, we fix a Euclidean distance ``$\dist$'' on $\R^2$ and let $L\subset\R^2$ be the line containing $I$. Since both edges of $\Delta$ issuing from $x_0$ only touch $L$ at $x_0$, there is a constant $c>1$ such that
	$$
	\dist(x,x_0)\leq c\,\dist(x,L)\,\text{ for all }x\in \Delta.
	$$
	Therefore, it follows from Proposition \ref{prop_womega} (\ref{item_womega2}) that $(-w_\Omega)^{-4}$ is bounded from above on $\Delta$ by a positive constant multiple of the function 
	$x\mapsto\dist(x,x_0)^{-\frac{4}{3}}$. But $u$ satisfies $\det\D^2u=(-w_\Omega)^{-4}$ on $\Delta$ with $u=0$ on $\pa\Delta$, hence has finite inner derivatives at $x_0$ by Lemma \ref{lemma_angle} (\ref{item_lemmaangle1}).
\end{proof}

\begin{corollary}\label{coro_trianglecone}
Let $T\subset\R^3$ be a convex cone such that the projectivization $\mathbb{P}(T)\subset\mathbb{RP}^2$ is a triangle with edges denoted by $I_0$, $I_1$ and $I_2$. Let $C$ be a convex cone contained in $T$ such that  the boundary of its projectivization $\mathbb{P}(C)$ meets the interior of every $I_i$. Then $T$ is a $C$-regular domain and the following assertions hold:
\begin{enumerate}[label=(\alph*)]
	\item\label{item_trianglecone1} If for each $i$ there is an open ellipse $E_i\subset \mathbb{P}(C)$ such that  $\pa E_i\cap I_i\neq\emptyset$ (see Figure \ref{figure_tc} \subref{figure_tc1}), then $T$ is uniquely foliated by complete affine $(C,k)$-surfaces as in the conclusions of Theorems \ref{thm_intromain} and \ref{thm_introfoliation} in Introduction.
	\item\label{item_trianglecone2} If $\pa\mathbb{P}(C)$ meets the interior of some edge $I_i$ at a single point and $I_i$ is not tangent to $\pa\mathbb{P}(C)$ at that point (see Figure \ref{figure_tc} \subref{figure_tc2}), then there does not exist complete affine $(C,k)$-surface generating $T$.
\end{enumerate}
\end{corollary}
Note that under the assumption of \ref{item_trianglecone2}, the convex curve $\pa\mathbb{P}(C)$ is not $\C^1$ at the point where it meets $I_i$ and both of its tangent directions there point towards the interior of the triangle $\mathbb{P}(T)$, as the figure illustrates.

\begin{proof}
Let $H_i$ denote the open half-space of $\R^3$ containing $T$ such that the boundary plane $\pa H_i$ projects to the line in $\mathbb{RP}^2$ containing $I_i$. Then $T=H_1\cap H_2\cap H_3$. Under the assumption on $C$, each $H_i$ is a $C$-null half-space, hence $T$ is a $C$-regular domain.
	
Let  $\Omega\subset\R^2$ be the bounded convex domain determined from $C$ through the procedure in Section \ref{subsec_analytic}, which can be identified projectively with $\mathbb{P}(C^*)$. The projectivization $\Delta:=\mathbb{P}(T^*)$ of the triangular cone $T^*$ dual to $T$ is a triangle contained in $\mathbb{P}(C^*)\cong\Omega$ with vertices on $\pa\Omega$. Let $x_0$, $x_1$ and $x_2$ denote these vertices and $\phi:\pa\Omega\to\Rp$ be the function with $\phi(x_i)=0$ ($i=1,2,3$) and $\phi=+\infty$ on $\pa\Omega\setminus\{x_1,x_2,x_3\}$. Then under the bijection in Theorem \ref{thm_graph} (\ref{item_thmgraph1}), we have $T=\sepi{\env{\phi}^*}$.
The additional assumptions on $\mathbb{P}(C)$ and $\mathbb{P}(T)$ in statements \ref{item_trianglecone1}
and \ref{item_trianglecone2} are equivalent to the following conditions on $\Omega$ and $\Delta$, respectively:
\begin{itemize}
	\item 
	The assumption of \ref{item_trianglecone1} $\Leftrightarrow$ $\Omega$ satisfies the exterior circle condition at each $x_i$.
    \item 
    The assumption of \ref{item_trianglecone2} $\Leftrightarrow$ $x_i$ is in the interior of a line segment in $\pa\Omega$.
\end{itemize}

The proofs of Theorems \ref{thm_intromain} and \ref{thm_introfoliation} given in Section \ref{subsec_statement} actually shows that the conclusions of these theorems hold for the $C$-regular domain $D$ 
corresponding to a given $\phi\in\LC(\pa\Omega)$ whenever the function $u$ provided by Theorem \ref{thm_main} (\ref{item_thmmain3}) for this $\phi$ and arbitrary $c>0$ is in $\S_0(\Omega)$. Therefore, by Proposition \ref{prop_imrpoper} \ref{item_improper1}, they hold in particular under the assumption of \ref{item_trianglecone1}. 
On the other hand, under the assumption of \ref{item_trianglecone2}, Proposition \ref{prop_imrpoper} \ref{item_improper2} asserts that there is no $u\in\S_0(\Omega)$ satisfying (\ref{eqn_main}), hence no complete affine $(C,k)$-surface generating $T$.	
\end{proof}

\bibliographystyle{amsalpha} 
\bibliography{ksurface}
\end{document}